\documentclass[paper=a4, fontsize=11pt]{scrartcl}

\usepackage[sc]{mathpazo}
\usepackage[scaled]{helvet}
\usepackage[T1]{fontenc}

\usepackage[english]{babel}															
\usepackage[protrusion=true,expansion=true]{microtype}	
\usepackage{amsmath,amsfonts,amsthm} 
\usepackage[pdftex]{graphicx}	
\usepackage{caption,subcaption}  

\usepackage{colortbl}
\usepackage{booktabs}
\usepackage{tabularx} 
\usepackage{float} 

\usepackage[dvipsnames]{xcolor}
\usepackage{fullpage}
\usepackage{natbib}

\usepackage{tabularx} 
\usepackage{enumerate,enumitem}
\usepackage{bm,bbm}
\newcommand\ind{\mathbbm{1}}     

\newtheorem{theorem}{Theorem}

\newtheorem{corollary}[theorem]{Corollary}
\newtheorem{proposition}[theorem]{Proposition}
\newtheorem{lemma}{Lemma}

\newtheorem*{algorithm}{Algorithm}

\usepackage[colorlinks=true,citecolor=blue,urlcolor=blue]{hyperref}

\usepackage{algorithm}
\usepackage{algpseudocode}
\usepackage{adjustbox}
\usepackage{changebar}

\newcommand{\spann}{\operatorname{span}}
\newcommand{\sign}{\operatorname{sign}}

\newcommand{\expit}{\operatorname{expit}}
\newcommand{\logit}{\operatorname{logit}}

\newcommand{\diff}{\mathrm{d}}

\DeclareMathOperator*{\argmin}{arg\,min}
\DeclareMathOperator*{\argmax}{arg\,max}

\definecolor{offwhite}{RGB}{255,250,240}
\definecolor{gray}{RGB}{155,155,155}
\definecolor{foreground}{RGB}{80,80,80}
\definecolor{background}{RGB}{255,255,255}
\definecolor{title}{RGB}{89,132,212}
\definecolor{subtitle}{RGB}{255,199,0}
\definecolor{hilit}{RGB}{248,117,79}
\definecolor{vhilit}{RGB}{255,111,207}
\definecolor{lolit}{RGB}{200,200,200}
\definecolor{lit}{RGB}{255,199,0}
\definecolor{mdlit}{RGB}{71,106,170}
\definecolor{link}{RGB}{248,117,79}
\definecolor{subcol}{RGB}{207,216,233}

\definecolor{darkraspberry}{rgb}{0.53, 0.15, 0.34}

\usepackage{sectsty}
\allsectionsfont{\centering \normalfont\scshape}

\usepackage{fancyhdr}
\usepackage{chngcntr}
\counterwithout{table}{section}
\counterwithout{figure}{section}

\title{Nearest-neighbor LASSO logistic regression for the gradient
}

\author{
Touqeer Ahmad$^{a}$,
Fran\c{c}ois Portier$^{b}$ \&
Gilles Stupfler$^{c}$
}

\date{
$^{a}$ {\small Department of Mathematics, University of Oslo, P.O. Box 1053, Blindern, Oslo, Norway,} \\[1ex]
$^{b}$ {\small CREST, ENSAI, University of Rennes, France.} \\[1ex]
$^{c}$ {\small Univ Angers, CNRS, LAREMA, SFR MATHSTIC, F-49000 Angers, France.}
}
\begin{document}

\maketitle
\textbf{Abstract.} This paper investigates a new approach to estimate the gradient of the conditional probability given the covariates in the binary classification framework. The proposed approach consists of fitting a localized nearest-neighbor logistic model with $\ell_1$-penalty in order to cope with possibly high-dimensional covariates. Our theoretical analysis shows that the pointwise convergence rate of the gradient estimator is optimal under very mild assumptions. Moreover, using an outer product of such gradient estimates at several points in the covariate space, we provide a new method for estimating the central subspace, a well-known object allowing to carry out dimension reduction within the covariate space. Our implementation uses cross-validation on the misclassification rate to estimate the dimension of this subspace. We find that the proposed approach outperforms existing competitors in synthetic and real data applications.

\vskip1ex



\section{Introduction}
\label{sec:intro}

\cbstart
Estimating the gradient of a regression function is a central problem in nonparametric statistics and machine learning. Applications include plug-in bandwidth selection for kernel smoothers \citep{ruppert1995effective} and the construction of confidence intervals in nonparametric regression \citep{eubank1993confidence}. Beyond smoothing and inference, gradient estimation plays a key role in uncovering structural properties of high-dimensional data. Dimension-reduction methods based on average derivative estimation \citep{hardle1989investigating,hardle1992bandwidth} were later extended through the outer-product of gradients framework \citep{xia2002adaptive,dalalyan:2008}. More recently, gradient estimation has been used in sparse settings for zeroth-order optimization and guided-gradient regression trees \citep{wang2018stochastic,ausset2021nearest}, as well as for explaining individual predictions of complex classification models \citep{zhang2019should}.


A widely studied technique for gradient estimation is \textit{local linear estimation} \citep{fan1996}. The method consists of fitting a linear approximation of the regression function within a 
neighborhood of each point $x$; the intercept estimates the function itself at the point $x$, while the slope coefficients estimate its gradient.
Local linear estimators achieve minimax optimal convergence rates \citep{stone1980optimal,stone1982optimal,fan1996}, satisfy asymptotic normality \citep{fan1996}, 
enjoy strong uniform convergence guarantees \citep{masry1996multivariate}, 
and correct first-order boundary bias \citep{fan1996}. Fitting local linear functions in the broader context of likelihood models has been investigated in \cite{rob_tib,MR1325121}.

Despite these favorable properties, local linear estimation becomes difficult when the dimension $p$ is large, because local neighborhoods contain very few points due to the curse of dimensionality~\citep[Section 4.5]{wasserman2006all}.
Then even a local linear fit may suffer from instability or overfitting. A natural way to address this issue is to introduce penalization.
Interpretability of the estimated gradient coefficients constitutes another important challenge. As emphasized in \cite{rosasco2013nonparametric}, partial derivatives provide a natural way ``to measure the importance of each variable in the model''. 
Promoting sparsity in the gradient thus yields a form of local variable selection (identifying which covariates are influential in the neighborhood of a given point). This provides a strong motivation for incorporating $\ell_1$-penalization as proposed in LASSO regression \citep{tibshirani1996regression}. In this context, the $\ell_1$-penalization is expected not only to mitigate overfitting, but also to enhance interpretability by producing sparse gradient estimates. In addition, it has proven particularly effective in high-dimensional settings where the number of covariates $p$ may be large relative to the sample size \citep{tsybakov2009simultaneous}. 


In this paper, we hence investigate a local version of a logistic, LASSO-based estimation procedure of the gradient of the 
conditional probability in the classification setup. Let $(Y,X)\in \{0,1\} \times \mathbb R^p$ be a random vector with distribution $ P$. The conditional class probability is given by $ \pi (x)  = \mathbb P (Y= 1| X= x)$ and its logit transform is defined as $\ell (x) := \operatorname{logit}(\pi(x)) := \log (\pi(x) / (1-\pi(x)) $. 
Let $(Y_i,X_i) _{i=1,\ldots, n}$ be an independent collection of random variables with common distribution $P$. For each $x\in \mathbb R^p$, 
we consider the following nearest-neighbor, penalized, local logistic (and hence convex) problem 
\begin{align*}
    & (\widehat{a}_n(x), \widehat{b}_n(x) ) : = \argmax_{(a,b)\in \mathbb{R}\times \mathbb{R}^p} \Bigg\{ \sum_{i \in N_k(x)} Y_i \log( \expit(a + b^T (X_i-x)) ) \\
    &\qquad \qquad \qquad \qquad +\sum_{i \in N_k(x)} (1-Y_i) \log( 1-\expit(a + b^T (X_i-x)) ) - \lambda \|b\|_1 \Bigg\} 
\end{align*}
where $N_k(x)$ is the index set of the $k$-nearest neighbors to point $x$, $\expit(t):=\exp(t)/(1+\exp(t))$ and $\| \cdot \|_1$ is the $\ell_1-$norm on $\mathbb{R}^p$. The proposed approach might be seen as a 
nearest-neighbor, penalized adaptation of \cite{MR1325121} where standard kernel smoothing is used with $\lambda = 0$. Using the nearest-neighbor method 
ensures an adaptive bandwidth choice, in the sense that the bandwidth corresponding to a given number $k$ of nearest neighbors is random and depends on the position of the covariates around the target point $x$, while keeping computational time reasonable thanks to algorithms such as $k-d$ tree search. Using the $\ell_1$-penalty, as explained before, mitigates overfitting by inducing sparsity in the estimated coefficients, thereby enabling stable estimation in high-dimensional settings. \textcolor{black}{One closely related contribution is \cite{ausset2021nearest}, which studies a local linear LASSO-based estimator of the gradient in a nonparametric regression setting, rather than within the classification framework considered in the present work.}

The main result of the paper is the weak convergence of $\{ (\widehat{a}_n(x), \widehat{b}_n(x) ) - (\ell(x), \nabla \ell(x))\}  $, when suitably rescaled, under very mild conditions on $P$. 
It can be compared with the weak convergence results from \cite{fu2000asymptotics}, obtained in the linear regression framework; in contrast to~\cite{fu2000asymptotics}, we study logistic regression rather than linear regression, which leads to a non-quadratic likelihood and requires a different local asymptotic analysis, and we incorporate nearest-neighbor localization. The limiting distributions are similar as both put positive mass at the atom $0$ where the true parameter's coordinate is $0$, illustrating the method's ability to select variables. In contrast, our convergence rates are different from that of \cite{fu2000asymptotics} due to the nearest-neighbor localization and to our focus on nonparametric gradient estimation. The obtained rates for $\ell (x)$ and $\nabla \ell(x) $ achieve the minimax optimal rates of convergence as described in \cite{stone1982optimal}. In particular, our theoretical results compare favorably to the ones established in \cite{kanshi2022}, building on earlier work by~\cite{mukherjee2006estimation},  where the gradient is estimated using a \textit{Reproducing Kernel Hilbert Space} (RKHS) technique (see below our Corollary \ref{coro:main} for a precise comparison). 
\textcolor{black}{Our result can also be compared to that of \cite{MR1325121} (resp. \cite{ausset2021nearest}) where a similar rate of convergence in probability is obtained without using any penalty and with standard Nadaraya-Watson weights (resp. with LASSO penalization and nearest-neighbor localization but in classical regression).}

To show the usefulness of our gradient estimation method, we consider the following procedure: the gradient estimates at several points are aggregated using the standard outer product of gradients of \textit{e.g.}~Section 3.1 in \cite{xia2002adaptive}. The dimension reduction space, called central subspace, is then recovered by finding the eigenspaces of this matrix corresponding to its $d$ largest eigenvalues. Our use of the LASSO penalty promotes sparsity in the estimated vectors, but differs conceptually from the approach of \cite{li2007sparse}. In our framework, the dimension reduction matrix is constructed directly from sparse (nonparametric) gradient estimates. By contrast, \cite{li2007sparse} first estimates the dimension reduction matrix without imposing sparsity, using the inverse regression characterization of the dimension reduction space, and subsequently derives an eigenbasis that favors sparse components. Unlike earlier local linear techniques such as that of~\cite{lambert2006local} and~\cite{quali2023}, we may deal with 
large dimensions (of the order of several dozens when the sample size is $n=1000$). While most 
existing dimension reduction methods select the number $d$ of components in the reduction subspace by using rank testing procedures (see~\cite{bura2011dimension} and~\cite{portier2014bootstrap} for eigenvalue-based methods, and~\cite{luoli2016} for a technique using both the eigenvalues and eigenvectors), we 
rely on the underlying classification context by developing a simple cross-validation approach comparing the different sets of ordered eigenvectors given by the outer product of gradients. Surprisingly, this is helpful in improving the final accuracy of the classification method even when the true dimension of the reduction subspace is known.

The paper is organized as follows. In Section \ref{sec:main:back}, we describe our statistical framework 
and we construct our nearest-neighbor, penalized local logistic loss function, whose maximization gives rise to the proposed gradient estimator. In Section \ref{sec:main:estimator}, we establish the pointwise weak convergence of this estimator and Section~\ref{sec:extension} discusses an extension to multi-class generalized linear models. \textcolor{black}{In Section \ref{sec:numerical_exp}, we apply our gradient estimator in order to achieve dimension reduction in classification. More precisely, we employ an outer-product of gradients to estimate the central subspace as explained in Section~\ref{sec:aggregating}. 
The practical aspects of the final algorithm are discussed in Section \ref{algorith1}. We compare our approach to several competitors on synthetic and real data in Section~\ref{sec:numerical_exp:sim} and Section~\ref{sec:numerical_exp:data}. Section~\ref{sec:conclusion} concludes with a discussion of several research perspectives. 
Mathematical proofs and additional finite-sample results are postponed to the online Supplementary Material document.}

\cbend

\section{Statistical framework and main results}
\label{sec:main}
\subsection{Background}
\label{sec:main:back}

\cbstart

Let $Y\in \{0,1\}$ be a binary response variable with random covariate $X\in \mathbb R^p$. Let, for an $x$ {in the support $S_X$ of the distribution $P_X$} of $X$ (assumed to have nonempty interior),  
$
\pi(x) = \mathbb{P} (Y = 1 | X=x). 
$
The cross-entropy function $H:[0,1]\times [0,1]\mapsto [0,\infty]$, defined for any $q_1$ and $q_2$ in $[0,1]$ by 
\[
H(q_1,q_2)=-q_1 \log(q_2) - (1-q_1) \log(1-q_2), 
\]
with the convention $0\log 0=0$, is minimal if and only if $q_1=q_2$. As a consequence, the integrated cross-entropy
\[
\mathcal H (\pi, q) = \int_{z\in S_X} H(\pi(z),q(z)) P_X (\diff z), 
\]
viewed as a function of the map $z\in S_X\mapsto q(z)\in[0,1]$, is minimal if and only if $\pi = q$ on $S_X$. Note that, while $H$ acts on numbers in $[0,1]$, $\mathcal H$ acts on functions defined on $S_X$ and valued in $[0,1]$. Let $\varepsilon >0$, {denote by $B(x,\varepsilon)$ the closed ball with center $x$ and radius $\varepsilon$ for the Euclidean norm on $\mathbb{R}^p$,} and consider the localized version of the integrated cross-entropy, namely
\[
\mathcal H _{x,\varepsilon} (\pi, q) = \frac{ \int_{z\in B(x,\varepsilon) }  H(\pi(z),q(z)) P_X (\diff z)}{ \int_{z\in B(x,\varepsilon) }  P_X (\diff z)}.
\]
This version interpolates between $ H (\pi(x),q(x) ) $ and $ \mathcal H (\pi, q)$, as these two quantities are recovered when $\varepsilon \to 0$ and $\varepsilon \to \infty$, respectively.

The approach taken in this paper consists in minimizing 
an estimate of $\mathcal H _{x,\varepsilon} (\pi, q)$, when $\varepsilon $ is small and $\logit q$ is locally linear. The intuition is as follows. 
If one sets 
$
q_{\alpha, \beta} (z) = \expit (\alpha + \beta^T (z-x)),
$
and considers minimizing $\mathcal H _{x,\varepsilon} (\pi, q_{\alpha,\beta})$ with respect to $(\alpha, \beta)\in \mathbb R \times \mathbb R^p$, then the minimizers $(\alpha, \beta) $ will be close to $(\ell (x), \nabla \ell (x))$ as $\varepsilon$ is getting small. This 
is justified by a Taylor expansion of $\ell $ around $x$, \textit{i.e.}, $\pi (z) \simeq  \expit( \ell (z) + \nabla \ell (x)  ^T (z- x)) $ when $z \simeq x$.

Our goal, therefore, is to minimize an estimate of $\mathcal H _{x,\varepsilon} (\pi, q_{\alpha, \beta})$ with respect to $(\alpha, \beta) \in \mathbb R  \times \mathbb R^p$. To construct our estimator, one can recognize that 
\begin{align*}
\mathcal H _{x,\varepsilon} (\pi, q_{\alpha, \beta}) 
& =\frac{ \mathbb E [ (-Y\log(q_{\alpha, \beta}(X)) -(1-Y) \log(1-q_{\alpha, \beta}(X)) ) \ind _ {B(x,\varepsilon) }(X) ] }{\mathbb{P} (X\in B(x,\varepsilon) )}.
\end{align*}
The above expression allows to easily define an estimator of
$\mathcal H _{x,\varepsilon} (\pi, q_{\alpha, \beta})$ by replacing the expectation with a sample average. Our approach, described in the subsequent section, relies on nearest-neighbor localization, which makes it possible to adapt the value of $\varepsilon$ to regions having different density values for $X$.

\cbend

\subsection{Nearest-neighbor penalized local logistic estimator}
\label{sec:main:estimator}

\cbstart Let $(Y_i,X_i)_{1\leq i\leq n}$ be a collection of independent random variables with the same distribution as $(Y,X)$. \cbend Fix $x\in \mathbb{R}^p$ and let $N_k(x)\subset \{1,\ldots,n\}$ be the set gathering the indices of the $k$-nearest neighbors $X_i$ of the point $x$; we shall assume in the theory below that the distribution of $X$ has a density w.r.t.~Lebesgue measure, so that ties will not happen with probability 1 and $N_k(x)$ is well-defined. \cbstart The empirical nearest-neighbor counterpart of the local integrated cross-entropy 
{$\mathcal{H}_{x,\varepsilon}(\pi,q_{\alpha, \beta})$} is then \cbend
\[
-\frac{1}{k} \sum_{i \in N_k(x)} Y_i \log( \expit(a + b^T (X_i-x)) ) +  (1-Y_i) \log( 1-\expit(a + b^T (X_i-x)) ). 
\]
Introducing the LASSO penalty $\lambda \| b \|_1$, where $\| \cdot \|_1$ is the $\ell_1-$norm on $\mathbb{R}^p$, and rescaling, we naturally obtain a nearest-neighbor, penalized, local logistic estimator \cbstart of $\logit \pi(x)$ and its gradient as \cbend
\begin{equation}
\label{eqn:penNNlogit}
(\widehat{a}_n(x), \widehat{b}_n(x) ) = \argmax_{(a,b)\in \mathbb{R}\times \mathbb{R}^p} \{ L_n(a,b) - \lambda \| b \| _1 \}
\end{equation}
with 
\begin{align}
\nonumber
L_n(a,b) 
    &= \sum_{i \in N_k(x)} \big\{ Y_i \log( \expit(a + b^T (X_i-x)) ) \\& \qquad  +  (1- Y_i) \log( 1-\expit(a + b^T (X_i-x)) ) \big\} \nonumber\\ 
\label{eqn:penNNlogit_GLMform}
    &= \sum_{i \in N_k(x)} Y_i (a + b^T (X_i-x)) - \log( 1+\exp(a + b^T (X_i-x)) ).
\end{align}
Our main theoretical result is that one can obtain the asymptotic distribution of this pair of estimators under very weak assumptions on the distribution of $X$ and the conditional distribution of $Y|X=x$. We spell out these \cbstart assumptions \cbend and their interpretation below.
\begin{enumerate}[label=(A\arabic*), wide=0.5em, leftmargin=*]
    \item \label{cond:new1} The distribution of $X$ has a continuous density $f_X$ with respect to the Lebesgue measure on $\mathbb{R}^p$ and $f_X(x)>0$.  
    \item \label{cond:new2} 
    The function $\pi : \mathbb{R} ^p \to [0,1]$ is twice differentiable with continuous second order derivatives at $x$ and such that $\pi(x) \in (0,1)$.
\end{enumerate}
Assumption \ref{cond:new1} ensures that there are enough points around $x$ for the nearest-neighbor procedure to work. It also ensures the good probabilistic behavior of the bandwidth
\[
\widehat{\tau}_{n,k}(x) := \inf\left\{ \tau\geq 0:  \sum_{i=1}^n \ind_{ B(x,\tau) }(X_i) \geq k \right\}
\]
corresponding to the smallest radius $\tau \geq 0$ such that the ball $B(x,\tau)$ contains at least $k$ points from the sample. Actually, the fact that $X$ has a continuous distribution w.r.t.~Lebesgue measure yields 
\begin{align*}
    L_n(a,b) 
    &= \sum_{i=1}^n \big\{ Y_i \log( \expit(a + b^T (X_i-x)) ) \\ & + (1- Y_i) \log( 1-\expit(a + b^T (X_i-x)) ) \big\} \ind_{B(x,\widehat{\tau}_{n,k}(x))}(X_i).
\end{align*}
It also turns out that if $k=k_n\to\infty$ with $k/n\to 0$, then under~\ref{cond:new1}, $\widehat{\tau}_{n,k}(x)/\tau_{n,k}(x)\to 1$ in probability, where, if $V_p$ denotes the volume of the Euclidean unit ball in $\mathbb{R}^p$, 
\[
\tau_{n,k}(x) = \left(  \frac k n \frac 1 {f_X(x) V_p} \right) ^{1/p}.
\]
See Lemma~1 in~\cite{portier2021nearest}. It is then reasonable to write, for $n$ large enough, 
\begin{align*}
    L_n(a,b)\approx \overline{L}_n(a,b) 
    &= \sum_{i=1}^n \big\{ Y_i \log( \expit(a + b^T (X_i-x)) ) \\ &+  (1- Y_i)\log( 1-\expit(a + b^T (X_i-x)) ) \big\} \ind_{B(x,\tau_{n,k}(x))}(X_i). 
\end{align*}
The asymptotic behavior of $\overline{L}_n(a,b)$ is much easier to study than that of $L_n(a,b)$, since it is a sum of independent and identically distributed random variables; like $L_n(a,b)$, it defines a concave objective function and therefore one should expect that (up to technical details) the asymptotic behavior of $(\widehat{a}_n(x),\widehat{b}_n(x))$ will follow from the pointwise convergence of $\overline{L}_n(a,b)$. \cbstart The key result in order to make this intuition rigorous is a functional central limit theorem for nearest-neighbor estimators, which is of independent interest, and it is the first main result of this paper. Here and throughout $\| \cdot \|_2$ denotes the standard Euclidean norm.
\begin{theorem}[Central limit theorem for nearest-neighbor estimators] 
\label{theo:tightness}
Let $E$ be a nonempty and finite set. Assume that the data is made of independent copies  $(Y_i,X_i)_{1\leq i\leq n}$  of the random pair $(Y,X) \in E\times \mathbb R^p$ and that \ref{cond:new1} is fulfilled. Assume that $k:=k_n \to \infty$ is such that $k/n \to 0$. Let $\Psi_n : E \times \mathbb R^p \to \mathbb R^q$ be a sequence of measurable vector-valued functions and suppose that 
there is a positive integer $n_0$ such that 
\[
\sup_{n\geq n_0} \, \sup_{z\in A_{n,k}(x)} \, \max_{y\in E} \| \Psi_n(y,z)\|_2 <\infty,
\]
where $A_{n,k}(x) = B(x,(3/2)^{1/p}\tau_{n,k}(x))$. Define a centered and c\`adl\`ag stochastic process $(Z_n(\tau))_{\tau>0} $ by
\[
Z_n ( \tau ) = \frac{1}{\sqrt{k}} \sum_{i=1}^n \left\{\Psi_n (Y_i,X_i) \ind_{B(x,\tau )}(X_i) - \mathbb E [\Psi_n (Y,X) \ind_{B(x,\tau )} (X) ]  \right\}.
\]
\begin{enumerate}[label=(\roman*)]
\item Then $Z_n( \widehat{\tau}_{n,k}(x) ) = O_{\mathbb{P}}(1)$. 
\item Let $\Sigma_n^2(X) = \mathbb{E}[ \Psi_n (Y, X)\Psi_n (Y, X)^T | X  ] $. If moreover there is a (positive semidefinite) matrix-valued function $t\mapsto \Sigma^2(t,x)$ such that 
\[
\forall t\in [1/2,3/2], \ \int_{B(0,1)} \Sigma_n^2( x + \tau_{n,k}(x) t^{1/p} v ) \diff v \to V_p \Sigma^2(t,x), 
\]
then 
\[
Z_n ( \widehat{\tau}_{n,k}(x) ) = Z_n ( \tau_{n,k}(x) ) + o_{\mathbb{P}}(1) \stackrel{\mathrm{d}}{\longrightarrow} \mathcal{N}(0, \Sigma^2(1,x)).
\]
\end{enumerate}
\end{theorem}

Theorem~\ref{theo:tightness}(ii) is obtained by, first, showing that the empirical process $(Z_n(\tau))$ is tight and converges weakly to a continuous Gaussian process within a suitable functional space 
when the radius $\tau$ is restricted to being of similar order as the nearest-neighbor radius $\tau_{n,k}(x)$, and then by combining that with the fact that $\widehat{\tau}_{n,k}(x)/\tau_{n,k}(x)\to 1$ in probability. Let us highlight that, while the focus of Theorem~\ref{theo:tightness} is the sequence of random variables
\[
\sum_{i=1}^n \Psi_n(Y_i,X_i)\mathbf{1}_{B(x,\hat{\tau}_{n,k}(x))}(X_i),
\]
the results in \cite{portier2021nearest} focus on sequences of stochastic processes of the form
\[
\sum_{i=1}^n f(Y_i,X_i)\mathbf{1}_{B(x,\hat{\tau}_{n,k}(x))}(X_i),
\]
indexed by functions $f$ in a suitable family of functions $\mathcal{F}$. The fact that $\Psi_n$ varies with $n$ is the key reason why Theorem~\ref{theo:tightness} does not follow in a straightforward manner from the results of \cite{portier2021nearest}. See Section~\ref{sec:proofs:tightness} in the Supplementary Material document for full details of the proof. Besides, we discuss in Section~\ref{sec:extension} how Theorem~\ref{theo:tightness} applies to the obtention of asymptotic theory for a much larger class of estimators in Generalized Linear Models for categorical data. \cbend

Because of the LASSO penalty term in \eqref{eqn:penNNlogit}, the asymptotic distribution of $\widehat{b}_n(x) - \nabla\ell(x) $ will depend on the 
\cbstart local \cbend active set associated to $\nabla\ell(x)$, defined as the set of indices $j$ such that $\nabla\ell_j(x) \neq 0$. This is in line with the asymptotic distribution obtained for the standard least squares LASSO regression estimator, see \cite{fu2000asymptotics} and~\cite{zou2006adaptive}. Let, for any real number $t$, the quantity $\mathrm{sgn}(t) = \ind_{[0,\infty)}(t)-\ind_{(-\infty,0)}(t)$ be the sign of $t$, that is, $\mathrm{sgn}(t)=1$ when $t\geq 0$ and $-1$ otherwise. Finally, define 
\[
\Gamma(x) =  \pi(x)(1-\pi(x)) \begin{pmatrix} 1 & 0_p^T \\ 0_p & \frac{1}{p+2} I_p \end{pmatrix}
\]
where $0_p$ is the zero vector in $\mathbb{R}^p$ and $I_p$ denotes the identity matrix of order $p$. Our \cbstart next \cbend main result provides the limiting distribution of the pair $(\widehat{a}_n(x),\widehat{b}_n(x))$. In this result, Assumption~\ref{cond:new2} ensures that the gradient of $\ell = \logit \pi$ is well-defined, with the second order derivatives of $\pi$ and hence of $\ell$ coming into play when 
evaluating the bias term incurred when localizing. Denote by $\Delta \pi(x)$ the Laplacian of $\pi$ \cbstart at $x$. \cbend 
\begin{theorem}[Convergence of nearest-neighbor penalized local logistic regression estimators] 
\label{theo:main}
Suppose that \ref{cond:new1} and \ref{cond:new2} are fulfilled. If $k:=k_n  \to \infty $ and  $\lambda : = \lambda_n $ are such that 
$k^{1+p/2}/n\to\infty$, $k^{1+p/4}/n$ is bounded and $n \lambda^p / k^{1+p/2} \to c\in [0,\infty)$ then we have
\begin{align*}
    & \begin{pmatrix}
        \sqrt{k}(\widehat{a}_n(x)  - \ell (x)) \\ 
        \tau_{n,k}(x) \sqrt{k} (\widehat{b}_n(x) - \nabla \ell (x))
    \end{pmatrix}\\  
    &\stackrel{\mathrm{d}}{=} \argmax_{u=(u_0,u_1,\ldots,u_p)^T\in \mathbb{R}^{p+1}} \Bigg\{ u^T (W_n(x) + T_n(x)) - \frac 1 2 u^T\Gamma(x) u \\
    &- ( c f_X(x) V_p )^{1/p} \left(  \sum_{j=1}^p \mathrm {sgn}(\nabla \ell_j (x)) u_j \ind_{\{ \nabla \ell_j (x) \neq 0 \}} + |u_j| \ind_{\{ \nabla \ell_j (x) = 0 \}} \right) \Bigg\} + o_{\mathbb{P}}(1)
\end{align*}
with $W_n(x) \stackrel{\mathrm{d}}{\longrightarrow} \mathcal{N}(0,\Gamma(x)) $ and 
\begin{align*}
&T_n(x) = \\
&\tau_{n,k}^2(x) \sqrt{k} \left( \frac{1}{2(p+2)} \left( \Delta \pi(x) - \frac{1-2\pi(x)}{\pi(x)(1-\pi(x))} \| \nabla \pi(x) \|_2^2 \right) \begin{pmatrix} 1 \\ 0_p \end{pmatrix} + o_{\mathbb{P}}(1) \right).
\end{align*}
\end{theorem}
The conditions on $k$ and $\lambda$ in Theorem~\ref{theo:main} are equivalent to assuming that $\tau_{n,k}(x) \sqrt{k}\to\infty$, $\tau_{n,k}^2(x) \sqrt{k}$ is bounded and $\lambda / (\tau_{n,k}(x) \sqrt k)$ converges to the finite constant $( c f_X(x) V_p )^{1/p}$. Condition $\tau_{n,k}(x) \sqrt{k}\to\infty$ is necessary in order to be able to write a Taylor expansion of the penalized component in the loss function (this is not needed for the analysis of the non-penalized local logistic maximum likelihood function, see
Lemma~\ref{lem:Rn}).
In pointwise results on local linear kernel quasi-maximum likelihood estimation with one-dimensional covariates, Theorem~1a in~\cite{MR1325121} requires $\sqrt{nh^3} = h\sqrt{nh}\to\infty$, where $h$ is the kernel bandwidth; note that $h$ and $\tau_{n,k}(x)$ play the same role and that, for $p=1$, the kernel regression analogue of $k$ is a quantity proportional to $nh$, so that conditions $h\sqrt{nh}\to\infty$ and $\tau_{n,k}(x) \sqrt{k}\to\infty$ are indeed analogous, and constitute the assumptions required in order to ensure consistency of the local linear estimators of $\ell$ and its gradient. Of course, condition $\tau_{n,k}(x) \sqrt{k}\to\infty$ is automatically satisfied if $\tau_{n,k}^2(x) \sqrt{k}$ converges to a finite positive limit, namely, when the nonparametric bias-variance tradeoff is achieved and the optimal rate of convergence of the local linear estimator is found. 
\cbstart Condition $k/n\to 0$, which is standard in nearest-neighbor estimation, follows from assuming that $k\to \infty$ and $\tau_{n,k}^2(x) \sqrt{k}$ is bounded. Also note that Assumption (A1) and the condition that $\pi$ is twice continuously differentiable, appearing in Assumption (A2), are contained in the statement of Theorem~3 in~\cite{MR1325121}. Condition $\pi(x)\in (0,1)$, part of Assumption (A2), is contained in the requirement that the conditional variance is nonzero in Condition (3) in~\cite{MR1325121}. \cbend

It follows that, under Assumptions \ref{cond:new1} and \ref{cond:new2} plus classical conditions linking the number of nearest neighbors and the penalizing constant $\lambda$, the two estimators $\widehat{a}_n(x)$ and $\widehat{b}_n(x)$ converge respectively at the rate $1/\sqrt{k}$ and $1/(\tau_{n,k}(x) \sqrt{k})$. \cbstart Since $\tau_{n,k}(x)$ is proportional to $(k/n)^{1/p}$, the rate for estimation of the probability is $1/\sqrt k + (k/n)^{2/p}$, while the rate for the estimation of the gradient is $(k/n)^{-1/p}(1/\sqrt k + (k/n)^{2/p})$. Interestingly, the estimation of the probability behaves similarly as that of classical $k$-NN (see Section 14.3 in \cite{biau2015lectures}), meaning that local linear estimation does not improve the rate of convergence when estimating a twice differentiable function. This is because the standard $k$-NN method is already optimal for twice differentiable functions, due to the even-symmetry of the localizing kernel, that is, $\ind_{B(x,\hat \tau _{n,k}(x))}(x+u) = \ind_{B(x,\hat \tau _{n,k}(x))}(x-u)$. A precise study of the constants involved in the limiting results, and of the behavior of local linear estimation at the boundary, could highlight further differences, but this is beyond the scope of our paper. \cbend 

Note that for non-penalized local logistic regression ($\lambda = 0$ and thus $c=0$), the result of Theorem~\ref{theo:main} is just 
\[
\begin{pmatrix} \sqrt{k}(\widehat{a}_n(x)  - \ell (x)) \\ 
    \tau_{n,k}(x) \sqrt{k} (\widehat{b}_n(x) - \nabla \ell (x))
\end{pmatrix}
\stackrel{\mathrm{d}}{=} \mathcal{N}( 0, \Gamma^{-1}(x) )   +  O_{\mathbb{P}} ( \tau_{n,k}^2(x) \sqrt{k} ) + o_{\mathbb{P}}(1).
\]
In particular, if $k \tau_{n,k}^4(x) \to 0$, a straightforward application of the delta-method yields 
\[
\sqrt{k} (\expit(\widehat{a}_n(x)) - \pi(x))\stackrel{\mathrm{d}}{\longrightarrow} \mathcal{N}\left( 0, \pi(x)(1-\pi(x)) \right)
\]
as expected from standard maximum likelihood theory when the logistic regression model is valid. Condition $k \tau_{n,k}^4(x) \to 0$, which makes the bias term $T_n(x)$ vanish asymptotically, again has a straightforward analogue in local linear kernel quasi-maximum likelihood estimation; for $p=1$, the corresponding condition is $n h^5\to 0$, which is exactly the bias condition necessary to eliminate the smoothing bias term in Theorem~1a in~\cite{MR1325121}. 

\cbstart In the penalized case, there is an interesting parallel with fixed-dimension LASSO regression, considered in~\citet[Theorem~2]{fu2000asymptotics}. Observe that, since the object of interest here is the gradient $\nabla \ell (x)$, the speed of convergence $\tau_{n,k}(x)\sqrt{k}$ should be seen as the analogue of the speed of convergence $\sqrt{n}$ in LASSO regression. Then, setting $D(x)=( c f_X(x) V_p )^{1/p}$, condition  $n \lambda^p / k^{1+p/2} \to c$ is equivalent to $\lambda / (\tau_{n,k}(x) \sqrt k)\to D(x)$, which plays the same role as condition $\lambda_n/\sqrt{n}\to \lambda_0$ in Theorem~2 of~\cite{fu2000asymptotics}. If furthermore $k \tau_{n,k}^4(x) \to 0$, then the weak limit of Theorem~\ref{theo:main} is exactly 
\begin{align*}
    &\argmin_{u=(u_0,u_1,\ldots,u_p)^T\in \mathbb{R}^{p+1}} \Bigg\{ -2 u^T W(x) + u^T\Gamma(x) u \\
    &+ 2 D(x) \left(  \sum_{j=1}^p u_j \mathrm {sgn}(\nabla \ell_j (x)) \ind_{\{ \nabla \ell_j (x) \neq 0 \}} + |u_j| \ind_{\{ \nabla \ell_j (x) = 0 \}} \right) \Bigg\} 
\end{align*}
where $W(x)$ has a $\mathcal{N}(0,\Gamma(x))$ distribution. This is an obvious analogue of the weak limit in Theorem~2 of~\cite{fu2000asymptotics} (note the constant 2 in front of the constant $D(x)$, due to~\cite{fu2000asymptotics} using a different normalization from ours). In particular, if some of the $\nabla \ell_j (x)$ are 0, the limiting distribution will put positive probability at the origin, as the calculation on pp.1361-1362 of~\cite{fu2000asymptotics} shows. This indicates that the estimator we consider is able to perform variable selection. \cbend

Another immediate corollary of Theorem~\ref{theo:main} can also be given on the estimation rate of the gradient $\nabla \ell (x)$ by $\widehat{b}_n(x)$. 
\begin{corollary}[Rate of convergence of the gradient estimator] 
\label{coro:main}
Under the conditions of Theorem~\ref{theo:main}, $\widehat{b}_n(x)$ is a consistent estimator of $\nabla \ell (x)$, and 
\[
\widehat{b}_n(x) - \nabla \ell (x)  = O_{\mathbb{P}} \left( \frac{1}{\tau_{n,k}(x) \sqrt{k}} \right)  +  O_{\mathbb{P}} (  \tau_{n,k}(x) ).
\]
\end{corollary}
%
The best achievable rate of convergence, obtained for $ k = n ^{ 4 / (p + 4) }$, is $  n ^{ - 1 / (p + 4) }  $ for the estimation of the gradient vector. This rate is optimal, see~\cite{stone1982optimal}, and is better than the rate $n^{-1/(6(p+4))}$ obtained in \cite{kanshi2022}. A similar (non-asymptotic) bound is obtained in \cite{ausset2021nearest} in the simpler case of a local least squares estimator of the gradient which is not appropriate in the classification problem. 

The above results only require 
regularity conditions at the point $x$. As a consequence, the best that can be hoped for is 
a pointwise asymptotic convergence result in the spirit of Theorem~\ref{theo:main}. Stronger results, such as uniform convergence results for $\widehat{b}_n(x)$ over compact subsets of the support of $X$, could be obtained under much more restrictive conditions. Let us \cbstart finally \cbend point out that from the estimator $(\widehat a_n(x), \widehat b_n(x) )$, one can easily construct, using the standard plug-in rule, an estimator of $(\pi(x), \nabla \pi(x) )$. This estimator will of course inherit the good statistical properties of $(\widehat a_n(x), \widehat b_n(x) )$ established before.

\subsection{Extension to multi-class generalized linear models}
\label{sec:extension}

Our general asymptotic theory also applies to much more general settings where $Y$ represents a class label chosen among a finite number of classes, and in generalized linear model (GLM)-type specifications using link functions that may not be equal to the simple logit link. We outline here how this may be done under our framework but omit the technical details for the sake of brevity.

Suppose for the moment that $Y\in E$ is a response variable living in the finite set $E=\{0,1,\ldots,N\}$, with $N\geq 1$, and that the conditional distribution of $Y$ given $X=x$ is part of the one-parameter exponential family, say with conditional probability mass function $y\mapsto h(y) \exp(y \theta(x) - C(\theta(x)))$, where $\theta$ is a univariate parameter mapping, $C$ is a known smooth, strictly monotonic and convex function, \cbstart and $h$ is a nonnegative normalizing function defined on $E$. Then it is known that the conditional mean $m(x)=\mathbb{E}(Y | X=x)$ satisfies $m(x)=C'(\theta(x))$; equivalently, $\theta(x)=C'^{-1}\circ m(x)$. Arguing in the same way as in Section~\ref{sec:main:estimator} and viewing the quantity 
$\theta(x)$ and its gradient as the targets of a local linear estimation procedure, 
one may define a nearest-neighbor, penalized local linear 
estimator as
\begin{equation}
\label{eqn:penNN_GLM}
(\widehat{a}_n(x), \widehat{b}_n(x) ) = \argmax_{(a,b)\in \mathbb{R}\times \mathbb{R}^p} \{ \mathcal{L}_n(a,b) - \lambda \| b \| _1 \}
\end{equation}
with 
\begin{equation}
\label{eqn:penNN_GLMLn}
\mathcal{L}_n(a,b) = \sum_{i \in N_k(x)} Y_i (a + b^T (X_i-x)) - C( a + b^T (X_i-x) ).
\end{equation}
The penalized local logistic estimator spelled out in~\eqref{eqn:penNNlogit} is a special case of the more general construction~\eqref{eqn:penNN_GLM}, with $N=1$ and $C(t)=\log(1+e^t)$, as can be seen by comparing~\eqref{eqn:penNNlogit_GLMform} with~\eqref{eqn:penNN_GLMLn}. \cbend This log-likelihood is again a concave objective function, so under further regularity conditions on $C$ and after calculations analogous to those of 
Sections~\ref{sec:proofs:loglik} and~\ref{sec:proofs:theo}, 
our general theory in 
Theorem~\ref{theo:tightness} 
and 
\cbstart Proposition~\ref{prop:mean_var_score_nn} 
will apply and yield, after straightforward calculations, that  $(\sqrt{k}(\widehat{a}_n(x)  - \theta (x)), \tau_{n,k}(x) \sqrt{k} (\widehat{b}_n(x) - \nabla \theta (x)))$ has a nontrivial weak limit. 
This estimation technique then provides a way of recovering  the 
gradient of the parameter function $\theta$ in a multi-class setting under a GLM specification. \cbend

The argument of course extends in a straightforward fashion to the situation where $Y$ and the corresponding parameter mapping are multivariate: for example, a multinomial model with $K\geq 2$ classes has conditional probability mass function 
\[
(y_1,\ldots,y_{K-1}) \mapsto \pi_1(x)^{y_1} \cdots \pi_{K-1}(x)^{y_{K-1}} \left( 1-\sum_{j=1}^{K-1} \pi_j(x) \right)^{1-\sum_{j=1}^{K-1} y_j}
\]
where $y_1,\ldots,y_{K-1}\in \{0,1\}$ and $\sum_{j=1}^{K-1} y_j\in \{0,1\}$. This is readily put in exponential form, resulting in a probability mass function of the type $y\mapsto h(y) \exp(y^T \theta(x) - C(\theta(x)))$ with $y=(y_1,\ldots,y_{K-1})$ and $\theta(x)=(\pi_1(x), \ldots \pi_{K-1}(x))$ being $(K-1)-$dimensional. In such a situation one would use the estimator 
\[
(\widehat{a}_n(x), \widehat{b}_n(x) ) = \argmax_{(a,b)\in \mathbb{R}^{K-1}\times \mathbb{R}^{(K-1)\times p}} \{ \mathcal{L}_n(a,b) - \lambda \| b \| _1 \}
\]
with, for any $(a,b)\in \mathbb{R}^{K-1}\times \mathbb{R}^{(K-1)\times p}$, 
\[
\mathcal{L}_n(a,b) = \sum_{i \in N_k(x)} Y_i^T (a + b^T (X_i-x)) - C( a + b^T (X_i-x) ).
\]
\cbstart This makes it possible to construct nearest-neighbor 
LASSO methodologies for the gradients of (suitably transformed) success probabilities \cbend of multinomial responses in either the nominal or ordinal models; see Chapter~8 in~\cite{agr2013} for a concise introduction to such models in the classical GLM setting.

\section{Application to dimension reduction}\label{sec:numerical_exp}

After reviewing the outer product of gradient framework for dimension reduction, we introduce a new algorithm, based on our estimator of the gradient, to achieve dimension reduction in classification. 
We then describe the different competitors and evaluation metrics that we shall use. We next analyze synthetic data examples and finally consider real data examples in order to showcase the benefits of our methodology in classification tasks. The code that may be used to reproduce our experiments is publicly available on GitHub\footnote{\href{https://github.com/touqeerahmadunipd/LLO_regression}{https://github.com/touqeerahmadunipd/LLO\_regression}}. 

\subsection{Dimension reduction and outer-product of gradient}
\label{sec:aggregating}

\cbstart Suppose that
\[
\pi(x) = \mathbb{P}(Y = 1 \mid X = x) = g(\beta^{\top} x),
\]
where $\beta \in \mathbb{R}^{p \times d}$ with $d \leq p$ is an unknown parameter matrix and $g : \mathbb{R}^{d} \to [0,1]$ is an unknown measurable function. The objective is to estimate the \textit{central subspace}, defined as $\spann(\beta)$, thereby reducing the dimensionality of the problem, and then to estimate $\pi$. This two-step strategy alleviates the difficulty of directly estimating $\pi$ when the ambient dimension $p$ is large.

The gradient $\nabla \pi(x)$ of $\pi$ at $x$ satisfies 
$
\nabla \pi(x) = \beta \nabla g(\beta^T x).
$
As such, $\nabla \pi(x)\in \spann (\beta)$; likewise,
\[
\nabla \ell(x) = \nabla \logit \pi(x) = \frac{1}{\pi(x) (1-\pi(x))} \nabla \pi(x) \in \spann(\beta), \mbox{ where } \ell=\logit \pi.
\]
To recover $\spann(\beta)$, it then suffices to estimate enough gradients of the form $\nabla \ell(x_j)$. This motivates introducing the matrix 
\[
M = \int_{\mathbb R^p}  \nabla \ell(x)  \nabla \ell(x) ^T \mu (\diff x) = \mathbb{E}_{X^* \sim \mu} [ \nabla \ell(X^*) \nabla \ell(X^*)^T]
\]
where $\mu$ is a probability measure supported on $S_X$. This approach is often referred to as the \textit{(expected) outer product of gradients} method \citep{samarov1993,MR1891738,xia2002adaptive,yuaxukpohsu2025}. Our approach departs from the existing literature in two key respects: the use of nearest-neighbor selection to define local regions, and the introduction of a LASSO penalty to enforce sparsity in the estimated gradient.


 Recall that for any $x$, the solution $\widehat{b}_n (x)$ of the optimization procedure \eqref{eqn:penNNlogit} is an estimator of $\nabla\ell(x)$. 
To estimate the matrix $M$, generate $X_i^* \sim \mu$, $i=1,\ldots,m$ independently and compute \cbend
\begin{equation}
\label{eqn:M_est}
\widehat{M} = \frac{1}{m} \sum_{i=1}^m  \widehat{b}_n (X_i^*)  \widehat{b}_n (X_i^*)^T.
\end{equation}   
One can then define $(\widehat \beta_1,\ldots, \widehat \beta_p)$ as the set of orthogonal eigenvectors of $\widehat{M}$, ordered according to their eigenvalues (in decreasing order). Finally, given $d\in\{1,\ldots, p\}$ (which is chosen in practice by cross-validation, as we shall explain  in Section~\ref{algorith1}), the projection matrix 
\[
\widehat{P}_{\widehat{\beta}} = \widehat{P}_{\widehat{\beta},d} = \sum_{k=1}^d \widehat \beta_k \widehat \beta_k^T
\]
defines an estimator of the projection on the central subspace of interest. A corollary from our main results can now be stated on the estimation of $M$. 
\begin{corollary}\label{corollary_matrix}
Suppose that $\mu$ is a finitely supported measure on $S_X$ and that for each $x\in \operatorname{supp} (\mu)$, Assumptions \ref{cond:new1} and \ref{cond:new2} are satisfied. If $k:=k_n  \to \infty $ and $\lambda : = \lambda_n $ are such that $k^{1+p/2}/n\to\infty$, $k^{1+p/4}/n$ is bounded and $n \lambda^p / k^{1+p/2}$ converges to a finite constant, 
then
\[
\widehat{M} - M = O_{\mathbb{P}} \left( \frac{1}{\tau_{n,k}(x) \sqrt{k}} \right)  +  O_{\mathbb{P}} (  \tau_{n,k}(x) ) + O_{\mathbb{P}} \left( \frac{1}{\sqrt m } \right) .
\]
\end{corollary}
%
\cbstart
The eigenprojector $\widehat{P}_{\widehat{\beta}}$ has the same rate of convergence, 
see Lemma 4.1 in \cite{tyler_81}. This allows consistent estimation of $\spann(\beta)$. 
\cbend

\cbstart
The rate of convergence in Corollary~\ref{corollary_matrix} is nonparametric. This might be improved by choosing the measure $\mu$ differently: observe that a natural choice for $\mu$, which is also the one made in our numerical experiments, is the empirical measure of $X_1,\ldots, X_n$. This is, of course, a random measure whose number of atoms is not bounded with respect to $n$, meaning that Corollary~\ref{corollary_matrix} does not apply to this choice of $\mu$.  From a theoretical perspective, we conjecture, following results given in \cite{MR1865333,dalalyan:2008,yuaxukpohsu2025}, that such a choice would be valid and lead to a different rate of convergence compared to the one given in the above corollary, although under substantially stronger assumptions than ours, such as extra conditions on the regularity of the design, and possibly moment conditions on $X$. This question is left for further research. \cbend 


\subsection{The algorithm}\label{algorith1}

The proposed dimension reduction algorithm is now described. Two versions will be considered when solving \eqref{eqn:penNNlogit}: $\lambda = 0$ and $\lambda >0$. We denote them by LLO($\lambda =0$) and LLO($\lambda > 0$), respectively, where LLO is shorthand for local logistic.

\subsubsection{Estimation of the dimension reduction matrix}\label{DRM}

\cbstart {\textbf{Description of the algorithm:}} For a given choice of $m$  (the number of vectors $b$ used to estimate $M$ {following \eqref{eqn:M_est}}), a given value of $\lambda$ (the penalization parameter in the optimization problem) \cbend and $k$ the number of neighbors, the computation of the dimension reduction matrix is straightforward; see Algorithm \ref{alg:Mestimate} below (the optimization 
uses the \texttt{R} function \texttt{glmnet} from the package of the same name).

\begin{algorithm}
\begin{algorithmic}[1]
\State {\textbf{Input:} $(X_1,Y_1),\ldots,(X_n,Y_n)\in \mathbb{R}^{p}\times \{ 0,1 \}$}, $\lambda>0$, $k\in \{1,\ldots, n\}$ and $m\in \{1,\ldots, n\}$
\State {\textbf{Output:} Dimension reduction matrix $\widehat{M}$}
\State {Draw uniformly a list $\mathcal X_m$ of $m$ 
different observations among $\mathcal X = (X_1,\ldots,X_n)$}
\For{each $x$ $\in \mathcal X_m$}
      \State{Compute $N_k(x)$, the index of the $k$ nearest neighbors to $x$ among $\mathcal X$}
      \State{Compute $\widehat{a}_n(x)$ and $\widehat{b}_n(x)$ according to~\eqref{eqn:penNNlogit}}  
\EndFor
\State{Return $\widehat{M} = \frac{1}{m} \sum_{x\in \mathcal X_m} \widehat{b}_n(x)\widehat{b}_n(x)^T$}
\end{algorithmic}
\caption{Estimation of $M$}
\label{alg:Mestimate}
\end{algorithm}

\cbstart To mitigate biases that may arise from imbalanced class distributions, we exclude samples where, after finding the closest neighbors, either class 0 or class 1 is rare, defined here as having fewer than $5$ points within one of the two classes. Ideally, one should set $m=n$ so that the gradient is estimated at each data point; here we set $m=n/4$ to save computation time, as, in our experience, this choice does not substantially adversely affect finite-sample results. \cbend
\vskip1ex
\noindent
\cbstart {\textbf{Hyperparameter selection:}} We now discuss the choice of the hyperparameters $\lambda $ and then $k$. To estimate the directions featured in \eqref{eqn:M_est}, we use either the pure nearest-neighbor logistic regression without penalization ($\lambda = 0) $ or its penalized version ($\lambda >0$). For the latter, we employ the following cross-validation (CV) to set $\lambda$. Our approach consists in selecting the same parameter $\lambda$ for all $x\in \mathcal X_m$ in order to decrease runtime. We first localize the data around the central point  $x=\overline{x}=\frac{1}{n}\sum_{i=1}^n X_i$ with the default choice of $\lfloor \sqrt{n} \rfloor$-nearest neighbors to $x$. Then we select $\lambda$ by regular $10$-fold cross-validation applied to the localized data. In other words, we divide the (localized) data into $10$ randomly selected subsets of equal size, each subset serving as a validation set while the remaining subsets are used for model fitting. The fitted model is assessed on the validation set using the misclassification error, that is, the proportion of observations whose label is not correctly predicted, which is 
the empirical counterpart of the misclassification risk $\mathcal{R}(g)= \mathbb{P}(g(X)\neq Y)$ for a given classifier $g$; in this cross-validation procedure, an observation is labeled as $1$ if and only if its predicted probability of success by the nearest-neighbor logistic estimator $\widehat{\pi}_n(\overline{x})=\expit(\widehat{a}_n(\overline{x}))$ exceeds $1/2$. This evaluation of the quality of the fitted model is then done across a sequence of $\lambda$ values: more precisely, we use the \texttt{R} function \texttt{cv.glmnet} from the \texttt{glmnet} package with \texttt{type.measure=class}, which 
computes the average misclassification error across all validation sets, and the regularization parameter $\lambda$ selected is the one minimizing this error. 

A possible way to choose $k$ is also through a CV procedure. Specifically, we consider $k\in\{1,5,…,500\}$ 
and then use Algorithm \ref{alg:Mestimate} to obtain a reduced representation of the data that is ultimately used to fit a classifier (here a $10$-NN classifier was employed); again, the quality of each classifier is evaluated usign the misclassification error. For the penalized method LLO$(\lambda > 0$), where $\lambda$ is set with the above CV procedure, 
the selected values of $k$ are 
closer to the default $(k = \lfloor \sqrt{n} \rfloor)$ (see Figure~\ref{fig_lam}), and the misclassification error remains stable over a range of $k$, indicating reasonably low sensitivity to this parameter.
Given this empirical stability and to reduce computational burden, we retain the default choice $k = \lfloor \sqrt{n} \rfloor$ for LLO$(\lambda > 0)$ 
throughout the numerical experiments below. A slight improvement may be expected by fine-tuning. Concerning the case $\lambda = 0$, the sensitivity to $k$ is slightly more pronounced compared to $\lambda >0$, so we prefer to fix $k$ using the above described CV scheme in this case. It is worth noting that we later use the local-likelihood-based method~\citep{lambert2006local} and Principal Weighted Support Vector Machines (WSVM)~\citep{shin2017principal} as competitors; the bandwidth parameter $h$ for the former will also be chosen via CV.
\cbend 

\subsubsection{Selecting the dimension of the reduction subspace}


\cbstart
Following the seminal paper by \cite{li91}, the leading approach to dimension estimation is based on a sequential rank testing procedure. Each test considers the following null hypothesis $H_0:$ ``the rank of the matrix is $r$'', and is implemented sequentially with increasing $r$; the first value for which $H_0$ is not rejected gives the rank estimate. This estimate also determines the number $d$ of components retained for dimension reduction. We refer to \cite{bura2011dimension} and \cite{portier2014bootstrap} for details on the rank testing procedures used to estimate the dimension of the reduction subspace.

The approach taken in our work is not rank-based, but is instead driven by cross-validation on the underlying classification task. Our rationale is that determining the dimension solely from the rank of the dimension reduction matrix can be overly restrictive, as additional information from the prediction task may also be useful. We adopt this broader perspective: task-specific information is exploited alongside rank information to estimate the dimension more effectively. Concretely, rank information enters our procedure through the construction of candidate subspaces, which are sums of eigenspaces associated with decreasing eigenvalues. The candidate subspaces are then compared using cross-validation.

\cbend

Our procedure first divides the data into a training set $(X_{\mathrm{train}},Y_{\mathrm{train}})$ and a testing set $(X_{\mathrm{test}},Y_{\mathrm{test}})$. The matrix $\widehat{M}$ and its $p$ orthogonal eigenvectors $\widehat \beta_1,\ldots, \widehat \beta_p$, in decreasing order according to their eigenvalues, are estimated from the training set following the procedure described in Algorithm \ref{alg:Mestimate}. For every $d\in \{1,\ldots,p\}$, the first $d$ eigenvectors are gathered in a matrix $\widehat \beta _{(1:d)}\in \mathbb R^{p\times d}$, the sets of covariates $X_{\mathrm{train}}$ and $X_{\mathrm{test}}$ are projected onto the sets of lower-dimensional covariates $X_{\mathrm{train}} \widehat{\beta}_{(1:d)}$ and $X_{\mathrm{test}} \widehat{\beta}_{(1:d)}$, and a classifier is learned based on the training set $(X_{\mathrm{train}} \widehat{\beta}_{(1:d)}, Y_{\mathrm{train}})$. We use here  
the \texttt{knn} nearest-neighbor classifier, that is, at a given point $x$, the result of the majority vote among the nearest neighbors of $x$ within the space of projections of the covariates in the training set (with ties broken at random). For this classifier, the training step merely consists of storing the covariates in the training set along with their labels, which will form the basis for the vote at each point. In the real data analysis, we shall also compare our results with the Random Forest classifier, whose training step is nontrivial. The performance of the chosen classifier is then evaluated on the test set $(X_{\mathrm{test}} \widehat{\beta}_{(1:d)}, Y_{\mathrm{test}})$, and the dimension retained is the one for which the classifier has the lowest misclassification risk.

This forward iterative procedure, summarized in Algorithm~\ref{alg:feature}, 
\cbstart strikes \cbend a balance between dimension reduction and the preservation of relevant information for classification. 
We shall compare the results obtained with the situation where the correct dimension of the reduction subspace is known in order to assess the influence of the dimension selection step.

\begin{algorithm}[t]
\begin{algorithmic}[2]
\State{\textbf{Input:} Dataset $(X,Y)$ with $X\in \mathbb{R}^{n\times p}$ and $Y\in  \{ 0,1 \}^n $, classification algorithm $g$ (kNN, random forest...), and parameters $\lambda>0$, $k\in \{1,\ldots, n\}$, $m\in \{1,\ldots, n\}$ and {$K\geq 2$}}
\State{\textbf{Output:} {Dimension of reduction subspace}}
\State{Estimate $\widehat{M}$ using Algorithm \ref{alg:Mestimate} and compute the eigenvectors $\widehat \beta_1,\ldots, \widehat \beta_p$ of $\widehat{M}$}
\State{Split $(X,Y)$ into $K$ folds $(X_{(j)},Y_{(j)})_{j=1,\ldots,K}$}
\For{each $d \in \{ 1,\ldots, p\} $}
    \State{{Define $\widehat{\beta}_{(1:d)}=[\widehat \beta_1 \cdots \widehat \beta_d]$}}
\For{each $j\in \{ 1,\ldots, K\} $}
    \State{Define $(X_{\mathrm{train}},Y_{\mathrm{train}}) =  (X,Y)\backslash (X_{(j)},Y_{(j)})$}
    \State{Train the classification rule $g$ on data $( X_{\mathrm{train}} \widehat{\beta}_{(1:d)} , Y_{\mathrm{train}})$}
    \State{{Evaluate its misclassification risk $R_{j,d}$ using $(X_{(j)} \widehat \beta _{(1:d)}, Y_{(j)})$}}
\EndFor

\State{{Compute $R_{d} = \frac{1}{K} \sum_{j=1} ^K R_{j,d}$}}
\EndFor
\State{{Return: $d$ minimizing $R_d$}}
\end{algorithmic}
\caption{Estimation of the dimension $d$}
\label{alg:feature}
\end{algorithm}

\subsection{Simulation study}\label{sec:numerical_exp:sim}

In 
the examples presented below, $(X_i, Y_i)_{i=1,\ldots, n}$ is a collection of independent and identically distributed random copies of the pair $(X, Y)$, and $X$ is a vector of independent centered and unit Gaussian covariates with dimension $p\geq 8$. We consider the following four examples:
\\
\noindent{\textbf{Example 1.} The 
response $Y$ follows a Bernoulli distribution with parameter \cbstart $\expit(X_1)$. \cbend
\\
\noindent\textbf{Example 2.} The 
response is $Y= \sign\{ \sin(X_1)+ X_2^2 +0.2\epsilon \}$, where $\epsilon \sim \mathcal N (0,1)$. \\
\cbstart \noindent\textbf{Example 3.} The 
response is $Y= \sign\{(X_1+0.5)(X_2-0.5) +0.2\epsilon \}$, where $\epsilon \sim \mathcal N (0, 1)$. \cbend\\
\noindent\textbf{Example 4.} The 
response is $Y= \sign\{\log(X_1^2)(X_2^2 + X_3) +0.2\epsilon \}$, where $\epsilon \sim \mathcal N (0,1)$.

Example 1 is the simplest instance of a logistic regression model with a single relevant feature, while Examples 2, 3, and 4 are closely related to examples considered by \cite{meng2020sufficient}. We compare the proposed approach, with penalization LLO$(\lambda>0)$ and without penalization LLO$(\lambda=0)$, with the following existing competitors: \cbstart SAVE,
which is an inverse regression technique~\citep{cook1991discussion,li1992principal} (using the \texttt{dr} function from the \texttt{R} package of the same name), POTD, relying on optimal transport~\citep{meng2020sufficient} (implemented using code provided on Github\footnote{\href{https://github.com/ChengzijunAixiaoli/POTD}{https://github.com/ChengzijunAixiaoli/POTD})}), 
WSVM~\citep{shin2017principal} (using the \texttt{psvmSDR} package in \texttt{R}) and 
LGSIM~\citep{lambert2006local}. As we mentioned in Section~\ref{algorith1}, we select {the tuning parameter $\lambda$ in LLO$(\lambda>0)$ by CV, and the tuning parameter $k$ in LLO$(\lambda=0)$ by CV as well.} For LGSIM, we select the bandwidth parameter $h = c \cdot h_{Scott}$, where $h_{Scott}$ denotes Scott's rule, {by determining the scaling constant $c\in \{0.5,1,1.5,2,2.5,3,3.5,4,5,10\}$ using 5-fold cross-validation. Model performance is assessed using the misclassification error of a $10$-NN classifier with projected covariates as for LLO$(\lambda>0)$ in Section \ref{algorith1}. For WSVM, the bandwidth parameter $h \in \{2,2.5,3,3.5,4,5,6,7,8,9,10,12\}$} is tuned similarly via 5-fold cross-validation. \cbend

The methods are compared using, on the one hand, the distance between the estimated central subspace and the true central subspace using the Frobenius distance between the projection matrices on these subspaces, that is, $d(\mathcal{S}(\widehat{\beta}), \mathcal{S}({\beta}))= \| P_{\widehat{\beta}}  - P_{\beta} \|_F$ with $P_{\beta}  = \beta \beta^T$ for any orthogonal matrix $\beta$ of $p-$dimensional vectors. In each example, the central subspace is explicit: \cbstart
\begin{itemize}
\item In Example 1, the central subspace is spanned by the first vector of the canonical basis in $\mathbb{R}^p$, 
\item In Examples 2 and 3, it is spanned by the first two vectors of this basis, 
\item In Example 4, it is spanned by the first three vectors of this basis.
\end{itemize}
\cbend As such, Examples 1 to 4 are ranked in order of complexity of the dimension reduction problem.

On the other hand, these dimension reduction methods naturally give rise to classification procedures, in the following way: if $P$ denotes the projection matrix on an estimated central subspace of dimension $d$, then a nearest-neighbor classifier at $X=x$ is defined by the result of the majority vote among the nearest neighbors of $Px$. We therefore compare the nearest-neighbor classifiers obtained in this way using each of the dimension reduction procedures we consider; for this, and in each example, we generate independently a test data set on which all the nearest-neighbor classifiers are run (for this classification task we take $k=10$, as in Section~5 of~\cite{meng2020sufficient}), and their misclassification rates on this test data set are stored. It is worth noting that an easy dimension reduction problem may not automatically translate into an easy classification problem: in Example 1, for instance, the conditional probability that $Y=1$ is often close to $1/2$, meaning that the misclassification risk is bound to be high even if the central subspace is correctly identified. The distance to the central subspace and the misclassification risk thus give two different pieces of information about the accuracy of each method and the difficulty of each problem. We consider sample sizes $n=500, 1000, 2000, 3000, 4000$, and in each case the misclassification risk and the distance to the central subspace are averaged over $N=1000$ independent replications.

Results are represented in Figures~\ref{fig:distance+mcrisk12} and~\ref{fig:distance+mcrisk34}, first in the case when the dimension $d$ of the reduction subspace is assumed to be known and correctly specified in each example, and the dimension of the full covariate space is $p=8$. We observe that the penalized nearest-neighbor local logistic method LLO($\lambda>0$) performs best in terms of estimation of the central subspace, especially for low sample sizes \cbstart(except for Example 2 where non-penalized and penalized versions give similar results). The non-penalized version LLO$(\lambda=0)$ is substantially less accurate in Examples 1, 3, and 4\cbend, which highlights the importance of penalization. The nearest-neighbor majority vote classifier fed with the estimated central subspace obtained using the LLO($\lambda>0$) gradient estimates is also best among the tested methods as far as misclassification risk is concerned and performs almost as well as if the correct central subspace was used. \cbstart For larger sample sizes, LLO($\lambda>0$), POTD, and LGSIM appear to be the strongest overall competitors, while WSVM appears to be a poor performer in all examples for both small and large sample sizes. \cbend 

In Figure~\ref{fig:distance_Ex4}, we moreover examine how the methods perform when the correct dimension reduction subspace and sample size are fixed but the dimensions of the ambient space and estimated central subspace vary, \textit{i.e.} we consider Example~4 for $n=1000$ with $X$ a vector of standard Gaussian random variables having dimension $p\in \{8,16,32,64\}$, and various dimensions $1\leq d\leq 6$ for the estimated central subspace. The LLO($\lambda>0$) method again performs best overall, and markedly improves over competitors when $p\leq 16$. It is interesting to note that the best results for LLO($\lambda>0$) are found when the dimension $d$ is correctly specified if $p\leq 16$; for higher dimensions this is not necessarily the case, although a dimension close to the dimension of the correct central subspace will tend to yield better results. The results related to Examples 1 to 3 are similar and are deferred to 
Figures~\ref{fig:distances_pvarying}--\ref{fig:mcrisks} 
for the sake of brevity.

\begin{figure}[t!]
	\begin{subfigure}{.45\textwidth}
		\centering
		\includegraphics[width=1\linewidth]{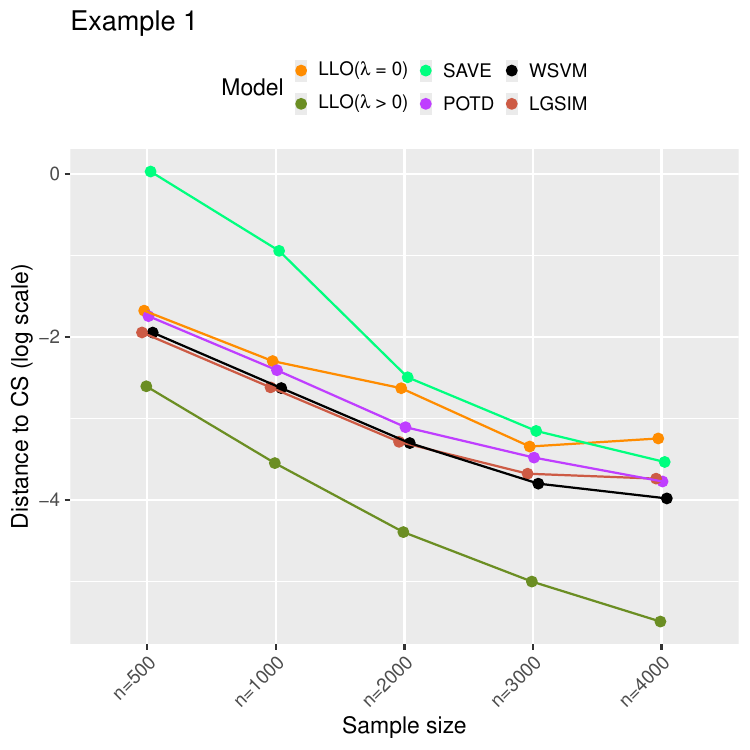}  
		\caption{{Distance to central subspace}}
	\end{subfigure}
    \hfill
	\begin{subfigure}{.45\textwidth}
		\centering
		\includegraphics[width=1\linewidth]{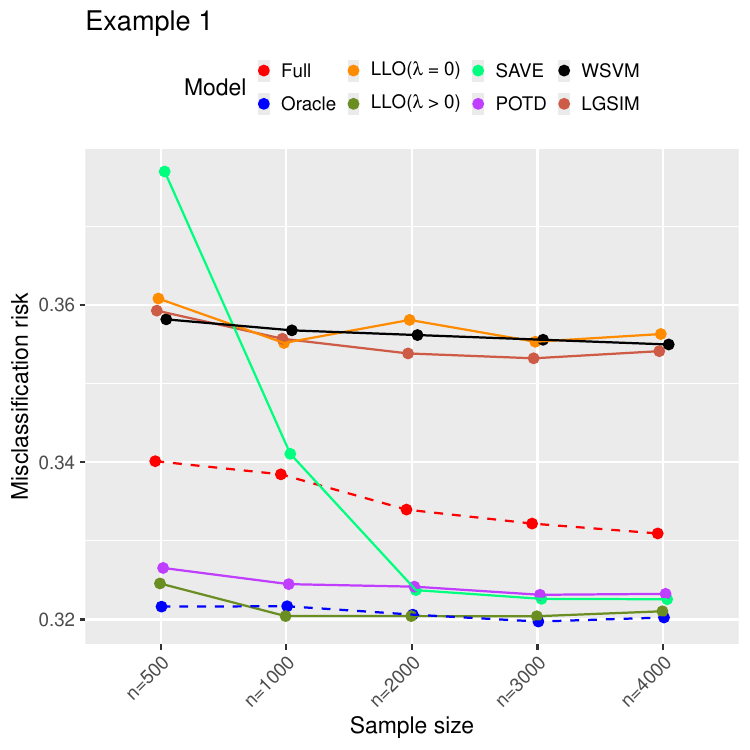}  
		\caption{Misclassification risk}
	\end{subfigure}
	\newline
	\begin{subfigure}{.45\textwidth}
		\centering
		\includegraphics[width=1\linewidth]{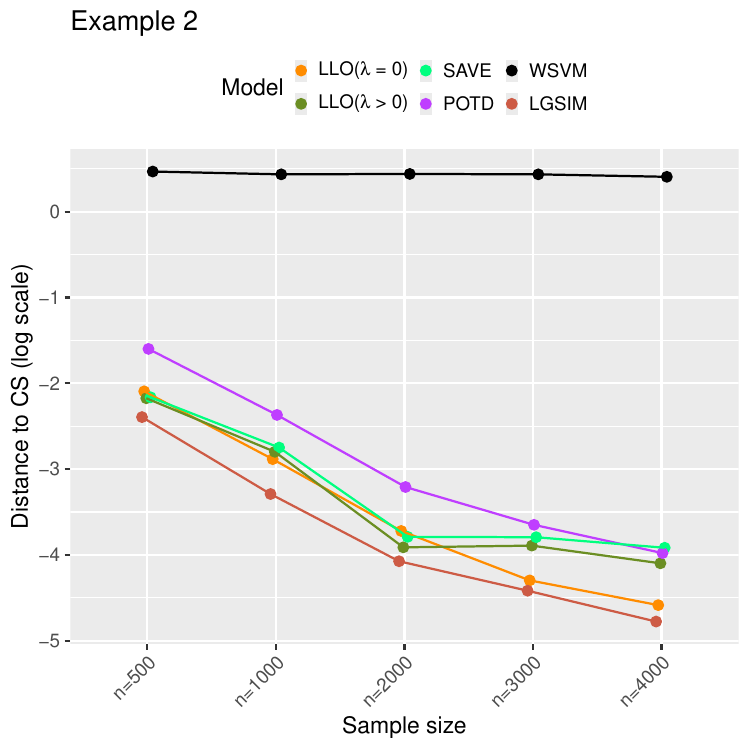}  
		\caption{{Distance to central subspace}}
	\end{subfigure}
    \hfill
    \begin{subfigure}{.45\textwidth}
		\centering
		\includegraphics[width=1\linewidth]{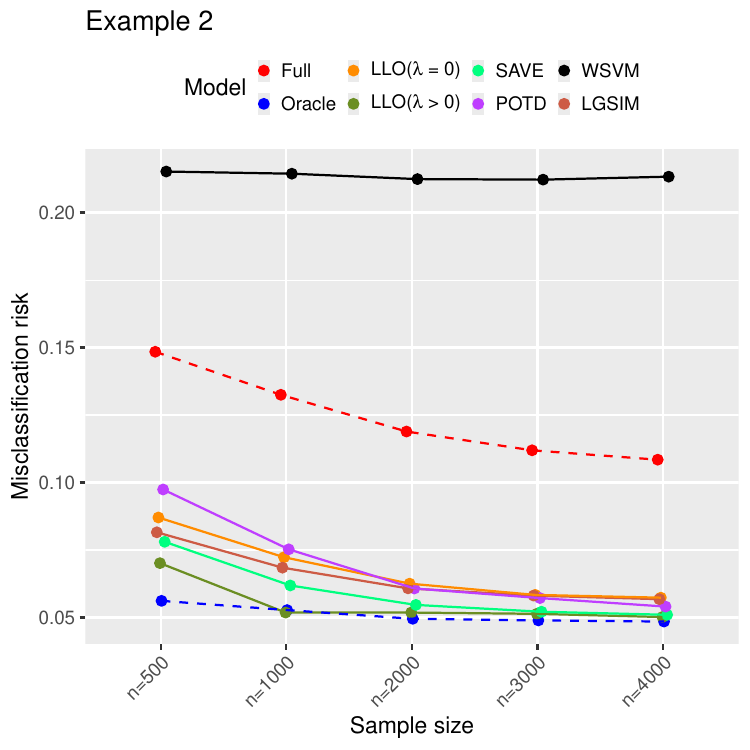}  
		\caption{Misclassification risk}
	\end{subfigure}
	\caption{Simulation study -- Distance to the central subspace (left, on the log scale) and misclassification risk (right) in Example 1 (top) and Example 2 (bottom), averaged over $N=1000$ replications in each situation. The covariate has dimension $p=8$ and the dimension $d$ is chosen as the dimension of the correct population central subspace (\textit{i.e.}~$d=1$ in Example 1 and $d=2$ in Example 2). In the right-hand panels, the misclassification risk represented relates to the $10$-NN classifier using the set of projected covariates on the estimated subspace produced by each method; in addition, ``Full'' denotes this classifier on the full, non-projected set of covariates, and ``Oracle'' denotes this classifier using the covariates projected on the correct population central subspace.}
	\label{fig:distance+mcrisk12}
\end{figure}

\begin{figure}[t!]
	\begin{subfigure}{.45\textwidth}
		\centering
		\includegraphics[width=1\linewidth]{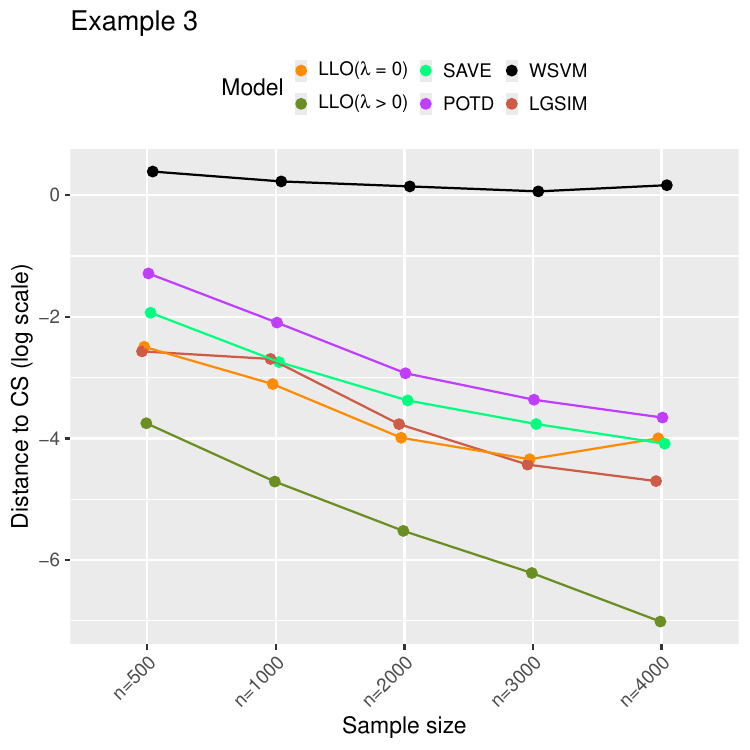}  
		\caption{{Distance to central subspace}}
	\end{subfigure}
    \hfill
	\begin{subfigure}{.45\textwidth}
		\centering
		\includegraphics[width=1\linewidth]{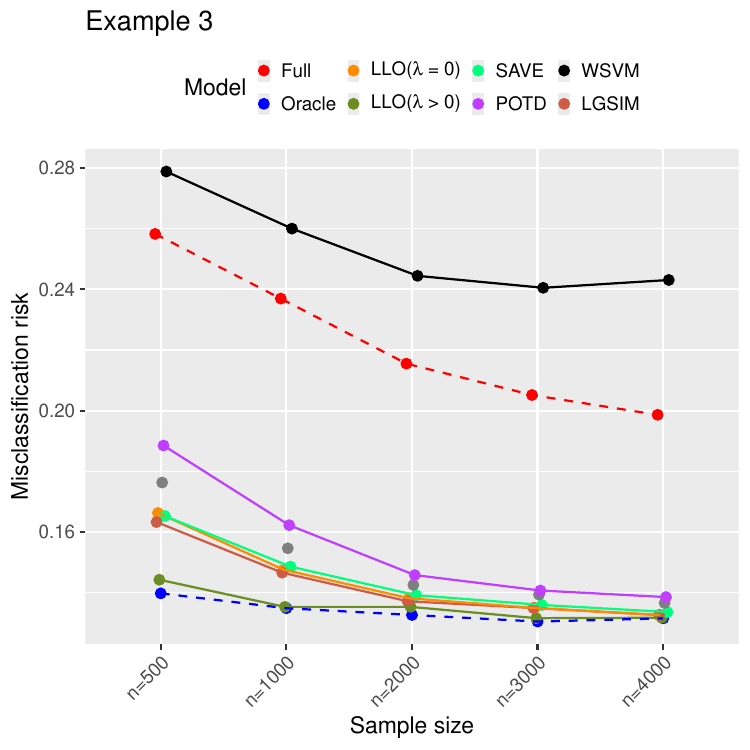}  
		\caption{Misclassification risk}
	\end{subfigure}
	\newline
	\begin{subfigure}{.45\textwidth}
		\centering
		\includegraphics[width=1\linewidth]{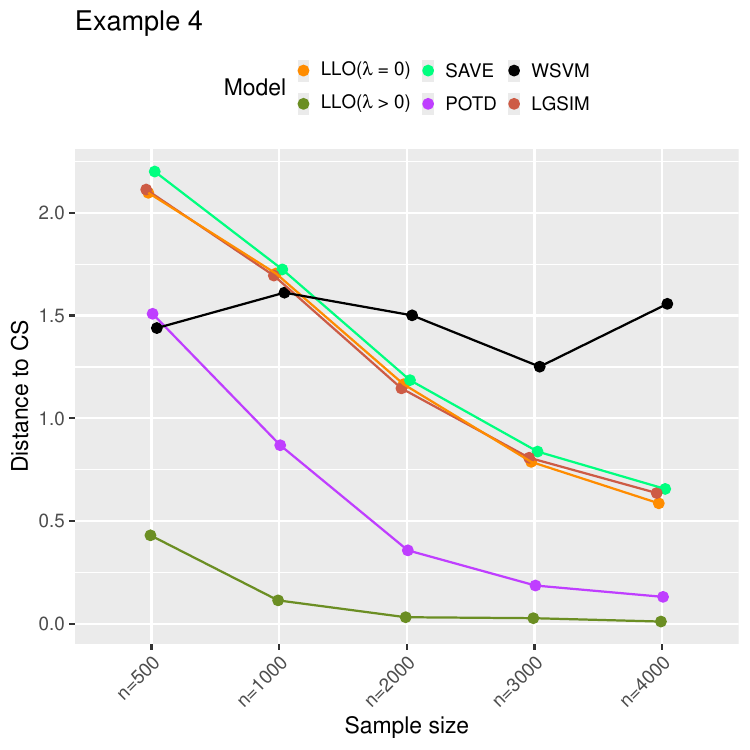}  
		\caption{{Distance to central subspace}}
	\end{subfigure}
    \hfill
    \begin{subfigure}{.45\textwidth}
		\centering
		\includegraphics[width=1\linewidth]{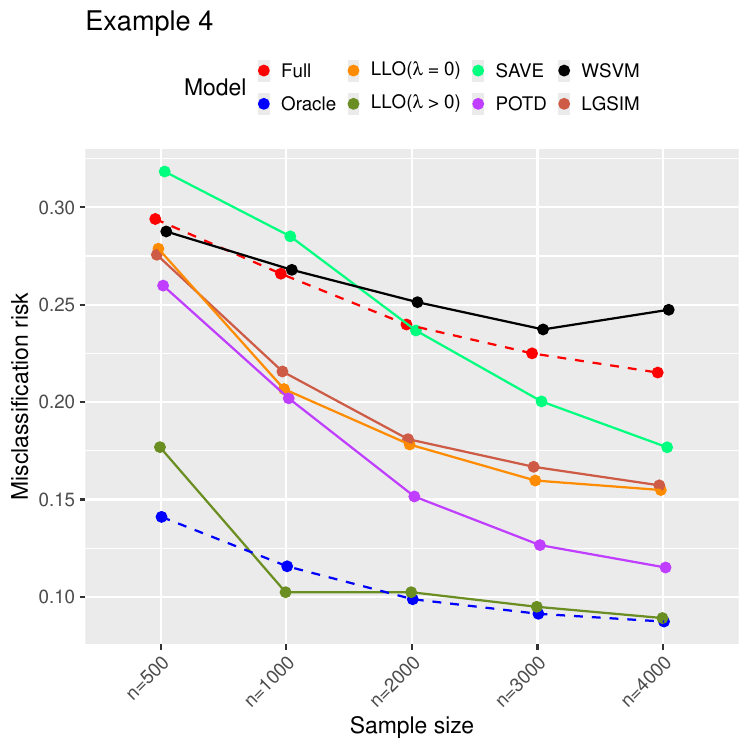}  
		\caption{Misclassification risk}
	\end{subfigure}
	\caption{Simulation study -- Distance to the central subspace (left) and misclassification risk (right) in Example 3 (top) and Example 4 (bottom), averaged over $N=1000$ replications in each situation. The covariate has dimension $p=8$ and the dimension $d$ is chosen as the dimension of the correct population central subspace (\textit{i.e.}~$d=2$ in Example 3 and $d=3$ in Example 4). In the right-hand panels, the misclassification risk represented relates tto the $10$-NN classifier using  the set of projected covariates on the estimated subspace produced by each method; in addition, ``Full'' denotes this classifier on the full, non-projected set of covariates, and ``Oracle'' denotes this classifier using the covariates projected on the correct population central subspace.}
	\label{fig:distance+mcrisk34}
\end{figure}

\begin{figure}[t!]
 \begin{subfigure}{.45\textwidth}
		\centering
		\includegraphics[width=1\linewidth]{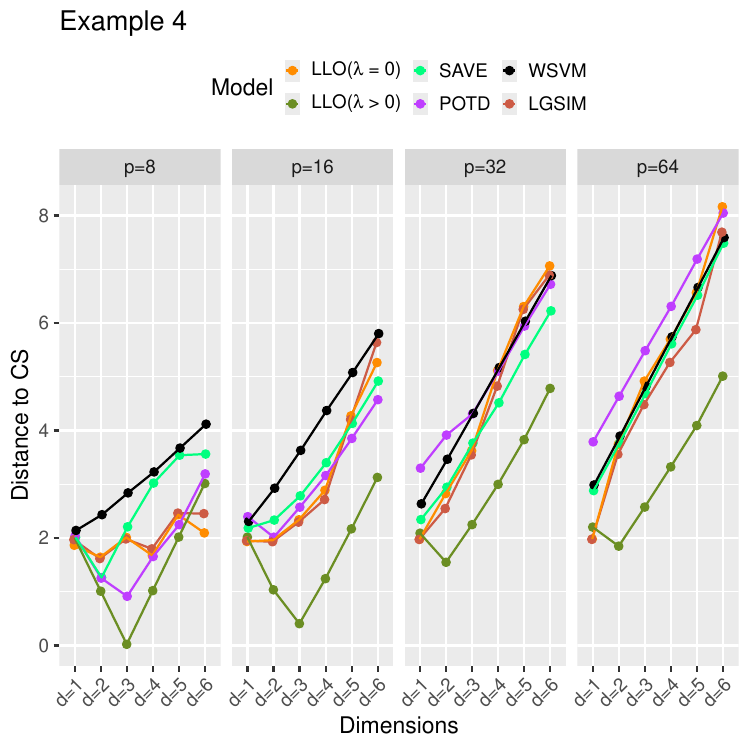}  
		\caption{{Distance to central subspace}}
	\end{subfigure}
 \begin{subfigure}{.45\textwidth}
		\centering
		\includegraphics[width=1\linewidth]{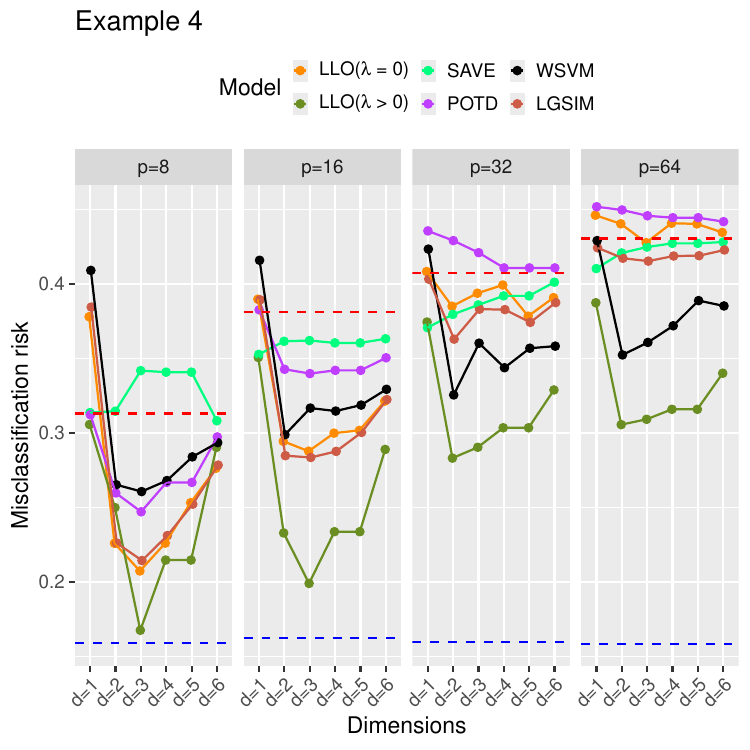}  
		\caption{Misclassification risk}
	\end{subfigure}
	\caption{Simulation study -- Distance to the central subspace (left) and misclassification risk (right) in Example 4, averaged over $N=1000$ replications of a sample of size $n=1000$, as a function of the dimension $d\in \{1,\ldots,6\}$ of the estimated central subspace and $p\in \{8,16,32,64\}$ of the full covariate space. In the right-hand panels, the red dashed line corresponds to the nearest-neighbor classifier with $d=p$, and the blue dashed line corresponds to this classifier using the covariates projected on the correct population central subspace.}
	\label{fig:distance_Ex4}
\end{figure}

We finally assess the performance of our complete workflow, namely, dimension reduction and selection through Algorithms~\ref{alg:Mestimate} and~\ref{alg:feature} and then classification using the nearest-neighbor classifier. Again, we focus on Example~4 for the sake of brevity, in the case $n=1000$ with $X$ a vector of independent centered and unit Gaussian random \cbstart covariates \cbend having dimension $p\in \{8,16,32,64\}$. We first ran a series of experiments to check the accuracy of our cross-validation dimension selection procedure outlined in Algorithm~\ref{alg:feature} under the same experimental design. It is seen in the left panel of Figure~\ref{fig:analysis_Ex4} that in the majority of cases, the dimension selected is within one unit of the correct dimension, and that dimension selection is typically quite accurate for $p\leq 16$. In the right panel of Figure~\ref{fig:analysis_Ex4}, we represent the misclassification risks of the LLO$(\lambda=0)$ and LLO$(\lambda>0)$ methods following the full workflow, and we compare their misclassification rates with the misclassification probability of their versions when $d$ is chosen as the dimension of the correct central subspace. It is striking that the rate of misclassification tends to be lower when following the full workflow (and therefore when estimating the dimension by cross-validation on the misclassification probability) than when using the correct dimension. \cbstart The interpretation is that 
when the dimension is large, estimating each vector $\widehat{b}_n(x)$ used to construct the matrix $\widehat{M}$ becomes more difficult, which can lead to substantial estimation errors. As expected, this estimation error in $\widehat{M}$ propagates to the estimated eigenvectors, so that directions that are truly relevant for predicting $Y$ may only appear among a small subset of eigenvectors. This 
can result in relevant coordinates being identified at a smaller eigenvalue order than in the population case (as in Figure \ref{fig:analysis_Ex4}). As a consequence, the dimension that should be selected based on the estimated matrix $\widehat{M}$ may be smaller (or larger) than the true underlying dimension. \cbend
It is also apparent that the performance of the full workflow tends to be slightly more robust to an increase in the dimension of the ambient covariate space. The results in Examples~1 to~3 are broadly similar and can be found in 
Figures~\ref{fig:dimension_select}--\ref{fig:dim_comparison}.

\begin{figure}[t!]
\begin{subfigure}{.45\textwidth}
		\centering
		\includegraphics[width=1\linewidth]{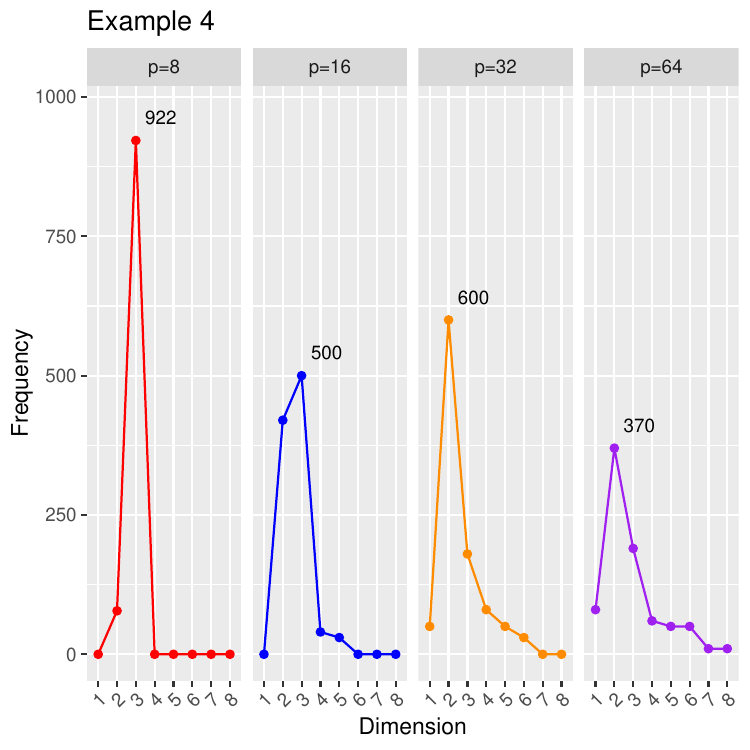}  
\end{subfigure}
\begin{subfigure}{.45\textwidth}
		\centering
		\includegraphics[width=1\linewidth]{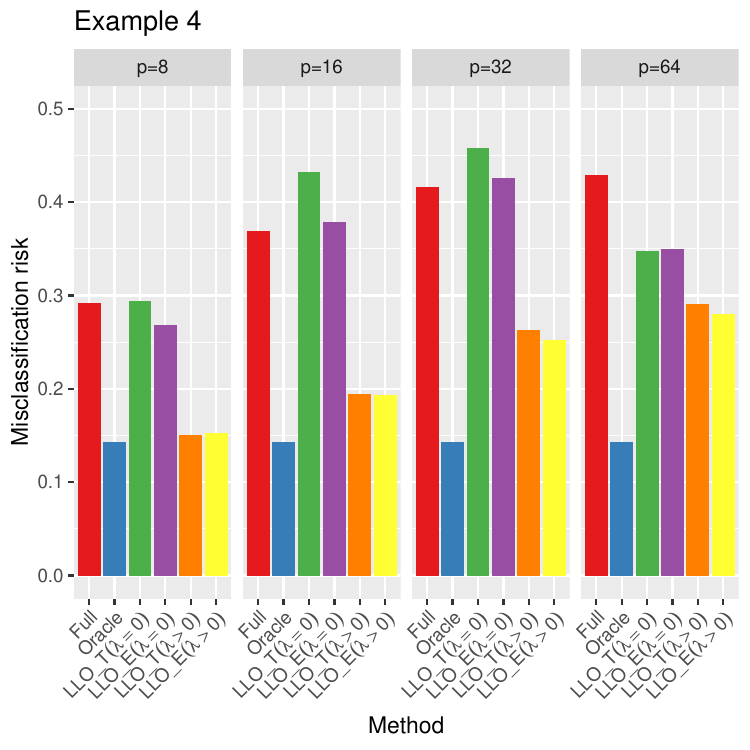}  
\end{subfigure}
	\caption{Simulation study -- Left panel: Dimension selection through Algorithm~\ref{alg:feature}, where the number indicated above the curve gives the number of times the dimension selected in the (absolute or relative) majority of cases was chosen. Right panel: Misclassification risk of the nearest-neighbor classifier with, from left to right, $d=p$ (red bar), the covariates projected on the correct population central subspace (blue bar), the central subspace estimated using the non-penalized LLO($\lambda=0$) method under correct specification of the dimension (green bar) and with the dimension estimated by cross-validation (purple bar), and the central subspace estimated using the penalized LLO($\lambda>0$) method under correct specification of the dimension (orange bar) and with the dimension estimated by cross-validation (yellow bar). Both panels are produced using $N=1000$ independent replications of a sample of size $n=1000$ and dimensions $p\in \{8,16,32,64\}$ of the full covariate space are considered.}
	\label{fig:analysis_Ex4}
\end{figure}

\subsection{Real data analyses}\label{sec:numerical_exp:data}

\cbstart We apply the proposed methodology to three real datasets, two obtained from the UCI repository and one from the NASA POWER data archive: \cbend
\begin{itemize}[wide,labelindent=0pt]
\item  The Hill-Valley (HV) dataset\footnote{\href{https://archive.ics.uci.edu/dataset/166/hill+valley}{https://archive.ics.uci.edu/dataset/166/hill+valley}}. Each data point is made of 100 real numbers $\bm{x}_i=(x_{i,j})_{1\leq j\leq 100}$ 
which create a curve in the two-dimensional plane that features a hill (a ``bump'' in the curve) or a valley (a ``dip'' in the curve). The data consists of the $n=1212$ pairs $(Y_i,\bm{x}_i)\in \{ 0,1 \}\times \mathbb{R}^{100}$, where 
$Y_i=1$ if and only if the curve features a hill.

\item \cbstart The Rennes city precipitation (RP) dataset\footnote{\href{https://power.larc.nasa.gov/data-access-viewer/}{https://power.larc.nasa.gov/data-access-viewer/}}. 
After data cleaning, the dataset contains $n=1826$ observations with $p=28$ different environmental features regarding daily precipitation for the city of Rennes, France, from January 2021 to December 2025. We treat precipitation as the response and make it binary, categorizing days as ``dry'' when precipitation is less than 0.5mm and ``wet'' otherwise. 
\cbend

\item The Wisconsin Diagnostic Breast Cancer (WDBC) dataset\footnote{\href{https://archive.ics.uci.edu/dataset/17/breast+cancer+wisconsin+diagnostic}{https://archive.ics.uci.edu/dataset/17/breast+cancer+wisconsin+diagnostic}} used in~\cite{shin2014probability}. A total of $n=569$ subjects are diagnosed with breast tumors, either benign or malignant. Ten features of breast cell nuclei are measured for each subject, with the mean, standard error, and largest values recorded for each feature, leading to $p=30$ predictors in total.
\end{itemize}

\begin{table}[H]
\small
    \centering
\caption{
 {Misclassification risk, the area under the ROC curve, and computing time (in seconds) for each classifier applied to the three real datasets. In each case, the prediction exercise is carried out on the selected testing set.}}
\begin{tabular}{ p{3cm} p{0.1cm} p{0.8cm} p{0.8cm}p{0.8cm} p{0.011cm} p{0.8cm}p{0.8cm} p{0.8cm}}
\toprule
{Classifier} & \multicolumn{4}{c}{Random Forest}  &\multicolumn{4}{c}{knn} \\
 \cmidrule(l){3-5}\cmidrule(l){7-9}
 &{}& Miscl. risk &{AUC} & {Est. time}  & {}&Miscl. risk &{AUC} & {Est. time}\\
 \hline
  \text{Hill-Valley (HV)}  &&  & & & &  & \\
  \hline
  \text{No dimension reduction
  } && 0.437 & 0.563 & 3.95&&   0.481 & 0.516 & 1.48 \\
 \text{LLO$(\lambda=0)$}  && 0.212 & 0.855 & 3.16  && 0.429 & 0.597 & \textbf{0.73}\\
 \text{LLO$(\lambda>0)$}   && \textbf{0.115} & \textbf{0.952} & \textbf{1.44} && \textbf{0.126} & \textbf{0.953} & \textbf{0.73}\\

\hline
 \text{Rennes Precipitation (RP)}  &&  & & & &  & \\
  \hline
 \text{No dimension reduction
  } && 0.172 & 0.0.884 & 0.36 &&  0.168 & 0.909 & 3.11 \\
 \text{LLO$(\lambda=0)$}  && \textbf{0.154} & 0.918 & 0.16 && 0.159 & 0.919 & 2.64\\
 \text{LLO$(\lambda>0)$}   && \textbf{0.154} & \textbf{0.920} & \textbf{0.14} &&  \textbf{0.153}  & \textbf{0.922} & \textbf{1.93 }\\
 \hline

  \text{Wisconsin Diagnostic Breast Cancer (WDBC)} &&  & & & &  & \\
 \hline
 \text{No dimension reduction
  } && 0.064 & 0.979 & 0.88  && 0.058 & 0.972 & 0.41 \\
 \text{LLO$(\lambda=0)$} && 0.053 & 0.987 & 0.8 &&  0.058 & \textbf{0.985}  & \textbf{0.28}\\
 \text{LLO$(\lambda>0)$}  && \textbf{0.035} & \textbf{0.99} & \textbf{0.69} &&  \textbf{0.029} & 0.982 & 0.32\\
\bottomrule
\end{tabular}
\label{tab:realapp_classi}
\end{table}

Each dataset is divided at random into a training set and a testing set, approximately made of $70\%$ and $30\%$ of the original data, respectively. Table \ref{tab:realapp_classi} compares the performance of the complete workflow (dimension reduction and selection and then classification) using the full space of covariates, the covariates projected on the dimension reduction subspace provided by the proposed LLO$(\lambda>0)$ method, and its version obtained using the non-penalized version LLO$(\lambda=0)$, paired with either the \texttt{knn} or the \texttt{RandomForest} classifier, when applied to the testing set. The classification procedure using LLO$(\lambda>0)$ generally has a comparable or lower misclassification risk, a comparable or higher AUC (this can also be seen by comparing the ROC curves of each classification procedure, see 
Figure~\ref{fig:realapp_ROC}) 
with comparable or lower computing time with respect to the non-penalized version, and always improves substantially upon the classifier not featuring dimension reduction. 
We note that, as expected, the cross-validation procedure for the dimension of the reduction subspace involving the estimation of the matrix $M$ via LLO$(\lambda>0)$ tends to select fewer components than its analogue using LLO$(\lambda=0)$, with comparable or higher accuracy, as shown in the top panels in Figure~\ref{fig:comparison_data} for the nearest-neighbor classifier (see 
Figure~\ref{fig:Dim_selection_rf} for similar results with the Random Forest classifier). It is noted that in the WBDC real data analysis, all the eigenvalues of the empirical outer product $\widehat{M}$ were found to be 0 from dimension $d=15$ and $d=22$ onwards when using the LLO($\lambda>0$) and LLO($\lambda=0$) method, respectively.

\begin{figure}[t]
	\centering
	\begin{subfigure}{.32\textwidth}
		\centering
\includegraphics[width=1\linewidth]{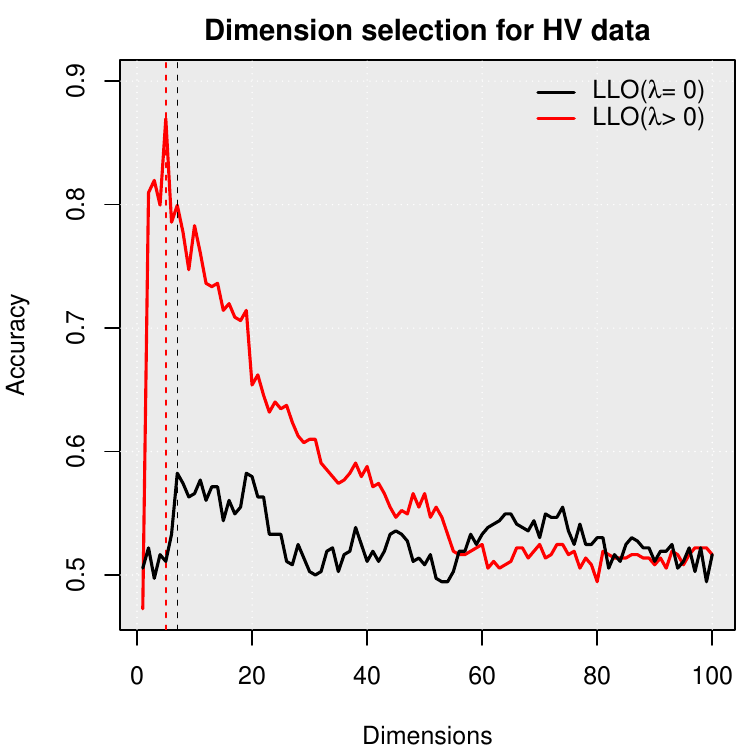} 
		\caption{HV -- Dimension selection}
	\end{subfigure} \hfill
\begin{subfigure}{.32\textwidth}
		\centering
		\includegraphics[width=1\linewidth]{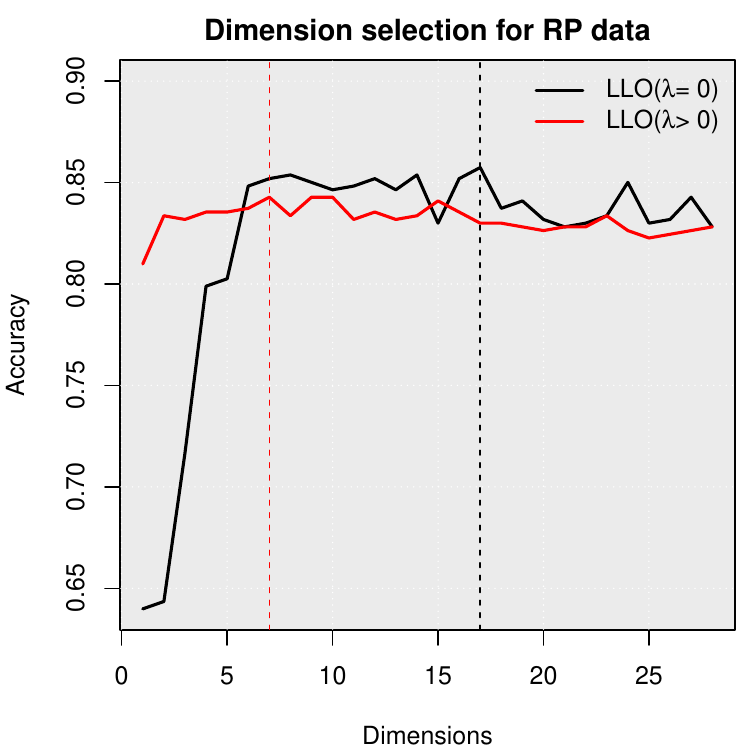}  
		\caption{\textcolor{black}{RP} -- Dimension selection}
	\end{subfigure} \hfill
    \begin{subfigure}{.32\textwidth}
		\centering
		\includegraphics[width=1\linewidth]{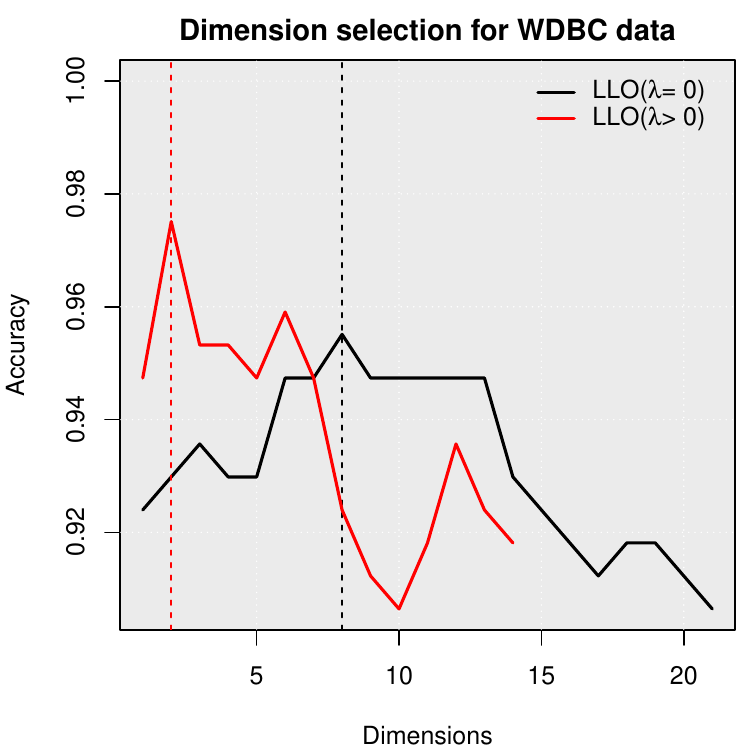} 
		\caption{\textcolor{black}{WDBC} -- Dimension selection}
	\end{subfigure} \hfill

    \begin{subfigure}{.32\textwidth}
		\centering
		\includegraphics[width=1\linewidth]{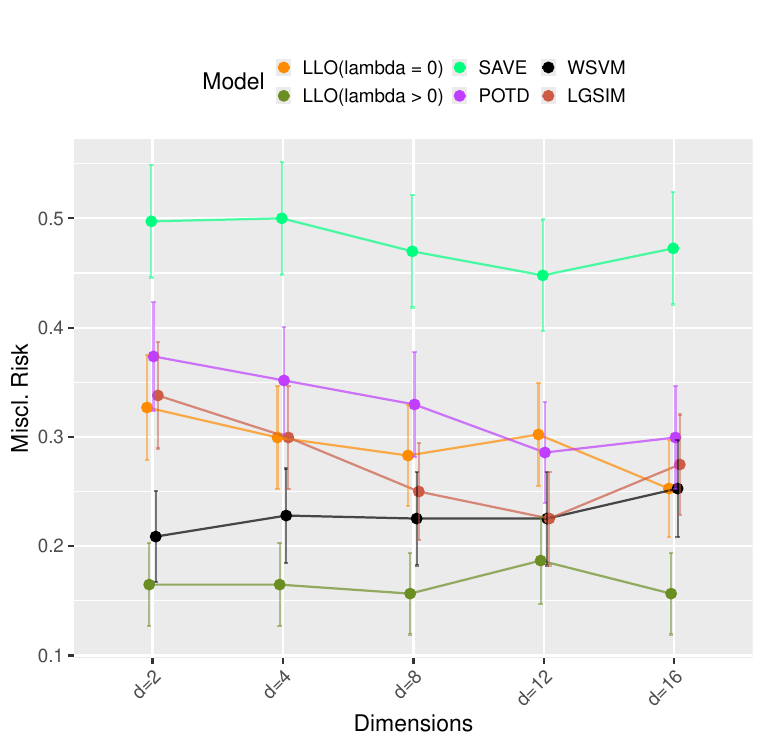}  
		\caption{HV -- Comparison between methods}
	\end{subfigure} \hfill
    \begin{subfigure}{.32\textwidth}
		\centering
\includegraphics[width=1\linewidth]{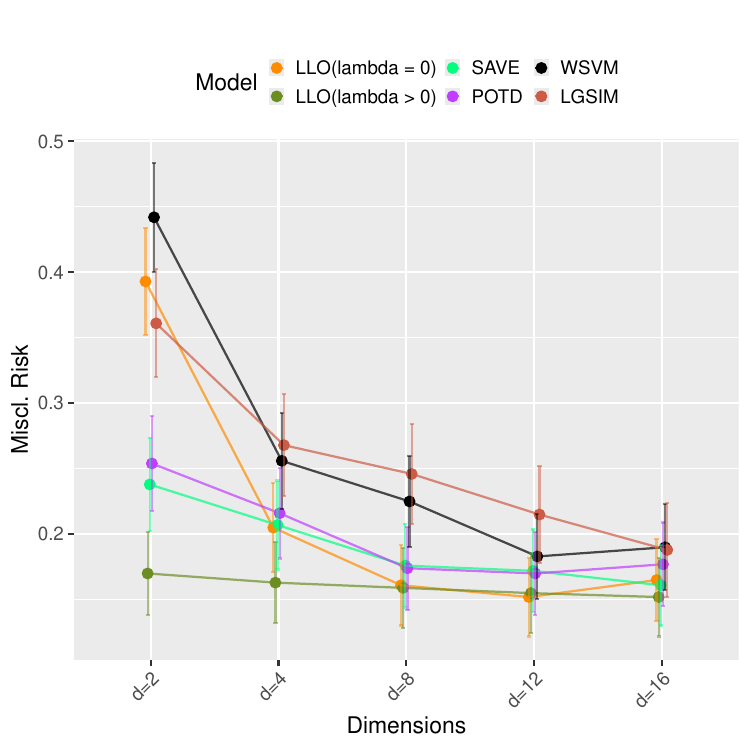}  
		\caption{RP -- Comparison between methods}
	\end{subfigure} \hfill
 	\begin{subfigure}{.32\textwidth}
		\centering
\includegraphics[width=1\linewidth]{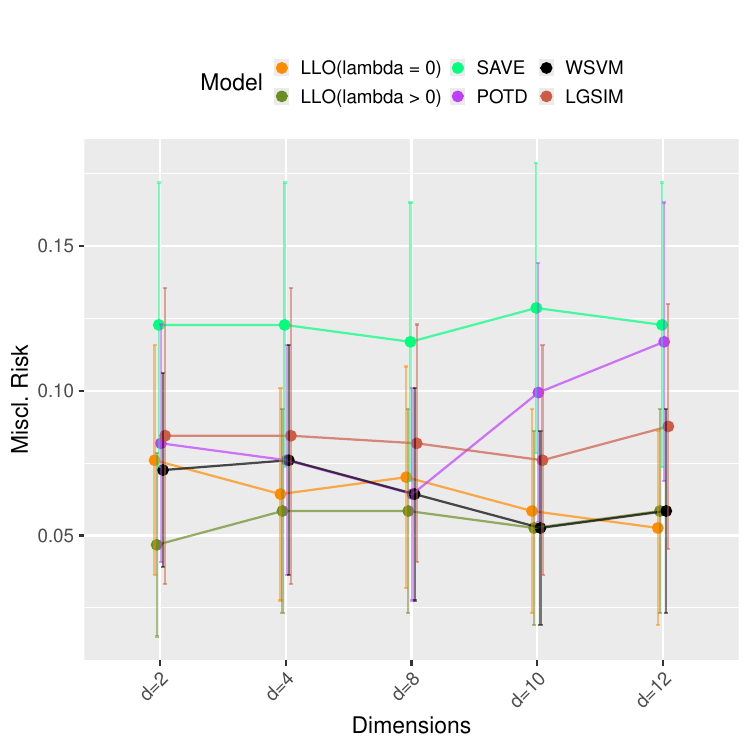}  
		\caption{WDBC -- Comparison between methods}
	\end{subfigure}
	\caption{Real data analysis - Top panels: Dimension selection through cross-validation for LLO($\lambda=0$) and LLO($\lambda>0$), using the nearest-neighbor classifier. Bottom panels: Estimated misclassification risk related to the nearest-neighbor classifier on the set of projected covariates on the estimated subspace produced by each method for a dimension $d$ of the central subspace 
    in $\{2, 4, 8, 12, 16\}$ for HV and RP and $\{2, 4, 8, 10, 12\}$ for WDBC, 
    with 95\% asymptotic Gaussian Wald-type confidence intervals (the sample size used is the size of the testing set).}
	\label{fig:comparison_data}
\end{figure}

\cbstart We also compare these classifiers to those obtained after dimensionality reduction using the SAVE, POTD, WSVM, and LGSIM methods. The misclassification risk is compared for estimated central subspaces of varying dimension $d$. 
\cbend It can be seen in the bottom panels of Figure~\ref{fig:comparison_data} that the classifier obtained after dimension reduction through the proposed LLO$(\lambda>0)$ approach is consistently superior to its competitors. The classifier based on LLO$(\lambda=0)$ dimension reduction performs worse but its performance gets closer to that of the LLO$(\lambda>0)$-based classifier as the dimension of the reduction subspace increases. Similar conclusions can be reached by considering 
ROC curves; 
see 
Figures~\ref{fig:comparison_data_roc_HV}--\ref{fig:comparison_data_roc_WDBC}.

\cbstart \section{Discussion and perspectives}
\label{sec:conclusion}

The proposed gradient estimation method is flexible and demonstrates several advantages over competing approaches for dimension reduction in classification. The theoretical framework developed in this work could potentially be extended to derive further results beyond those considered here, \textit{e.g.}~Nadaraya-Watson-type localization with nonrandom bandwidth. 

 A number of aspects warrant further investigation. A first direction for future research is the development of finite-sample guarantees that hold uniformly over the covariate space. To be more specific, it would be of interest to derive uniform upper bounds on quantities of the form 
\[
\sup_{x \in \mathbb{R}^p} \left\| \widehat{b}_n(x) - \nabla \ell(x) \right\|,
\]
where \(\|\cdot\|\) denotes a norm on \(\mathbb{R}^p\). Establishing such results is challenging, as the pointwise weak convergence result in Theorem~\ref{theo:main} cannot be directly extended to a \textit{uniform weak convergence} result in $\ell_\infty(S)$ because localization produces estimators that are not tight when viewed as functions of the covariate value. We nonetheless conjecture that it is possible to extend Corollary~\ref{coro:main} to a uniform result, although it requires moving beyond the asymptotic representations developed in \cite{hjopol1993}. Instead, one could rely on finite-sample uniform deviation bounds for minimizers of convex objective functions, as for instance in \cite{pmlr-v108-dalalyan20a}. Moreover, the analysis of the localization step must itself be carried out uniformly over the covariate space. Doing so would require stronger assumptions, such as uniform boundedness conditions of the form $f(x) \geq b>0$ and $\pi(x)\in [\pi_0,1-\pi_0]\subset (0,1)$ for all $x\in S$. This kind of uniform control has been investigated in the context of kernel estimators in~\cite{gine+g:02,dony2006uniform}, and for nearest-neighbor methods in~\cite{jiang2019non} and in~\cite{portier2021nearest}. Extending these techniques to the present framework of binary classification paired with LASSO penalization constitutes a nontrivial but promising line of research.

A second line of research is to investigate the behavior of the estimator considered in the high-dimensional setting. We recall that, in this article, we obtain a result characterizing the asymptotic behavior of the estimator as $n\to \infty$, in a fixed-$p$ setting. Within this fixed-dimensional framework, our result describes the limiting distribution of the (suitably rescaled) estimator; in the absence of localization, our analysis yields a corresponding asymptotic distribution result for the standard logistic Lasso. This framework is conceptually distinct from the high-dimensional setting, where the ambient dimension \( p \) may be large relative to the sample size \( n \). Such situations are typically studied under asymptotic frameworks in which \( p = p_n \) grows with \( n \), or alternatively where \( p \) is fixed but allowed to be sufficiently large without degrading much the statistical accuracy. The latter perspective is often supported by sharp non-asymptotic deviation inequalities. A common requirement in such analyses is the presence of low intrinsic complexity. For instance, in LASSO regression, one typically assumes that only a small subset of coefficients is nonzero, as emphasized in the work of \cite{tsybakov2009simultaneous}.

In the present framework, however, sparsity alone appears insufficient, because the nonparametric nature of the estimation introduces a bias term that scales poorly with the ambient dimension, namely, of order \( (1/n)^{1/p} \). Consequently, even if the estimated coefficients are sparse, any potential gains may be wiped out by this intrinsic bias term. To address this limitation and better capture the potential of our approach, it would be natural to exploit low-dimensional geometric structure in the covariates. Specifically, assuming that the data lie on (or near) a manifold of intrinsic dimension \( d_0 < p \) could significantly improve the bias scaling. This intuition is supported by recent work on nearest-neighbor regression with manifold-valued covariates \citep{NIPS2011_05f971b5}, as well as local linear regression using Nadaraya-Watson weights \citep{bickel}. In such settings, the bias term scales as \( (1/n)^{1/d_0} \), rather than \( (1/n)^{1/p} \), suggesting that the proposed method could naturally benefit from low-dimensional structure without requiring substantial modification. A rigorous investigation of this perspective is left for future work.

A third avenue for future work that we identify is the extension of the theory in order to accommodate other types of penalties. Two relevant examples would be the SCAD~\citep[Smoothly Clipped Absolute Deviation, see][]{fanli2001} penalty and MCP~\citep[Minimax Concave Penalty, see][]{zha2010}. These penalty functions are known to have oracle properties in identifying the active set and the efficient estimation of nonzero coefficient. They are, however, not convex, so that the theory developed in this paper does not apply. While it is beyond the scope of this paper, it will be interesting to investigate whether the oracle properties of the SCAD and MCP penalties extend to the local classification context considered here.

To conclude, we highlight one further avenue for future research. Most results in the sufficient dimension reduction literature concern the estimation of the subspace itself, and rarely address the use of the estimated basis for the downstream classification or regression task. Some results on regression function estimation exist: \cite{cadre2010dimension} consider a nearest-neighbour approach and establish consistency and rates of convergence, while \cite{forzani2024asymptotic} derive pointwise weak convergence for the Nadaraya--Watson estimator together with uniform consistency. Both of these references are confined to the regression setting, and we are not aware of any work establishing excess risk bounds in the classification context. We believe this is an important open question, particularly given the broad variety of methods available for dimension reduction and representation learning more generally. 
\cbend 
\section*{Acknowledgments}

T.~Ahmad acknowledges support from the R\'egion Bretagne through project SAD-2021-MaEVa. G.~Stupfler acknowledges support from grants ANR-19-CE40-0013 (ExtremReg project), ANR-23-CE40-0009 (EXSTA project) and ANR-11-LABX-0020-01 (Centre Henri Lebesgue), the TSE-HEC ACPR Chair ``Regulation and systemic risks'', and the Chair Stress Test, RISK Management and Financial Steering of the Foundation Ecole Polytechnique.

\bibliographystyle{chicago}
\bibliography{Bibliography}

\newpage
\renewcommand{\thetable}{S.\arabic{table}}
\setcounter{table}{0}
\renewcommand{\thefigure}{S.\arabic{figure}}
\setcounter{figure}{0}

\renewcommand{\theequation}{S.\arabic{equation}}
\setcounter{equation}{0}
\renewcommand{\thesection}{S.\arabic{section}}
\setcounter{section}{0}
\setcounter{page}{1}


\begin{center}
\section*{Supplementary material for \\ ``Nearest-neighbor LASSO logistic regression for the gradient''}

\end{center}
\section{Mathematical proofs}
\label{sec:proofs}

\subsection{Weak convergence of argmins of convex functions}

We first adapt a result from \cite{hjopol1993} to obtain a convergence result on the minima of our loss function, which has a very specific structure due to the LASSO penalty. Recall that a function $F: \mathbb R ^q \to \mathbb R $ is said to be $\mu$-strongly convex, for a given $\mu>0$, whenever $u\mapsto F (u) - \frac{1}{2}\mu\|u\|_2^2 $ is convex. In particular: 
\begin{itemize}
\item Strongly convex functions on $\mathbb R ^q$ have a unique global minimum on $\mathbb R ^q$: this follows from the fact that (i) strongly convex functions on $\mathbb R ^q$ are strictly convex and continuous, and (ii) an equivalent definition of strong convexity of $F$ is that, for any $x,x_0\in \mathbb R ^q$ and any subgradient $v$ of $F$ at $x_0$, 
\begin{equation}
\label{eqn:strong_cvx}
F(x)\geq F(x_0) + v^T (x-x_0) + \frac{\mu}{2} \| x-x_0 \|_2^2,
\end{equation}
and therefore strongly convex functions are bounded from below by strictly convex polynomials of degree 2, meaning that they tend to $+\infty$ as $\| x \|_2\to +\infty$.
\item If $\Gamma\in \mathbb R^{q\times q}$ is symmetric and positive definite, and $f: \mathbb R ^q \to \mathbb R $ is convex, then $F:u \mapsto u^T \Gamma  u + f(u)$ is $\mu$-strongly convex, with $\mu/2$ being the smallest eigenvalue of $\Gamma$.
\end{itemize}
We may now state our first result, on the convergence of minimizers of random convex functions approximated pointwise by random strongly convex functions.
\begin{lemma}
\label{lem:pollard}
Let $F:\mathbb R^q \to \mathbb R$ be $\mu$-strongly convex, $(S_n)$ be a random sequence and $B_n: u\in \mathbb R^q \mapsto u^T S_n + F(u) \in \mathbb R$. Let \cbstart {$\beta_n = \argmin _{u \in \mathbb R^q} B_n(u)$.} \cbend 
\begin{enumerate}[label=(\roman*)]
\item If $S_n = O_{\mathbb{P}}(1)$, then $\beta_n = O_{\mathbb{P}}(1)$.
\item Let moreover $A_n: \mathbb R^q \to \mathbb R$, $n\geq 1$, be a sequence of random convex functions. If, pointwise in $u\in \mathbb R^q$, 
\[
A_n (u) - B_n(u) \stackrel{\mathbb{P}}{\longrightarrow} 0,
\]
then $\argmin_{u\in \mathbb R^q} A_n(u)$ is nonempty for $n$ large enough and one can construct a measurable sequence $\alpha_n\in \mathbb R^q\cup\{\infty\}$, $n\geq 1$, such that $\alpha_n \in \argmin_{u\in \mathbb R^q} A_n(u)$ and $\| \alpha_n - \beta_n\|_2  = o_{\mathbb P} (1)$.
\end{enumerate}
\end{lemma}

\begin{proof}[Proof of Lemma~\ref{lem:pollard}] (i) For an arbitrary $s\in \mathbb R^q$, let $B_s: u\in \mathbb R^q \mapsto u^T s + F(u) \in \mathbb R$, so that $B_n\equiv B_{S_n}$. Then $B_s$ is $\mu$-strongly convex and has a unique minimizer $u(s)\in \mathbb R^q$ for any $s$. Let $s,t\in \mathbb R^q$. The characterization~\eqref{eqn:strong_cvx} of strong convexity and the fact that $u(s)$ (resp.~$u(t)$) minimizes $B_s$ (resp.~$B_t$) together imply that
\[
B_{s} ( u(t) ) - B_{s} ( u(s) ) \geq \frac{\mu}{2} \| u(t) - u(s)\|_2^2 
\]
and 
\[
B_{t} ( u(s) ) - B_{t} ( u(t) ) \geq \frac{\mu}{2} \| u(t) - u(s)\|_2^2.
\]
Consequently
\begin{align*}
\mu \| u(t) - u(s)\|_2  ^2 &\leq   B_{s} ( u(t) ) - B_{t} ( u(t) ) + B_{t} ( u(s) ) -B_{s} ( u(s) ) \\
&=  (u(t)- u(s))^T (s - t) \\
&\leq \|u(t) - u(s)\| _2 \|t - s\|_2 .    
\end{align*}
It follows that $u$ is Lipschitz continuous. Let now $\varepsilon>0$. Since $S_n=O_{\mathbb{P}}(1)$ and $\beta_n=u(S_n)$, there is a compact set $K\subset \mathbb{R}^q$ such that $\liminf_{n\to\infty} P(\beta_n\in u(K)) \geq \liminf_{n\to\infty} P(S_n\in K)\geq 1- \varepsilon$. By continuity of $u$, the image $u(K)$ of $K$ by $u$ is compact, so that indeed $\beta_n = O_{\mathbb{P}}(1)$.
\vskip1ex
\noindent
(ii) Lemma~2 in \cite{hjopol1993} and the discussion right before it~\citep[drawn from the appendix in][]{nie1992} yield the existence of a measurable sequence of minima $\alpha_n$ of $A_n$, satisfying
\begin{align*}
    &\mathbb{P} ( |\alpha_n -  \beta_n | \geq \delta ) \\
    &\leq \mathbb{P} \left( \sup_{ \|u -\beta_n \|_2\leq \delta} |A_n(u) -B_n(u) | \geq \frac{1}{2} \inf_{\| u-\beta_n\|_2=\delta} (B_n(u) - B_n(\beta_n)) \right).
\end{align*}
By strong convexity of $B_n$, 
\[
\mathbb{P} ( |\alpha_n -  \beta_n | \geq \delta ) \leq \mathbb{P} \left( \sup_{ \|u -\beta_n \|_2\leq \delta} |A_n(u) -B_n(u) | \geq \frac{\delta^2\mu}{4} \right).
\]
Consequently, it is enough to show that 
\[
\sup_{ \|u -\beta_n \|\leq \delta} |A_n(u) -B_n(u) | = o_{\mathbb{P}} (1).
\]
Let $\eta>0$ and, for any compact set $K\subset \mathbb{R}^q$, let $K_{\delta}$ be the (compact) set of those points $x\in \mathbb{R}^q$ whose distance to $K$ is not greater than $\delta$. Then 
\begin{align*}
\mathbb{P} \left(\sup_{ \|u-\beta_n \|_2\leq \delta} |A_n(u) -B_n(u) |  >\eta \right) &\leq \mathbb{P} \left(\sup_{ u\in K_ \delta} |A_n(u) -B_n(u) |  >\eta \right) \\
    &+ \mathbb{P}( \beta_n\notin K ).
\end{align*}
Combine Lemma 1 in \cite{hjopol1993} and (i) of the present Lemma to obtain that the left-hand side converges to 0, which is the required result.
\end{proof}

\subsection{Functional weak convergence of certain nearest-neighbor empirical processes (Theorem~\ref{theo:tightness} and its proof)}
\label{sec:proofs:tightness} 

\cbstart {To prove 
Theorem~\ref{theo:tightness}, }\cbend we first recall the following lemma on the convergence of the $k$-NN bandwidth; see Lemma~1 in~\cite{portier2021nearest}.
\begin{lemma}[\cite{portier2021nearest}]
\label{lem:portier}
Suppose that \ref{cond:new1} is fulfilled. Assume that $k:=k_n \to \infty$ is such that $k/n \to 0$. Then $\widehat{\tau}_{n,k}(x)/\tau_{n,k}(x) \stackrel{\mathbb{P}}{\longrightarrow} 1$.
\end{lemma}
 Recall the notation
\[
Z_n ( \tau ) = \frac{1}{\sqrt{k}} \sum_{i=1}^n \left\{\Psi_n (Y_i,X_i) \ind_{B(x,\tau )}(X_i) - \mathbb E [\Psi_n (Y,X) \ind_{B(x,\tau )} (X) ]  \right\}.
\]
We then state a result on the weak convergence of $Z_n(\tau)$ as a stochastic process, which is the key to the proof of Theorem~\ref{theo:tightness}. Set $s\wedge t=\min(s,t)$ and $s\lor t=\max(s,t)$ for $s,t\in \mathbb{R}$, and recall the notation $V_p$ for the volume of the unit Euclidean ball in $\mathbb{R}^p$. Let also $\ell^\infty ([1/2,3/2])$ denote the space of uniformly bounded vector-valued functions defined on $[1/2,3/2]$ (we do not emphasize the dimension of the image space for the sake of notational convenience). This is a metric space with respect to the uniform metric $d(f,g) = \sup_{t\in[1/2,3/2]} \| f(t)-g(t)\|_2$.
\begin{lemma}[Tightness and weak convergence of $Z_n$] 
\label{lem:tightness}
Let $E$ be a nonempty and finite set. Assume that the data is made of independent copies  $(Y_i,X_i)_{1\leq i\leq n}$  of the random pair $(Y,X) \in E\times \mathbb R^p$ and that \ref{cond:new1} is fulfilled. Let $\Psi_n : E \times \mathbb R^p \to \mathbb R^q $ be a sequence of measurable vector-valued functions and suppose 
that there is a positive integer $n_0$ such that 
\[
\Psi_{\infty} :=\sup_{n\geq n_0} \, \sup_{z\in A_{n,k}(x)} \, \max_{y\in E} \| \Psi_n(y,z)\|_2 <\infty,
\]
where $A_{n,k}(x) = B(x,(3/2)^{1/p}\tau_{n,k}(x))$. 
\begin{enumerate}[label=(\roman*)]
\item If $k:=k_n  \to \infty$ is such that $k/n \to 0$, then the stochastic process
\[
\left\{ Z_n (t ^{1/p}\tau_{n,k}(x)) \right\}_{ t \in [1/2,3/2] }
\]
is 
tight in $\ell^\infty ([1/2,3/2])$.
\item Let $\Sigma_n^2(X) = \mathbb{E}[ \Psi_n (Y, X)\Psi_n (Y, X)^T | X  ] $. If moreover there is a (positive semidefinite) matrix-valued function $t\mapsto \Sigma^2(t,x)$ such that 
\[
\forall t\in [1/2,3/2], \ \int_{B(0,1)} \Sigma_n^2( x + \tau_{n,k}(x) t^{1/p} v ) \diff v \to V_p \Sigma^2(t,x), 
\]
then this same stochastic process converges weakly in $\ell^\infty ([1/2,3/2])$ 
to a continuous Gaussian process with covariance function $(s,t)\mapsto (s \wedge t) \Sigma^2(s \wedge t,x)$.
\end{enumerate}
\end{lemma}
We shall prove (i) by showing tightness of any real-valued projection of the stochastic process of interest with respect to the uniform metric on $\ell^{\infty}([1/2,3/2])$ using general empirical process theory from~\cite{wellner1996} and~\cite{vandervaart1998} linking weak convergence on $\ell^{\infty}([1/2,3/2])$ to asymptotic uniform equicontinuity on $[1/2,3/2]$, equipped with the standard distance between real numbers. Continuity of the limiting process in (ii) on $[1/2,3/2]$ will then follow from Theorem 1.5.7 and Addendum 1.5.8 p.37 in~\cite{wellner1996}. As a side note, let us highlight that in doing so we shall in fact check the usual tightness conditions of the space of continuous functions on $[1/2,3/2]$, see Theorem~7.3 p.82 in~\cite{bil1999}; since the stochastic process of interest actually lives in the space of c\`adl\`ag functions on $[1/2,3/2]$, Theorem 13.4 p.142 in~\cite{bil1999} and its Corollary provide an alternative route to the proof of weak convergence and continuity of the limiting process.

We first recall a few definitions. Given a probability measure $Q$ on a measurable space $(S,\mathcal S)$, the metric space of square-integrable, Borel measurable real-valued functions on $S$ with respect to $Q$ is defined as
\begin{align*}
    &L_2(Q) \\
    &= \left\{ g: (S,\mathcal S)\to (\mathbb R,\mathcal{B}(\mathbb R)) \text{ such that  } \| g\|_{L_2(Q)} ^ 2 := Q(g^2) := \int_S g^2 dQ < \infty \right\} . 
\end{align*}
For two functions $\underline{f}, \overline{f} \in L_2(Q)$, the \textit{bracket} $[\underline{ f}, \overline{f}]$ is the set of all functions $g$ in $L_2(Q)$ such that $\underline{f} \leq g\leq  \overline{f}$ on $S$. A bracket $[\underline{ f}, \overline{f}]$ such that $\| \underline{ f} -  \overline{f} \| _{L_2(Q)} \leq \varepsilon$ is called an \textit{$\varepsilon$-bracket}, and for any $\mathcal G \subset L_2(Q)$, the \textit{$\varepsilon$-bracketing number}, denoted by $\mathcal N _{[\,]} (\mathcal G , L_2(Q) , \varepsilon)$, is defined as the smallest number of $\varepsilon$-brackets needed to cover $\mathcal G$. Finally, call an \textit{envelope function} for $\mathcal G$ any function $G:S\to \mathbb R$ such that $|g| \leq G$ on $S$ for any $g\in \mathcal G$.
\begin{proof}[Proof of Lemma~\ref{lem:tightness}] (i) Work on the space $T=[1/2,3/2]$ equipped with the usual distance between real numbers. Fix $u\in \mathbb{R}^q \setminus \{ 0 \}$. Let $P$ (resp.~$P_n$) denote the probability measure of $(Y,X)$ (resp.~the empirical probability measure on the set of pairs $(X_i,Y_i)$, $1\leq i\leq n$), so that $u^T Z_n (t^{1/p}\tau_{n,k}(x))=\sqrt n (P_n - P )( f_{n,t} )$ with 
\[
f_{n,t} (y,z) = \sqrt{ \frac{n}{k} } u^T \Psi_n (y,z) \ind_{B(x,t^{1/p}\tau_{n,k}(x) )} (z).
\]
An obvious envelope function for the set of measurable functions $\mathcal F_n = \{f_{n,t} :  t\in T\}$ is 
\[
F_n (y,z) = \sqrt{ \frac{n}{k} } \| u \|_2 \Psi_\infty  \ind_{B(x,(3/2)^{1/p}\tau_{n,k}(x) )} (z).
\]
Notice that when $n$ is large enough, $f_X(z)\leq 2 f_X (x)$ for any $z\in B( x,(3/2)^{1/p}\tau_{n,k}(x) )$. Therefore, for $n$ large enough, 
\[
P(F_n^2) \leq 2 f_X(x) \| u \|_2^2 \Psi_\infty^2 \times \frac{n}{k} \int_{B(x,(3/2)^{1/p}\tau_{n,k}(x) )} \diff z = 3 \Psi _\infty ^2 <\infty.     
\]
According to Theorem 2.11.23 p.221 in~\cite{wellner1996}~\citep[see also Theorem~19.28 p.282 in][]{vandervaart1998}, it is then sufficient to show that 
\begin{align}
\label{eqn:vdvw_1}
\forall \eta>0, \ \lim_{n\to\infty} P(F_n^2 \ind_{\{ F_n > \eta \sqrt n \}} ) &= 0,\\
\label{eqn:vdvw_2}
\mbox{and for every } \delta_n\downarrow 0, \ \lim_{n\to\infty}  \sup_{|t-s|\leq \delta_n } P [(f_{n,t} - f_{n,s} ) ^2]  &= 0 \\ 	
\label{eqn:vdvw_3}
\mbox{and } \lim_{n\to\infty}  \int _ 0 ^{\delta_n}   \sqrt { \log \mathcal N_{[\,]} \left(\mathcal F_n,  L_2( P )  ,  \varepsilon \| F_n \| _ {L_2( P) } \right)} \,\diff \varepsilon &=0.
\end{align}
Clearly 
\[
\forall \eta>0, \ P(F_n^2\ind_{\{ F_n >\eta \sqrt n \}}) \leq \frac{n}{k} \| u \|_2^2 \Psi _\infty ^2  \ind _ {\{ \| u \|_2 \Psi _\infty >\eta \sqrt k \}} = 0 \mbox{ for } n \mbox{ large enough}.
\]
This shows~\eqref{eqn:vdvw_1}. To prove~\eqref{eqn:vdvw_2}, pick $s,t\in [1/2,3/2]$ and write, for $n$ large enough,
\begin{align*}
P [(f_{n,t} - f_{n,s} ) ^2] &\leq 2 f_X(x) \| u \|_2^2 \Psi _\infty ^2 \times \frac{n}{k} \int_{\mathbb{R}^p} \big( \ind_{B(x,(s\lor t)^{1/p}\tau_{n,k}(x))}(z) -  \\
    & \qquad \qquad \qquad \qquad \qquad \qquad \ind_{B(x,(s\wedge t)^{1/p}\tau_{n,k}(x))} (z) \big)\diff z \\
    &= 2 \| u \|_2^2 \Psi _\infty ^2  |t-s|.
\end{align*}
This proves~\eqref{eqn:vdvw_2}. Convergence~\eqref{eqn:vdvw_3} of the sequence of bracketing integrals is easily obtained by following the proof of Lemma~2 in \cite{portier2021nearest}: brackets for $\sqrt{ \frac{n}{k} }  \ind_{B(x,t^{1/p}\tau_{n,k}(x))}$, $t\in [1/2,3/2]$, can be constructed as 
\[
\left[ \sqrt{ \frac{n}{k} }  \ind_{B(x,t_j^{1/p}\tau_{n,k}(x))}, \sqrt{ \frac{n}{k} }   \ind_{B(x,t_{j+1}^{1/p}\tau_{n,k}(x))} \right] 
\]
where the $t_j$ make up an $\varepsilon$-spaced set of increasing points in $[1/2,3/2]$. This concludes the proof of (i).
\vskip1ex
\noindent
(ii) Fix again $u\in \mathbb{R}^q \setminus \{ 0 \}$ and pick $s,t\in [1/2,3/2]$. By the Cram\'er-Wold device, and according to Theorem 2.11.23 p.221 in~\cite{wellner1996}, weak convergence to a (tight) Gaussian process will be guaranteed if we can show that, with the notation of (i), $P ( f_{n,t} f_{n,s} ) - P ( f_{n,t} ) P ( f_{n,s} )$ converges to $(s\wedge t) u^T \Sigma^2(s \wedge t,x) u$. This Gaussian process will then necessarily be centered because $Z_n(\tau)$ is so, and each of its univariate projections will be continuous on $[1/2,3/2]$ by Theorem~1.5.7 p.37 in~\cite{wellner1996} and its Addendum 1.5.8. 
\vskip1ex
\noindent
First of all, $|P ( f_{n,t} )| \leq P(F_n)$ and, for $n$ large enough, 
\[
P(F_n) \leq 2 f_X(x) \| u \|_2 \Psi_\infty \times \sqrt{\frac{n}{k}} \int_{B(x,(3/2)^{1/p}\tau_{n,k}(x) )} \diff z = O\left( \sqrt{\frac{k}{n}} \right) \to 0. 
\]
It then suffices to prove that $P ( f_{n,t} f_{n,s} )\to (s\wedge t) u^T \Sigma^2(s \wedge t,x) u$. Now
\[
P ( f_{n,t} f_{n,s} ) = \frac{n}{k} u^T \mathbb{E} [ \Psi_n (Y,X) \Psi_n (Y,X)^T \ind_{B(x,(s\wedge t)^{1/p}\tau_{n,k}(x) )} (X) ] u.
\]
A change of variables gives
\[
P ( f_{n,t} f_{n,s} ) = \frac{n}{k} (s\wedge t) \tau_{n,k}^p(x) \times u^T \left( \int_{B(0,1)} ( f_X \Sigma_n^2 )( x + \tau_{n,k}(x) (s\wedge t)^{1/p} v ) \diff v \right) u.
\]
Finally 
\begin{align*}
&\left| u^T \int_{B(0,1)} (f_X( x + \tau_{n,k}(x) (s\wedge t)^{1/p} v ) - f_X(x)) \Sigma_n^2( x + \tau_{n,k}(x) (s\wedge t)^{1/p} v ) \diff v \, u \right| \\
&\leq \sup_{v\in B(0,1)} |f_X( x + \tau_{n,k}(x) (s\wedge t)^{1/p} v ) - f_X(x)|  \\ & \times u^T \left( \int_{B(0,1)} \Sigma_n^2( x + \tau_{n,k}(x) (s\wedge t)^{1/p} v ) \diff v \right) u \\
&\to 0
\end{align*}
and then
$
P ( f_{n,t} f_{n,s} ) \to (s\wedge t) u^T \Sigma^2(s \wedge t,x) u.
$
The proof is complete.
\end{proof}

We can now combine Lemmas~\ref{lem:portier} and~\ref{lem:tightness} to write a proof of Theorem~\ref{theo:tightness}.
\begin{proof}[Proof of Theorem~\ref{theo:tightness}] Write
\[ 
Z_n ( \widehat{\tau}_{n,k}(x) ) =  Z_n( t_n^{1/p} \tau_{n,k}(x) ) = Z_n( \tau_{n,k}(x) )  + \{ Z_n( t_n^{1/p} \tau_{n,k}(x) ) - Z_n( \tau_{n,k}(x) ) \}
\]
with $t_n = (\widehat{\tau}_{n,k}(x) / \tau_{n,k}(x))^p$. 
\vskip1ex
\noindent
(i) By Lemma~\ref{lem:portier}, $t_n\stackrel{\mathbb{P}}{\longrightarrow} 1$, so that by tightness of the stochastic process $\{ Z_n (t ^{1/p}\tau_{n,k}(x)) \}_{ t \in [1/2,3/2] }$ in $\ell^\infty ([1/2,3/2])$ following from Lemma~\ref{lem:tightness}(i), one has $Z_n ( \widehat{\tau}_{n,k}(x) ) = O_{\mathbb{P}}(1)$. 
\vskip1ex
\noindent
(ii) By Lemma~\ref{lem:tightness}(ii), $Z_n( \tau_{n,k}(x) )$ converges weakly to a Gaussian distribution with mean $0$ and covariance matrix $ \Sigma^2(1,x)$, and the stochastic process $\{ Z_n (t ^{1/p}\tau_{n,k}(x)) \}_{ t \in [1/2,3/2] }$ converges weakly to a continuous Gaussian process in $\ell^\infty ([1/2,3/2])$. 
It follows that $| Z_n( t_n^{1/p} \tau_{n,k}(x) ) - Z_n( \tau_{n,k}(x) )| \to 0$ in probability, so that the desired weak convergence property of $Z_n( \widehat{\tau}_{n,k}(x) )$ holds.
\end{proof}

Besides being the crucial tool in the asymptotic analysis of the nearest-neighbor local logistic log-likelihood, Lemma~\ref{lem:tightness} and Theorem~\ref{theo:tightness} make it possible to obtain laws of large numbers for certain weighted averages of nearest neighbors and nearest-neighbor weighted empirical Gram matrices. These laws of large numbers will be used several times subsequently.

\begin{proposition}[Laws of large numbers for local linear nearest-neighbor estimators]
\label{prop:mean_var_score_nn}
Suppose that \ref{cond:new1} is fulfilled. Let $\varphi:\mathbb{R}^p\to\mathbb{R}$ be measurable, and continuous at the point $x$. If $k:=k_n \to \infty$ is such that $k/n \to 0$, then 
\[
\frac{1}{k} \sum_{i \in N_k(x)} \varphi(X_i) \begin{pmatrix} 1 \\ \frac{X_i - x}{\tau_{n,k}(x)} \end{pmatrix} \begin{pmatrix} 1 \\ \frac{X_i - x}{\tau_{n,k}(x)} \end{pmatrix}^T \stackrel{\mathbb{P}}{\longrightarrow} \varphi(x) \begin{pmatrix} 1 & 0_p^T \\ 0_p & \frac{1}{p+2} I_p \end{pmatrix}.
\]
Moreover, for any symmetric matrix $M\in \mathbb{R}^{p\times p}$, 
\[
\frac{1}{k} \sum_{i \in N_k(x)} \left( \frac{X_i - x}{\tau_{n,k}(x)} \right)^T M \left( \frac{X_i - x}{\tau_{n,k}(x)} \right) \begin{pmatrix} 1 \\ \frac{X_i - x}{\tau_{n,k}(x)} \end{pmatrix} \stackrel{\mathbb{P}}{\longrightarrow} \frac{\operatorname{tr}(M)}{p+2} \begin{pmatrix} 1 \\ 0_p \end{pmatrix}.
\]
\end{proposition}
\begin{proof}[Proof of Proposition~\ref{prop:mean_var_score_nn}] Identify in this proof the vector space of square matrices $(p+1)\times (p+1)$ having real coefficients with $\mathbb{R}^{(p+1)^2}$ equipped with its standard Euclidean norm $\| \cdot \|_2$ and consider the stochastic process 
\begin{align*}
Z_n(\tau) = \frac{1}{\sqrt{k}} \sum_{i=1}^n \left\{\Psi_n (Y_i,X_i) \ind_{B(x,\tau )}(X_i) - \mathbb E [\Psi_n (Y,X) \ind_{B(x,\tau )} (X) ]  \right\}
\end{align*}
with $E=\{0\}$, $Y_i=0$ for all $i$ and
\[
\Psi_n(y,z) \equiv \Psi_n(z) = \varphi(z) \begin{pmatrix} 1 \\ \frac{z - x}{\tau_{n,k}(x)} \end{pmatrix} \begin{pmatrix} 1 \\ \frac{z - x}{\tau_{n,k}(x)} \end{pmatrix}^T.
\]
Let $|\varphi|_{\infty}$ be a finite upper bound for $|\varphi|$ in a sufficiently small neighborhood $U$ of $x$, and let $n_0$ be such that for $n\geq n_0$, $B(x,(3/2)^{1/p}\tau_{n,k}(x)) \subset U$. Then 
\[
\sup_{n\geq n_0} \sup_{z\in B(x,(3/2)^{1/p}\tau_{n,k}(x))} \| \Psi_n(z) \|_2 \leq |\varphi|_{\infty} 
\left( 1 + \left( \frac{3}{2} \right)^{2/p} \right).
\]
Conclude, by Lemma~\ref{lem:tightness}(i), that the stochastic process $\{ Z_n (t^{1/p}\tau_{n,k}(x)) \}_{ t \in [1/2,3/2] }$ is tight in $\ell ^\infty ([1/2,3/2])$. It follows that, first of all, 
\begin{align*}
    &\sup_{t \in [1/2,3/2]} \left\| \frac{1}{k} \sum_{i=1}^n \varphi(X_i) \begin{pmatrix} 1 \\ \frac{X_i - x}{\tau_{n,k}(x)} \end{pmatrix} \begin{pmatrix} 1 \\ \frac{X_i - x}{\tau_{n,k}(x)} \end{pmatrix}^T \ind_{B(x,t^{1/p}\tau_{n,k}(x) )}(X_i) \right. \\
    &\left. \qquad \qquad \qquad - \frac{n}{k} \mathbb E \left[ \varphi(X) \begin{pmatrix} 1 \\ \frac{X - x}{\tau_{n,k}(x)} \end{pmatrix} \begin{pmatrix} 1 \\ \frac{X - x}{\tau_{n,k}(x)} \end{pmatrix}^T \ind_{B(x,t^{1/p}\tau_{n,k}(x) )} (X) \right] \right\|_2 \stackrel{\mathbb{P}}{\longrightarrow} 0.
\end{align*}
Obviously $\int_{B(0,1)} \diff z = V_p$ and $\int_{B(0,1)} z \diff z = 0$; moreover $\int_{B(0,1)} z_i z_j \diff z = 0$ when $i\neq j$, so that by rotational symmetry and a change to polar coordinates, 
\[
\int_{B(0,1)} z z^T \diff z = \left( \int_{B(0,1)} z_1^2 \diff z \right) I_p = \frac{1}{p} \left( \int_{B(0,1)} \| z \|_2^2 \diff z \right) I_p = \frac{V_p}{p+2} I_p.
\] 
Using the continuity of $\varphi$ and $f_X$ at $x$ and a linear change of variables, one finds 
\begin{align*}
    &\sup_{t \in [1/2,3/2]} \left\| \frac{n}{k} \mathbb E \left[ \varphi(X) \begin{pmatrix} 1 \\ \frac{X - x}{\tau_{n,k}(x)} \end{pmatrix} \begin{pmatrix} 1 \\ \frac{X - x}{\tau_{n,k}(x)} \end{pmatrix}^T \ind_{A} (X) \right] - t \varphi(x) \begin{pmatrix} 1 & 0_p^T \\ 0_p & \frac{t^{2/p}}{p+2} I_p \end{pmatrix} \right\|_2 \\
    &\leq \frac{3}{2} \times \frac{n}{k} \tau_{n,k}^p(x) \sup_{t \in [1/2,3/2]} \left\| \int_{B(0,1)} ( (\varphi f_X) (x +  t^{1/p} \tau_{n,k}(x) z ) - (\varphi f_X) (x) ) \right.\\
&\qquad \qquad \qquad \qquad \qquad \qquad \qquad \qquad \qquad \qquad \quad      \left.
    \begin{pmatrix} 1 & t^{1/p} z^T \\ t^{1/p} z & t^{2/p} z z^T \end{pmatrix} \diff z \right\|_2   \to 0,
\end{align*}
where $ A $ in the first above line stands for the ball $ B (x,t^{1/p}\tau_{n,k}(x) )$. Taking $t=t_n = (\widehat{\tau}_{n,k}(x) / \tau_{n,k}(x))^p$, which converges to 1 in probability, then yields 
\begin{align*}
    &\frac{1}{k} \sum_{i \in N_k(x)} \varphi(X_i) \begin{pmatrix} 1 \\ \frac{X_i - x}{\tau_{n,k}(x)} \end{pmatrix} \begin{pmatrix} 1 \\ \frac{X_i - x}{\tau_{n,k}(x)} \end{pmatrix}^T \\
    &= \frac{1}{k} \sum_{i=1}^n \varphi(X_i) \begin{pmatrix} 1 \\ \frac{X_i - x}{\tau_{n,k}(x)} \end{pmatrix} \begin{pmatrix} 1 \\ \frac{X_i - x}{\tau_{n,k}(x)} \end{pmatrix}^T \ind_{B(x,\widehat{\tau}_{n,k}(x) )}(X_i) \\
    &= \varphi(x) \begin{pmatrix} 1 & 0_p^T \\ 0_p & \frac{1}{p+2} I_p \end{pmatrix} + o_{\mathbb{P}}(1).
\end{align*}
This proves the first convergence. To show the second one, note it is an immediate consequence of the first convergence that 
\[
\frac{1}{k} \sum_{i \in N_k(x)} \left( \frac{X_i - x}{\tau_{n,k}(x)} \right)^T M \left( \frac{X_i - x}{\tau_{n,k}(x)} \right) \stackrel{\mathbb{P}}{\longrightarrow} \frac{\operatorname{tr}(M)}{p+2}. 
\]
Finally, to show that 
\[
\frac{1}{k} \sum_{i \in N_k(x)} \left\{ \left( \frac{X_i - x}{\tau_{n,k}(x)} \right)^T M \left( \frac{X_i - x}{\tau_{n,k}(x)} \right) \right\} \frac{X_i - x}{\tau_{n,k}(x)} \stackrel{\mathbb{P}}{\longrightarrow} 0_p,  
\]
repeat the proof of the first statement with the function 
\[
\Psi_n(y,z) \equiv \Psi_n(z) = \left\{ \left( \frac{z - x}{\tau_{n,k}(x)} \right)^T M \left( \frac{z - x}{\tau_{n,k}(x)} \right) \right\} \frac{z - x}{\tau_{n,k}(x)}
\]
and note that for any $i,j,k\in \{1,\ldots,p\}$, $\int_{B(0,1)} z_i z_j z_k \diff z = 0$.
\end{proof}

\subsection{Convergence properties of the nearest-neighbor local logistic log-likelihood}
\label{sec:proofs:loglik}

It follows from~\eqref{eqn:penNNlogit} that proving Theorem~\ref{theo:main} is equivalent to obtaining the convergence of 
\begin{align*}
    &( \sqrt{k}(\widehat{a}_n(x)-\ell(x)), \tau_{n,k}(x)\sqrt{k}(\widehat{b}_n(x)-\nabla\ell(x)) ) \\
    & \qquad \qquad = \argmax_{(a,b)\in \mathbb{R}\times \mathbb{R}^p} \left\{ L_n\left( \ell(x) + \frac{a}{\sqrt{k}}, \nabla\ell(x) + \frac{b}{\tau_{n,k}(x)\sqrt{k}} \right) \right.\\
    &\qquad \qquad\qquad \qquad\qquad \qquad\qquad \qquad\left.
    - \lambda \left\| \nabla\ell(x) + \frac{b}{\tau_{n,k}(x)\sqrt{k}} \right\| _1 \right\}.
\end{align*}
The objective function is concave, so Lemma~\ref{lem:pollard} suggests that it is enough to consider its convergence properties in order to recover the convergence of its minimizer. The quantity $L_n(a,b)$ is a log-likelihood, so we analyze the convergence of the first term in the above objective function using a Taylor expansion of order 2. This requires obtaining the asymptotic behavior of the corresponding score function at $(\ell(x),\nabla\ell(x))$, that is
\[
S_n(x)=\frac{1}{\sqrt{k}} \sum_{i \in N_k(x) } ( Y_i - \expit(\ell (x) + \nabla \ell(x)^T (X_i-x)) ) \begin{pmatrix} 1 \\ \frac{X_i - x}{\tau_{n,k}(x)} \end{pmatrix},
\]
and of its Hessian matrix at $(\ell(x),\nabla\ell(x))$, namely 
\[
H_n(x)=-\frac{1}{k}\sum_{i\in N_k(x) } \expit'(\ell (x) + \nabla \ell(x)^T (X_i-x)) \begin{pmatrix} 1 \\ \frac{X_i - x}{\tau_{n,k}(x)} \end{pmatrix} \begin{pmatrix} 1 \\ \frac{X_i - x}{\tau_{n,k}(x)} \end{pmatrix}^T
\]
where $\expit':s\mapsto\expit(s)(1-\expit(s))$ denotes the derivative of $\expit$. Recall the notation 
\[
\Gamma(x) =  \pi(x)(1-\pi(x)) \begin{pmatrix} 1 & 0_p^T \\ 0_p & \frac{1}{p+2} I_p \end{pmatrix}.
\]
Write $\Delta \pi(x)$ for the Laplacian of $\pi$ at $x$, \textit{i.e.}~the trace of its Hessian matrix.
\begin{lemma}[Convergence of the score function and Hessian of the local logistic log-likelihood]
\label{lem:score_hessian}
Suppose that \ref{cond:new1} and \ref{cond:new2} are fulfilled. If $k:=k_n  \to \infty $ is such that $k/n \to 0$, it holds that $S_n(x)=W_n(x) + T_n(x)$, where
\begin{align*}
    W_n(x) &= \frac{1}{\sqrt{k}} \sum_{i \in N_k(x) } ( Y_i - \pi(X_i) ) \begin{pmatrix} 1 \\ \frac{X_i - x}{\tau_{n,k}(x)} \end{pmatrix} \stackrel{\mathrm{d}}{\longrightarrow} \mathcal N (0,\Gamma(x)) 
\end{align*}
and
\begin{align*} 
     T_n(x) &= \tau_{n,k}^2(x) \sqrt{k} \left( \frac{1}{2(p+2)} \left( \Delta \pi(x) - \gamma(x) \| \nabla \pi(x) \|_2^2 \right) \begin{pmatrix} 1 \\ 0_p \end{pmatrix} + o_{\mathbb{P}}(1) \right)
\end{align*}
\cbstart with $\gamma (x) = \frac{1-2\pi(x)}{\pi(x)(1-\pi(x))}$. Moreover, \cbend
$
H_n(x) \stackrel{\mathbb{P}}{\longrightarrow} - \Gamma(x).
$
\end{lemma}  
The control of $T_n(x)$ uses the following consequence of the Taylor formula with integral form of the remainder: let, for any function $F:\mathbb{R}^q\to \mathbb R$ that is twice continuously differentiable at a point $z$, $H_F(z')$ denote its Hessian matrix at the point $z'$ when $z'$ is close enough to $z$. Then 
\begin{align}
\lim_{\eta\to 0} \sup_{\substack{z'\in B(z,\eta) \\ z'\neq z}} \frac{1}{\| z'-z \|_2^2} &\left| F( z' ) - F( z ) - ( z'-z ) ^T \nabla F ( z ) \right.\nonumber \\
 &\qquad \quad \left.- \frac{1}{2} ( z'-z ) ^T H_F( z ) ( z'-z ) \right| = 0.\label{eqn:taylor}
\end{align}
More precisely, for $z'\neq z$ but close enough to $z$, 
\begin{align}
\nonumber 
    &\frac{1}{\| z'-z \|_2^2} \left| F( z' ) - F( z ) - ( z'-z ) ^T \nabla F ( z )  - \frac{1}{2} ( z'-z ) ^T H_F( z ) ( z'-z ) \right|\\
\label{eqn:taylor_bound}
    &\leq \frac{1}{2} \| H_F( z' ) -  H_F( z )\|_2
\end{align}
where in the upper bound $\| \cdot \|_2$ is the operator norm induced by the Euclidean norm on $\mathbb{R}^q$.
\begin{proof}[Proof of Lemma~\ref{lem:score_hessian}] 
Since $\expit'=\expit(1-\expit)$, the convergence in probability of $H_n(x)$ to $-\Gamma(x)$ is an obvious consequence of Proposition~\ref{prop:mean_var_score_nn}. We concentrate on the convergence of $S_n(x)$. Write
\begin{align*}
S_n(x)  &= \frac{1}{\sqrt{k}} \sum_{i \in N_k(x) } ( Y_i - \pi(X_i) ) \begin{pmatrix} 1 \\ \frac{X_i - x}{\tau_{n,k}(x)} \end{pmatrix} \\
    &- \frac{1}{\sqrt{k}} \sum_{i \in N_k(x) } ( \expit(\ell (x) + \nabla \ell(x)^T (X_i-x)) - \pi(X_i) ) \begin{pmatrix} 1 \\ \frac{X_i - x}{\tau_{n,k}(x)} \end{pmatrix} \\
    &=: W_n(x) + T_n(x).
\end{align*}
We study first the quantity $W_n(x)$ and we then examine $T_n(x)$. 
\vskip1ex
\noindent
\textbf{Convergence of $W_n(x)$:} Consider the stochastic process 
\begin{align*}
Z_n(\tau) = \frac{1}{\sqrt{k}} \sum_{i=1}^n \left\{\Psi_n (Y_i,X_i) \ind_{B(x,\tau )}(X_i) - \mathbb E [\Psi_n (Y,X) \ind_{B(x,\tau )} (X) ]  \right\}
\end{align*}
with $\Psi_n:E\times \mathbb R^p=\{0,1\}\times \mathbb R^p \to \mathbb R^{p+1}$ defined by
\[
\Psi_n(y,z)= ( y - \pi(z) ) \begin{pmatrix} 1 \\ \frac{z - x}{\tau_{n,k}(x)} \end{pmatrix}.  
\]
Then 
\[
\sup_{n\geq 1} \sup_{z\in B(x,(3/2)^{1/p}\tau_{n,k}(x))} \max(\| \Psi_n(0,z) \|_2, \| \Psi_n(1,z) \|_2 ) \leq \sqrt{1 + \left( \frac{3}{2} \right)^{2/p}} <\infty.
\]
Obviously $\mathbb E [\Psi_n (Y,X) \ind_{B(x,\tau )} (X) ]=0$ since $\mathbb{E}[Y-\pi(X) \, | \, X]=0$, and 
\begin{align*}
\Sigma_n^2(X) &:= 
\mathbb{E}\left[ \left. ( Y - \pi(X) )^2 \begin{pmatrix}
    1\\
    \frac{X - x}{\tau_{n,k}(x)}
\end{pmatrix} \begin{pmatrix}
    1\\
    \frac{X - x}{\tau_{n,k}(x)}
\end{pmatrix}^T \, \right| \, X \right] \\
    &=\pi(X) (1-\pi(X)) \begin{pmatrix}
    1\\
    \frac{X - x}{\tau_{n,k}(x)}
\end{pmatrix} \begin{pmatrix}
    1\\
    \frac{X - x}{\tau_{n,k}(x)}
\end{pmatrix}^T.
\end{align*}
Recall also from the proof of Proposition~\ref{prop:mean_var_score_nn} that
\[
\int_{B(0,1)} z z^T \diff z = \frac{V_p}{p+2} I_p.
\]
Assumption \ref{cond:new2} then yields, for any $t>0$, 
\[
\int_{B(0,1)} \Sigma_n^2( x + \tau_{n,k}(x) t^{1/p} v ) \diff v \to V_p \pi(x)(1-\pi(x)) \begin{pmatrix} 1 & 0_p^T \\ 0_p & \frac{t^{2/p}}{p+2} I_p \end{pmatrix}.
\]
For $t=1$ the right-hand side is exactly $V_p\Gamma(x)$. Conclude, using Theorem~\ref{theo:tightness}(ii), that $Z_n ( \widehat{\tau}_{n,k}(x) ) \stackrel{\mathrm{d}}{\longrightarrow} \mathcal{N}(0, \Gamma(x))$, that is, 
\begin{align*}
    &\frac{1}{\sqrt{k}} \sum_{i \in N_k(x) } ( Y_i - \pi(X_i) ) \begin{pmatrix}
    1\\
    \frac{X_i - x}{\tau_{n,k}(x)}
\end{pmatrix} \\
    &= \frac{1}{\sqrt{k}} \sum_{i=1}^n ( Y_i - \pi(X_i) ) \begin{pmatrix}
    1\\
    \frac{X_i - x}{\tau_{n,k}(x)} 
\end{pmatrix} \ind_{B(x,\widehat{\tau}_{n,k}(x) )}(X_i) \\
    &\stackrel{\mathrm{d}}{\longrightarrow} \mathcal{N}(0,\Gamma(x))
\end{align*}
as announced.
\vskip1ex
\noindent
\textbf{Convergence of $T_n(x)$:} Fix $\varepsilon>0$. Write, for any $z$ close enough to $x$, 
\begin{align*}
    &\expit(\ell (x) + \nabla \ell(x)^T (z-x)) - \pi(z) \\
    &= \expit(\ell (x) + \nabla \ell(x)^T (z-x)) - \pi(x) - \nabla \pi(x)^T(z-x) \\
    &- (\pi(z) - \pi(x) - \nabla \pi(x)^T(z-x)).
\end{align*}
Recall that $\pi\mapsto \logit(\pi)$ has derivative $1/(\pi(1-\pi))$ on $(0,1)$. As a consequence, $\expit(\ell (x)) (1-\expit(\ell (x))) \nabla \ell (x) = \nabla \pi(x)$. Moreover, $s\mapsto\expit(s)$ has second derivative $s\mapsto \expit(s)(1-\expit(s))(1-2\expit(s))$. Applying~\eqref{eqn:taylor}, first to the function $\expit$ and then to the function $\pi$, which is twice continuously differentiable at $x$ by \ref{cond:new2}, leads to the existence of $\eta>0$ such that for all $z\in B(x,\eta)$, 
\begin{align*}
    &\bigg| \expit(\ell (x) + \nabla \ell(x)^T (z-x)) - \pi(x) - \nabla \pi(x)^T(z-x) \\
     & 
     \qquad \qquad \qquad  \qquad \qquad  \qquad \qquad  \qquad - \frac{1}{2} \gamma(x) (\nabla \pi(x)^T(z-x))^2 \bigg| 
    \leq \frac{\varepsilon}{2} \|z-x\|_2^2
\end{align*}
and 
\[
\left| \pi(z) - \pi(x) - \nabla \pi(x)^T(z-x) - \frac{1}{2} (z-x)^T H_{\pi}(x) (z-x) \right| \leq \frac{\varepsilon}{2} \|z-x\|_2^2.
\]
Since by Lemma~\ref{lem:portier} we have $\widehat{\tau}_{n,k}(x)/\tau_{n,k}(x)\to 1$ in probability and $\tau_{n,k}(x)\to 0$, one may conclude that 
\begin{align*}
&T_n(x) =  o_{\mathbb{P}}(\tau_{n,k}^2(x) \sqrt{k})+ \\
& \frac{1}{2\sqrt{k}} \sum_{i \in N_k(x) } \left( (X_i-x)^T H_{\pi}(x) (X_i-x) - \gamma (x) (\nabla \pi(x)^T(X_i-x))^2 \right) \begin{pmatrix} 1 \\ \frac{X_i - x}{\tau_{n,k}(x)} \end{pmatrix} .
\end{align*}
Now 
\begin{align*}
    &\frac{1}{\sqrt{k}} \sum_{i \in N_k(x) } \left( (X_i-x)^T H_{\pi}(x) (X_i-x) - \gamma(x) (\nabla \pi(x)^T(X_i-x))^2 \right) \begin{pmatrix} 1 \\ \frac{X_i - x}{\tau_{n,k}(x)} \end{pmatrix} \\
    &=\tau_{n,k}^2(x) \sqrt{k} \left( \frac{1}{k} \sum_{i \in N_k(x) } \left( \frac{X_i - x}{\tau_{n,k}(x)} \right)^T H_{\pi}(x) \left( \frac{X_i - x}{\tau_{n,k}(x)} \right) \begin{pmatrix} 1 \\ \frac{X_i - x}{\tau_{n,k}(x)} \end{pmatrix} \right. \\
    &\left.  \qquad- \frac{1}{k} \sum_{i \in N_k(x) } \gamma(x)  
    \left( \frac{X_i - x}{\tau_{n,k}(x)} \right)^T \nabla \pi(x) \nabla \pi(x)^T \left( \frac{X_i - x}{\tau_{n,k}(x)} \right) \begin{pmatrix} 1 \\ \frac{X_i - x}{\tau_{n,k}(x)} \end{pmatrix} \right) 
\end{align*}
and the second convergence of Proposition~\ref{prop:mean_var_score_nn} applies. The proof is complete.
\end{proof}

\subsection{Proof of Theorem~\ref{theo:main}}
\label{sec:proofs:theo}


{Note that $k/n\to 0$ because $\tau_{n,k}^2(x) \sqrt{k}$ is bounded.} Let the sequence of rescaling matrices $D_n$ be defined as
\[
D_n = \frac{1}{\sqrt{k}} \begin{pmatrix}
    1 & 0_p^T\\ 
    0_p & \frac{1}{\tau_{n,k}(x)} I_p 
\end{pmatrix}
\]
so that 
\[
\forall (a,b)\in \mathbb{R}\times \mathbb{R}^p, \ \frac{1}{\sqrt{k}}\begin{pmatrix}
    a\\
    \frac{b}{\tau_{n,k}(x)} 
\end{pmatrix} = D_n\begin{pmatrix}
    a\\
    b
\end{pmatrix}.
\]
We seek to apply Lemma~\ref{lem:pollard}(ii) with $A_n(u)=-A_{n,1}(u)+A_{n,2}(u)$, where, for given $u=(u_0,u_1,\ldots,u_p)\in \mathbb{R}^{p+1}$, we let
\begin{align*}
A_{n,1}(u) &= L_n \left( \begin{pmatrix} \ell(x) \\ \nabla\ell(x) \end{pmatrix} + D_n u\right)  -  L_n \begin{pmatrix} \ell(x) \\ \nabla\ell(x) \end{pmatrix}, \\ 
A_{n,2}(u) &= \lambda \left\{ \left\| \nabla\ell(x) + \frac{1}{\tau_{n,k}(x) \sqrt{k}} u_{(1:p)} \right\|_1 - \left\|    \nabla\ell(x)  \right\|_1\right\}, 
\end{align*}
with  $u_{(1:p)} = 
    (u_1, \ldots , u_p)^T$ %
and, if $S_n=S_n(x)$ denotes the score function of Lemma~\ref{lem:score_hessian}, 
\begin{align*}
B_n (u) &= - u^T S_n(x) + \frac 1 2 u ^T \Gamma(x) u \\
    &+ ( c f_X(x) V_p )^{1/p}\left(  \sum_{j=1}^p \mathrm {sgn}(\nabla \ell_j (x)) u_j \ind_{\{ \nabla \ell_j (x) \neq 0 \}} + |u_j| \ind_{\{ \nabla \ell_j (x) = 0 \}} \right).
\end{align*}
Since $\Gamma(x)$ is a positive definite matrix and $S_n(x)$ is bounded in probability by Lemma~\ref{lem:score_hessian}, the function $B_n$ satisfies the assumptions of Lemma~\ref{lem:pollard}. Noting that $L_n$ is concave, it follows that $A_n$ is convex, so that it is enough to show that $A_n-B_n$ converges pointwise to 0 in probability in order to apply Lemma~\ref{lem:pollard}(ii). Given that $H_n(x)$ converges to $-\Gamma(x)$ in probability by Lemma~\ref{lem:score_hessian} again, it is sufficient to prove that
\begin{align*}
R_n(u) &= A_{n,1}(u) - u^T S_n(x) - \frac 1 2 u^T H_n(x) u \stackrel{\mathbb{P}}{\longrightarrow} 0 
\end{align*}
and
\begin{align*}
    A_{n,2}(u) \stackrel{\mathbb{P}}{\longrightarrow} ( c f_X(x) V_p )^{1/p} \left\{ \sum_{j=1}^p \mathrm {sgn}(\nabla \ell_j (x)) u_j \ind_{\{ \nabla \ell_j (x) \neq 0 \}} + |u_j| \ind_{\{ \nabla \ell_j (x) = 0 \}} \right\}.
\end{align*}
The remainder term $R_n(u)$ is dealt with using the following lemma.

\begin{lemma}[Pointwise approximation of the nearest-neighbor local logistic log-likelihood]
\label{lem:Rn}
Suppose that \ref{cond:new1} and \ref{cond:new2} are fulfilled. If $k:=k_n  \to \infty $ is such that $k/n \to 0$, we have 
\[
\forall u\in \mathbb{R}^{p+1}, \ |R_n(u)| \leq \frac{\| u \|_2^3}{2\sqrt{2}} \times \frac{1}{\sqrt{k}} \left( \frac{ \widehat{\tau}_{n,k}(x) }{ \tau_{n,k}(x) } \vee 1 \right)^3,
\]
and in particular $|R_n(u)| = O_{\mathbb{P}}(1/\sqrt{k})$.
\end{lemma}
\begin{proof}[Proof of Lemma~\ref{lem:Rn}] From~\eqref{eqn:taylor_bound}, we find that if $F:\mathbb{R}^{p+1}\to \mathbb R$ is a twice continuously differentiable function having a Lipschitz continuous Hessian matrix $u\mapsto H_F(u)$, that is, there is $C>0$ with $\|H_F( v ) -  H_F( u )\|_2 \leq C\| v-u \|_2$ for any $u,v\in \mathbb{R}^{p+1}$, then 
\begin{equation}
\label{eqn:taylor_bound_lipschitz}
\left| F(v) - F(u) - (v-u) ^T \nabla F (u)  - \frac{1}{2} (v-u) ^T H_F( u ) (v-u) \right| \leq \frac{C}{2} \| v-u \|_2^3.
\end{equation}
We apply~\eqref{eqn:taylor_bound_lipschitz} with $F:u\mapsto L_n(z+D_n u)$ where $z = ( \ell(x), \nabla \ell(x)^T )^T$. We know that the gradient of $F$ at $0=0_{p+1}$ is $S_n(x)$ and its Hessian matrix at $0$ is $H_n(x)$; more generally 
\[
H_F(u)=-\frac{1}{k}\sum_{i\in N_k(x) } \expit'\left( (z + D_n u)^T \begin{pmatrix} 1 \\ X_i - x \end{pmatrix} \right) \begin{pmatrix} 1 \\ \frac{X_i - x}{\tau_{n,k}(x)} \end{pmatrix} \begin{pmatrix} 1 \\ \frac{X_i - x}{\tau_{n,k}(x)} \end{pmatrix}^T.
\]
Now $s\mapsto \expit''(s)=\expit(s)(1-\expit(s))(1-2\expit(s))$ is bounded by $1/4$, so by the mean value theorem,
\begin{align*}
    &\left| \expit'\left( (z + D_n u)^T \begin{pmatrix} 1 \\ X_i - x \end{pmatrix} \right) - \expit'\left( z^T \begin{pmatrix} 1 \\ X_i - x \end{pmatrix} \right) \right| \\
    &\leq \frac{1}{\sqrt{k}} \times \frac{1}{4} \| u \|_2 \left\| \begin{pmatrix} 1 \\ \frac{X_i - x}{\tau_{n,k}(x)} \end{pmatrix} \right\|_2. 
\end{align*}
Therefore 
\[
\| H_F(u) - H_F(0) \|_2 \leq \frac{1}{\sqrt{k}} \times \frac{1}{4} \| u \|_2 \times \frac{1}{k} \sum_{i\in N_k(x)} \left\| \begin{pmatrix} 1 \\ \frac{X_i - x}{\tau_{n,k}(x)} \end{pmatrix} \right\|_2^3.
\]
Inequality~\eqref{eqn:taylor_bound_lipschitz} applies and yields, for any $u\in \mathbb{R}^{p+1}$,
\begin{align*}
    | R_n(u) | &\leq \frac{1}{\sqrt{k}} \times \frac{1}{8} \| u \|_2^3 \times \frac{1}{k} \sum_{i\in N_k(x)} \left\| \begin{pmatrix} 1 \\ \frac{X_i - x}{\tau_{n,k}(x)} \end{pmatrix} \right\|_2^3 \\
    &\leq \frac{\| u \|_2^3}{2\sqrt{2}} \times \frac{1}{\sqrt{k}} \left( \frac{ \widehat{\tau}_{n,k}(x) }{ \tau_{n,k}(x) } \vee 1 \right)^3
\end{align*}
as required. The conclusion on the rate of convergence of $R_n(u)$ to 0 follows from Lemma~\ref{lem:portier}.
\end{proof}

To show the convergence of $A_{n,2}(u)$, note that for each $z\neq 0$ and $u\in \mathbb R$, there is $t>0$ small enough such that  $| z + tu| - | z|  = tu \, \mathrm{sgn}(z)   $; besides, if $z = 0$, then $| z+ tu| - | z| = t|u| $. Working componentwise then immediately yields the below lemma, from which the pointwise limit of $A_{n,2}$ follows. 
\begin{lemma}
\label{lem:lasso_pen}
For any $u,v\in \mathbb R^p$, there exists $t_0>0$ such that for all $t\in [0, t_0]$,
\[
\| v +  t u \|_1  -  \| v \|_1 = t \sum_{j=1}^p \mathrm {sgn}(v_j) u_j \ind_{\{ v_j \neq 0 \}} + |u_j| \ind_{\{ v_j = 0 \}}. 
\]
\end{lemma}

\subsection{Proof of Corollary \ref{corollary_matrix}}

 The proof follows from the decomposition
\begin{align*}
\|\widehat{M} - M \|_F & \leq  \left\| \frac{1}{m} \sum_{i=1} ^m \{ \widehat{b}_n (X_i^*)  \widehat{b}_n (X_i^*)^T  - \nabla \ell(X_i^*) \nabla \ell(X_i^*) ^T \}\right\|_F   \\
&+ \left\| \frac{1}{m} \sum_{i= 1} ^m  \nabla \ell(X_i^*) \nabla \ell(X_i^*) ^T - \int_ {\mathbb{R}^p} \nabla \ell(x) \nabla \ell(x) ^T \diff \mu(x) \right\|_F,   
\end{align*}
where $\|\cdot\|_F$ is the Frobenius norm. The first term in the above upper bound can be treated using Theorem~\ref{theo:main}:
\begin{align*}
    &\left\| \frac 1 m  \sum_{i= 1} ^m \{ \widehat{b}_n (X_i^*)  \widehat{b}_n (X_i^*)^T  - \nabla \ell(X_i^*) \nabla \ell(X_i^*) ^T \}\right\|_F \\
    &\leq \max_{x\in \mathrm{supp} (\mu) }\|  \widehat{b}_n (x)  \widehat{b}_n (x)^T  - \nabla \ell(x) \nabla \ell(x) ^T\|_F \\
    &= O_{\mathbb{P}} \left( \frac{1}{\tau_{n,k}(x) \sqrt{k}} \right)  +  O_{\mathbb{P}} (  \tau_{n,k}(x) ).
\end{align*}
The other term is controlled by computing its second moment:
\begin{align*}
    &\mathbb E \left( \left\| \frac 1 m \sum_{i= 1} ^m  \nabla \ell(X_i^*) \nabla \ell(X_i^*) ^T - \int_ {\mathbb{R}^p} \nabla \ell(x) \nabla \ell(x) ^T \diff \mu(x) \right\|_F^2 \right) \\
    &= \frac{1}{m} \mathbb E \left( \left\| \nabla \ell(X^*) \nabla \ell(X^*)^T - \int_ {\mathbb{R}^p} \nabla \ell(x) \nabla \ell(x) ^T \diff \mu(x) \right\|_F^2 \right) \\
    &\leq \frac{1}{m} \mathbb E \left( \left\| \nabla \ell(X^*) \nabla \ell(X^*)^T \right\|_F^2 \right) \\ 
    &= \frac 1 m \int_ {\mathbb{R}^p}  \left\|\nabla \ell(x) \right\|_2 ^4 \diff \mu(x) \leq \frac 1 m \max_{x\in \mathrm{supp} (\mu)}  \left\|\nabla \ell(x) \right\|_2 ^4.
\end{align*} 
The upper bound is finite by assumption. \qed

\section{Additional numerical results}
\label{sec:numerical_exp_suppl}

\subsection{Simulation study}

\cbstart 
We also conducted a simulation study to (1) evaluate the computing time of LLO$(\lambda=0)$ and LLO$(\lambda>0)$ in comparison with the competing local linear methods, (2) further investigate the impact of the number of nearest neighbors $k$ in both LLO$(\lambda=0)$ and LLO$(\lambda>0)$, as well as the influence of the bandwidth parameter $h$ for the competing approaches, namely LGSIM and WSVM, {and (3) ascertain the influence of finding the correct dimension of the central subspace.}

Figure~\ref{fig_compute_time} presents the average computation time for Examples 2 and 3 for gradient estimation with $p=8$ and increasing sample sizes ($n = 500$, $1000$, $2000$, $3000$, $4000$) across $1000$ simulation replicates. The computational cost increases with the sample size for all competing methods. Among the methods, the penalized LLO method consistently requires the least computation time, followed by the non-penalized version. In contrast, LGSIM exhibits a similar computation time when $n=500$ and the gap in computing time becomes more pronounced as $n$ increases, indicating that the proposed penalized $(\lambda > 0)$ and non-penalized $(\lambda = 0)$ approaches scale more efficiently for larger datasets. 
{As a conclusion, the proposed methods remain} computationally feasible even in moderate-to-large-sample settings, while providing a clear efficiency advantage over the LGSIM.



We next report results related to the tuning of the number of neighbors $k$ for LLO$(\lambda=0)$ and the bandwidth $h=c \cdot h_{Scott}$, where $h_{Scott}$ denotes Scott’s rule, for LGSIM by applying 5-fold cross-validation at a sample of randomly chosen local points $x_i$ of size $m=n/2$. We have also chosen $h$ for WSVM via cross-validation. The model is evaluated using the misclassification error, i.e., the empirical counterpart of the misclassification risk $\mathcal{R}(g)= \mathbb{P}(g(X)\neq Y)$. For tuning, we consider a grid with {$k \in \{1, 5, 10, \dots, 500\}$}, 
$c \in \{0.01, 0.13, 0.26, \ldots, 6\}$ and $h \in \{2.00,  2.27,  2.54, \ldots, 15\}$. The optimal parameter ($k$ or $c$ or $h$) is selected based on cross-validation performance. 
{We report the results for a single sample in Examples 2 and 3 in Figure~\ref{fig_lam} and Figure~\ref{fig_LGWS} to illustrate how this CV works. The results show that $k$, $c$ and $h$ play a significant role as they help to reduce the misclassification risk. For the penalized method LLO$(\lambda > 0$), we also investigated cross-validation-based choices of $k$ for Examples 2 and 3 (see the bottom row of Figure~\ref{fig_lam}). 
The selected values are 
closer} to the default $ (k = \lfloor \sqrt{n} \rfloor=33$ when $n=1000$), and the misclassification error remains stable over a range of $k$, indicating low sensitivity to this parameter. \cbend

Finally, similarly to Figure \ref{fig:distance_Ex4}, we represent in Figures~\ref{fig:distances_pvarying} and~\ref{fig:mcrisks} the distance to the central subspace and the misclassification risk in Examples 1, 2 and 3 when the correct dimension reduction subspace and sample size are fixed but the dimensions of the ambient space and estimated central subspace vary, \textit{i.e.} we fix $n=1000$, simulate a vector $X$ of independent centered and unit Gaussian random variables covariates having dimension $p\in \{8,16,32,64\}$, and the estimated central subspace has dimension $1\leq d\leq 6$. Again, the proposed LLO($\lambda>0$) method consistently yields superior results across all evaluated settings. \cbstart As $p$ increases, the performances of LLO($\lambda=0$), LGSIM and LLO($\lambda>0$) get closer, while the performances of the SAVE, POTD and WSVM methods tend to substantially deteriorate. \cbend Similarly to Figure~\ref{fig:analysis_Ex4}, we also provide in Figures~\ref{fig:dimension_select} and~\ref{fig:dim_comparison} results related to the dimension $d$ selected through the full workflow summarized in Algorithms~\ref{alg:Mestimate} and~\ref{alg:feature} and the misclassification risk of the hence obtained nearest-neighbor classifier in Examples 1, 2 and 3. 

\subsection{Real data analyses}

Figures~\ref{fig:Dim_selection_rf}--\ref{fig:comparison_data_roc_WDBC} contain extra results about the number of selected components by the proposed method with the random forest classifier and further elements about the predictive quality of the nearest neighbor and random forest classifiers when paired with one of the dimension reduction methods we consider.
\begin{figure}[htbp]
\includegraphics[width=0.45\linewidth, height=0.4\linewidth]{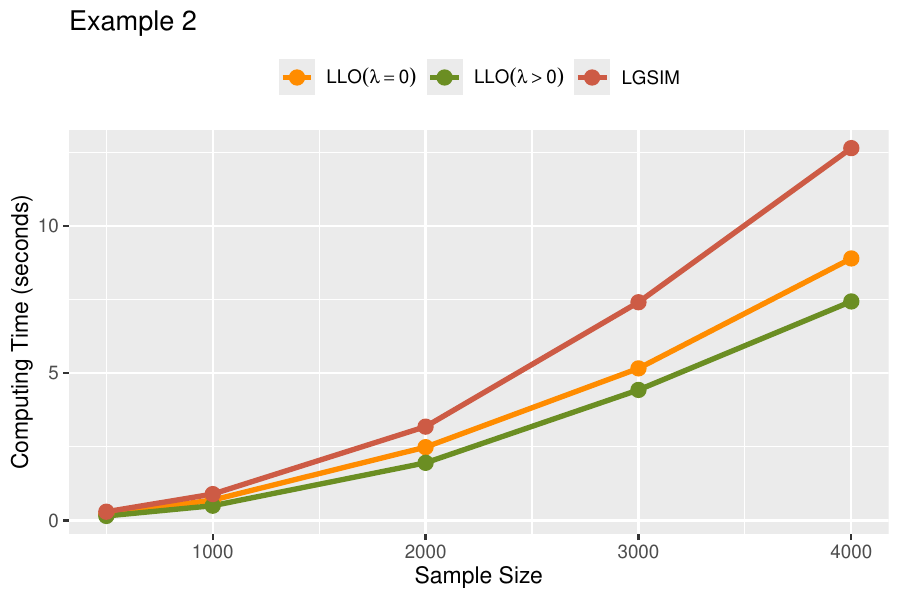}  
     \includegraphics[width=0.45\linewidth, height=0.4\linewidth]{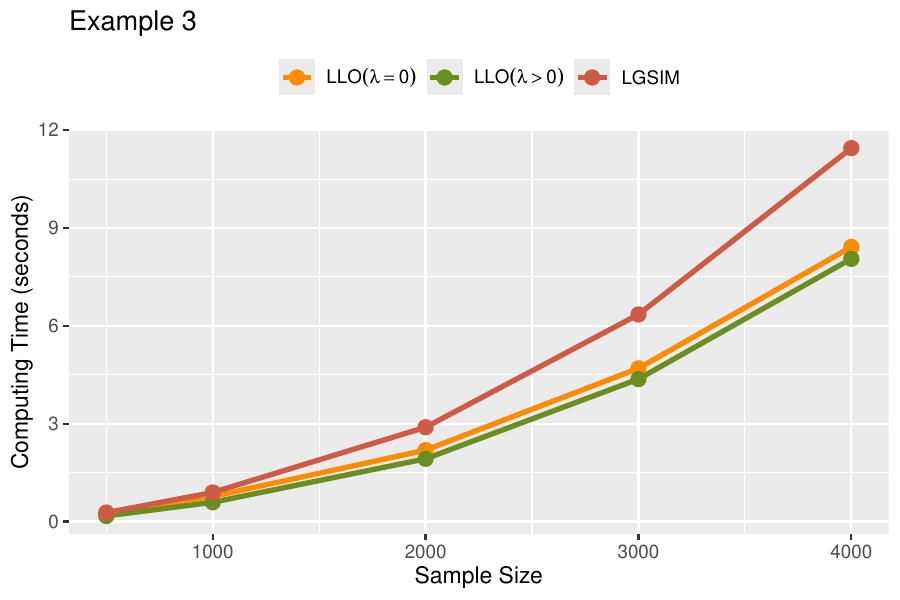}  

    \caption{ Average computing time for gradient estimation via LLO$(\lambda=0)$, LLO$(\lambda>0)$ and LGSIM for 1000 replicates when $p=8$ and sample size $n\in \{ 500, 1000, 2000, 3000, 4000 \}$.}\label{fig_compute_time}
\end{figure}
\begin{figure}[htbp]
\includegraphics[width=0.49\linewidth]{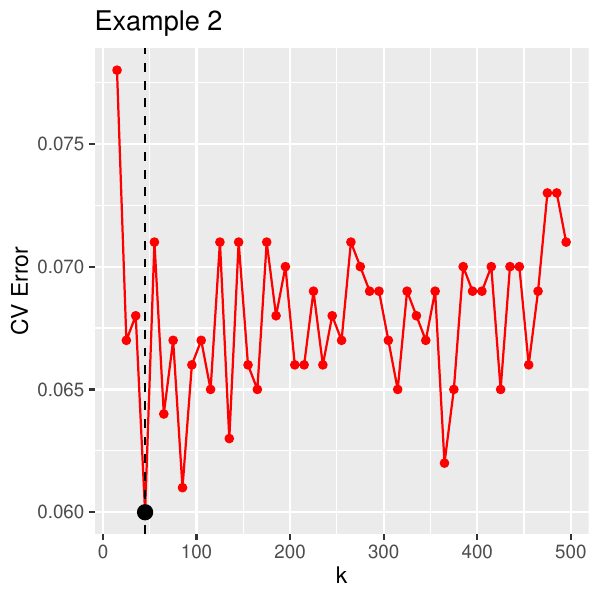}  
     \includegraphics[width=0.49\linewidth]{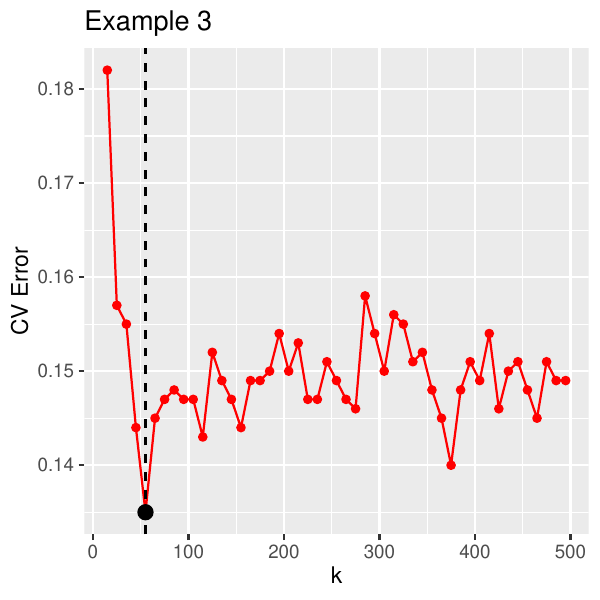}  
\includegraphics[width=0.49\linewidth]{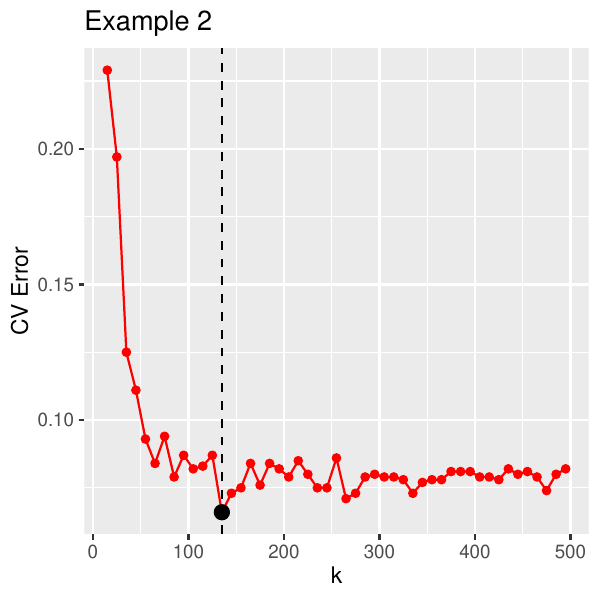}
\includegraphics[width=0.49\linewidth]{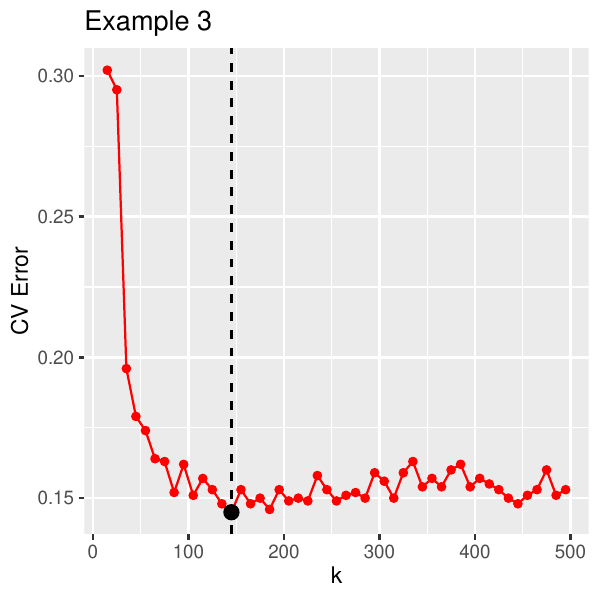}

    \caption{Selection of $k$ for LLO$(\lambda>0)$ (first row) and LLO$(\lambda=0)$ (second row) by CV for Examples 2 and 3 at a single realization with dimension $p=8$ and sample size $n=1000$. }\label{fig_lam}
\end{figure}
\begin{figure}[htbp]
\includegraphics[width=0.49\linewidth]{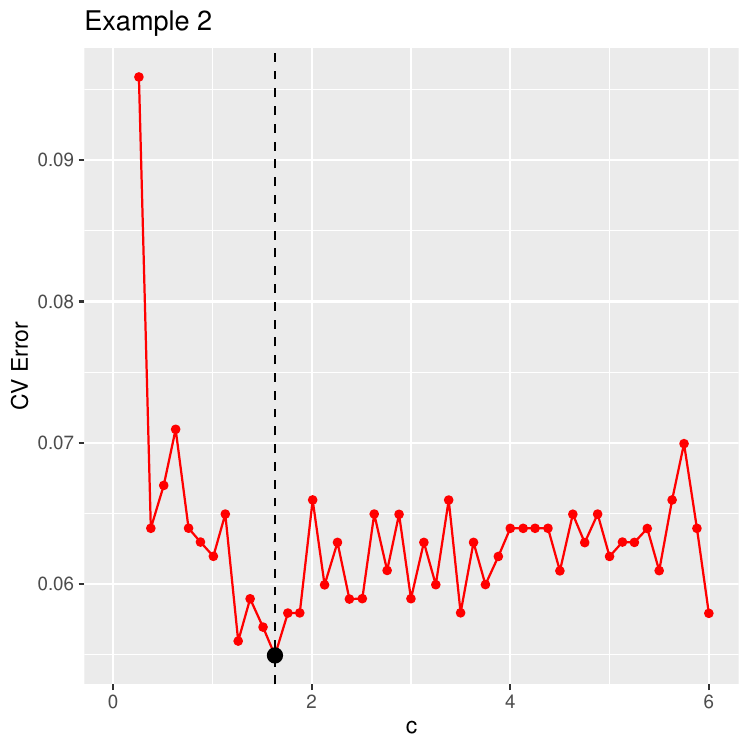}  
\includegraphics[width=0.49\linewidth]{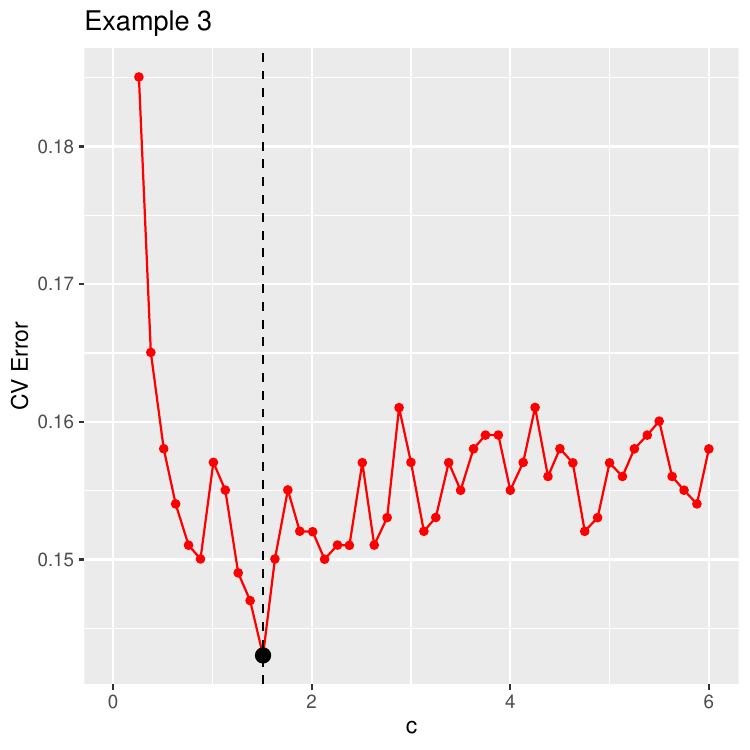} 
\includegraphics[width=0.49\linewidth]{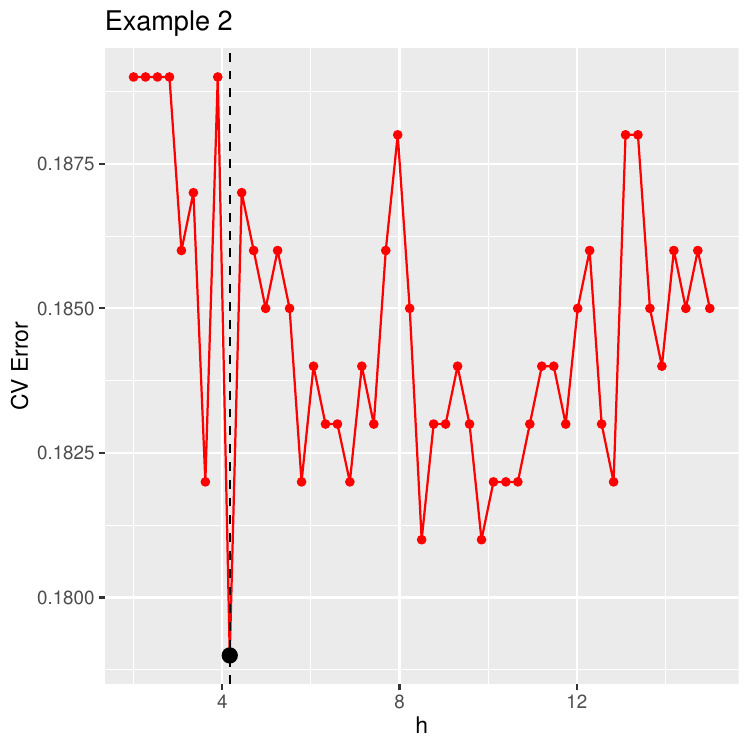}  
\includegraphics[width=0.49\linewidth]{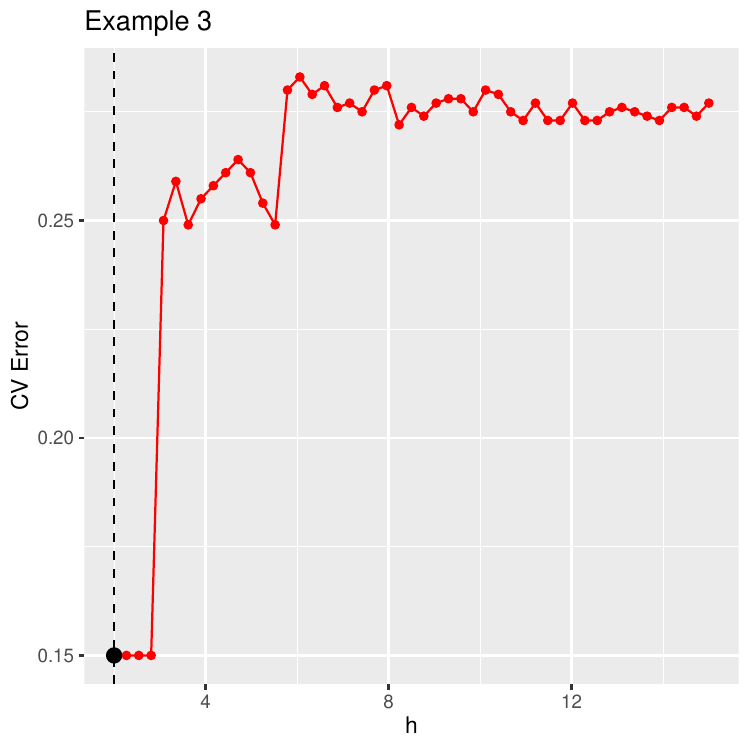} 

    \caption{Selection of $c$ for LGSIM (first row) and of $h$ for WSVM (second row) by CV for Examples 2 and 3 at a single realization  with dimension $p=8$ and sample size $n=1000$. }\label{fig_LGWS}
\end{figure}

\begin{figure}[H]
	\begin{subfigure}{.45\textwidth}
		\centering
		\includegraphics[width=1\linewidth]{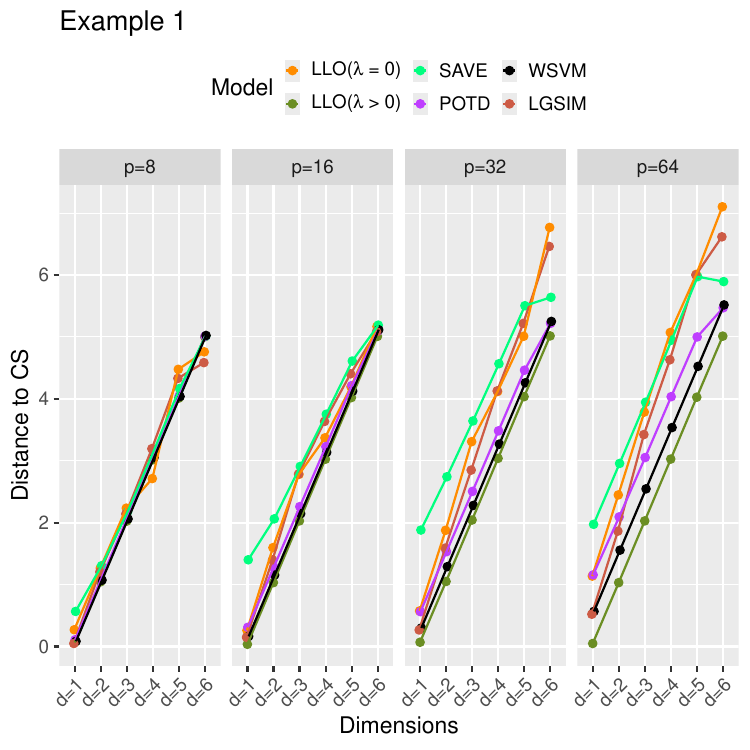}  
	\end{subfigure}
	\begin{subfigure}{.45\textwidth}
		\centering
		\includegraphics[width=1\linewidth]{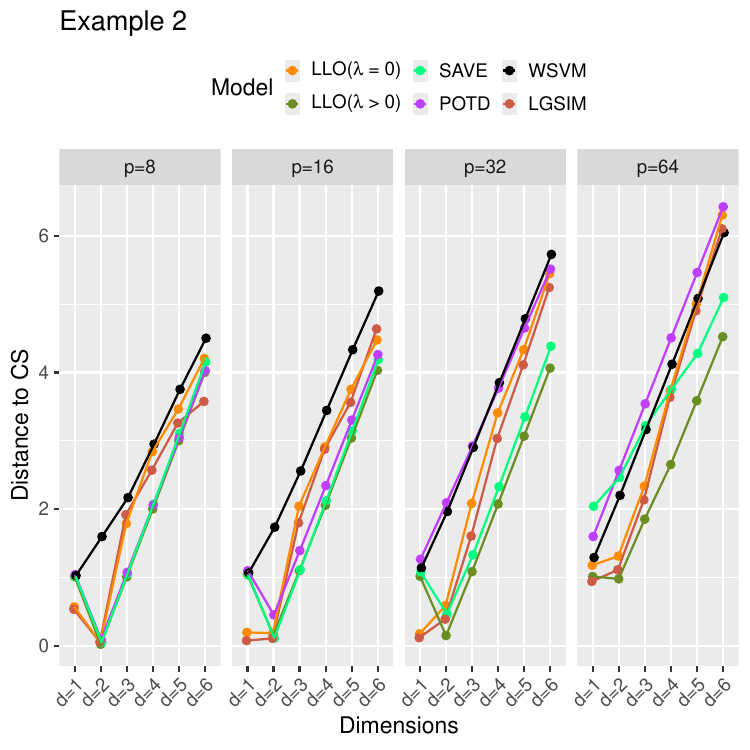}  
	\end{subfigure}
	\newline
	\begin{subfigure}{.45\textwidth}
		\centering
		\includegraphics[width=1\linewidth]{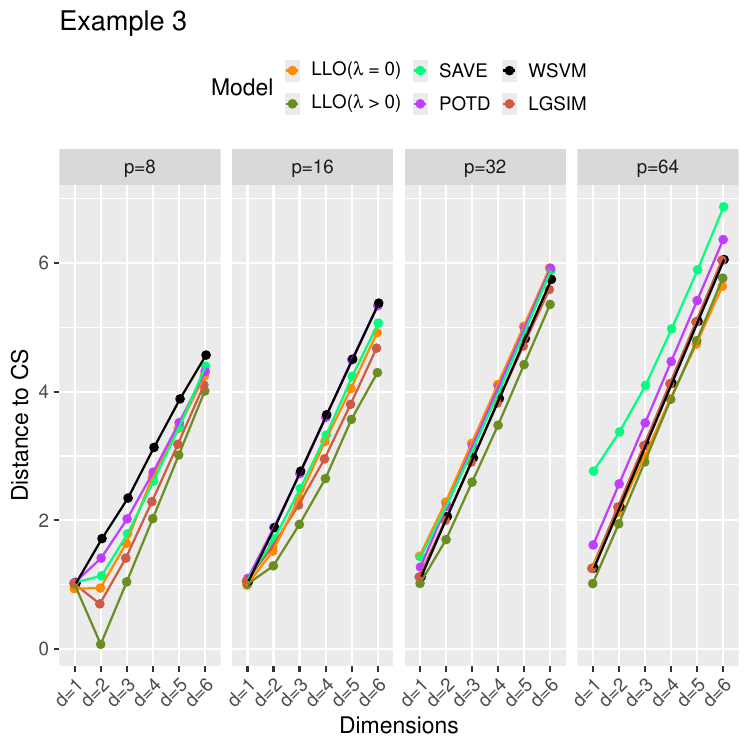}  
	\end{subfigure}
	\caption{Simulation study -- Distance to the central subspace and misclassification risk in Examples 1, 2 and 3, averaged over $N=1000$ replications of a sample of size $n=1000$, as a function of the dimension $d\in \{1,\ldots,6\}$ of the estimated central subspace and $p\in \{8,16,32,64\}$ of the full covariate space.}
	\label{fig:distances_pvarying}
\end{figure}

\begin{figure}[H]
	\begin{subfigure}{.45\textwidth}
		\centering
		\includegraphics[width=1\linewidth]{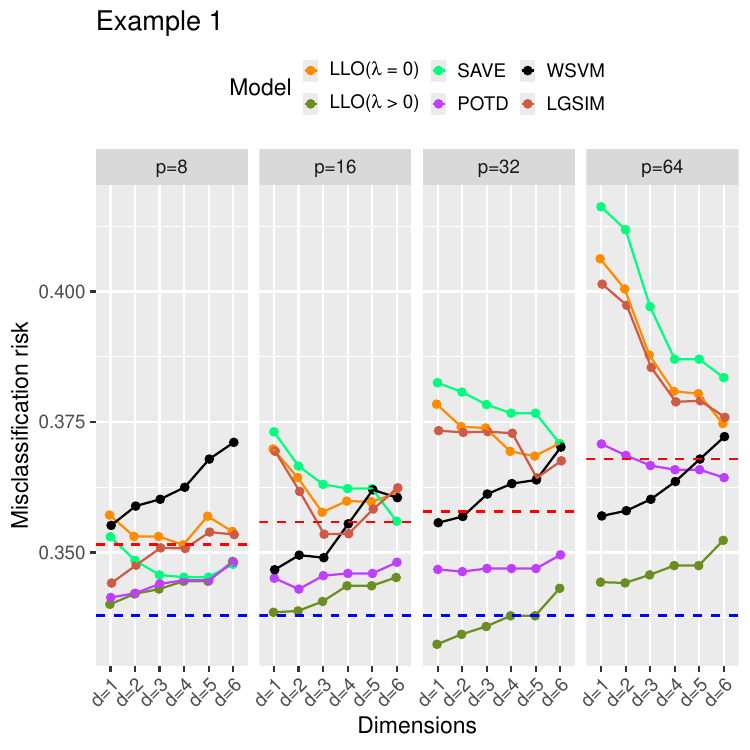}  
	\end{subfigure}
	\begin{subfigure}{.45\textwidth}
		\centering
		\includegraphics[width=1\linewidth]{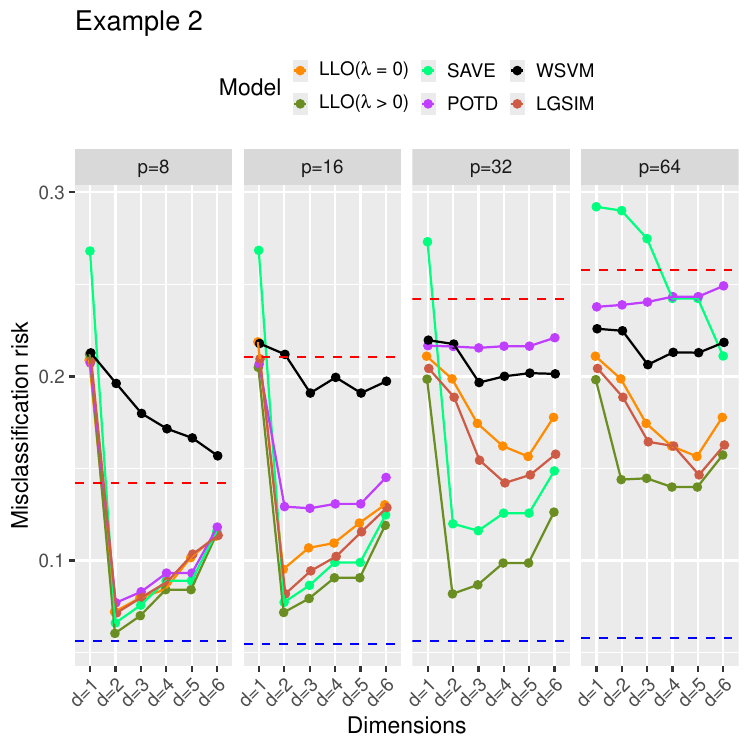}  
	\end{subfigure}
	\newline
	\begin{subfigure}{.45\textwidth}
		\centering
		\includegraphics[width=1\linewidth]{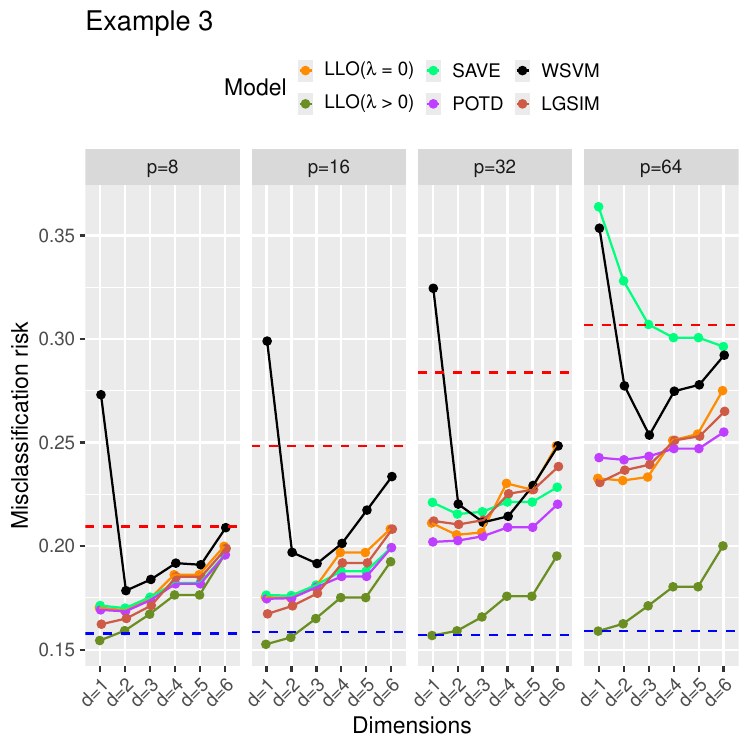}  
	\end{subfigure}
	\caption{{Simulation study -- Misclassification risk in Examples 1, 2 and 3, averaged over $N=1000$ replications of a sample of size $n=1000$, as a function of the dimension $d\in \{1,\ldots,6\}$ of the estimated central subspace and $p\in \{8,16,32,64\}$ of the full covariate space. In the right-hand panels, the red dashed line corresponds to the nearest-neighbor classifier with $d=p$, and the blue dashed line corresponds to this classifier using the covariates projected on the correct population central subspace.}}
	\label{fig:mcrisks}
\end{figure}

\begin{figure}[H]
 \begin{subfigure}{.45\textwidth}
		\centering
		\includegraphics[width=1\linewidth]{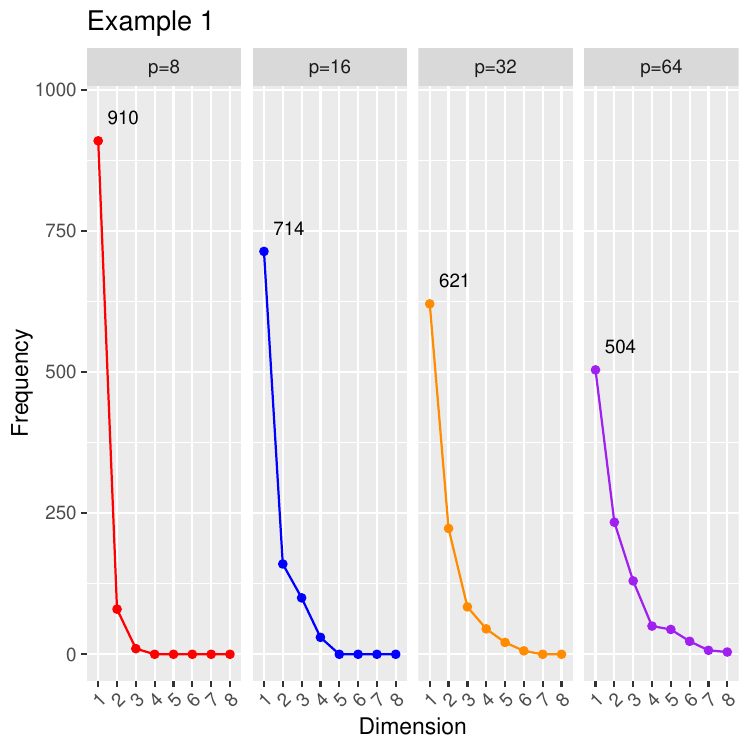}  
\end{subfigure}
 \begin{subfigure}{.45\textwidth}
		\centering
		\includegraphics[width=1\linewidth]{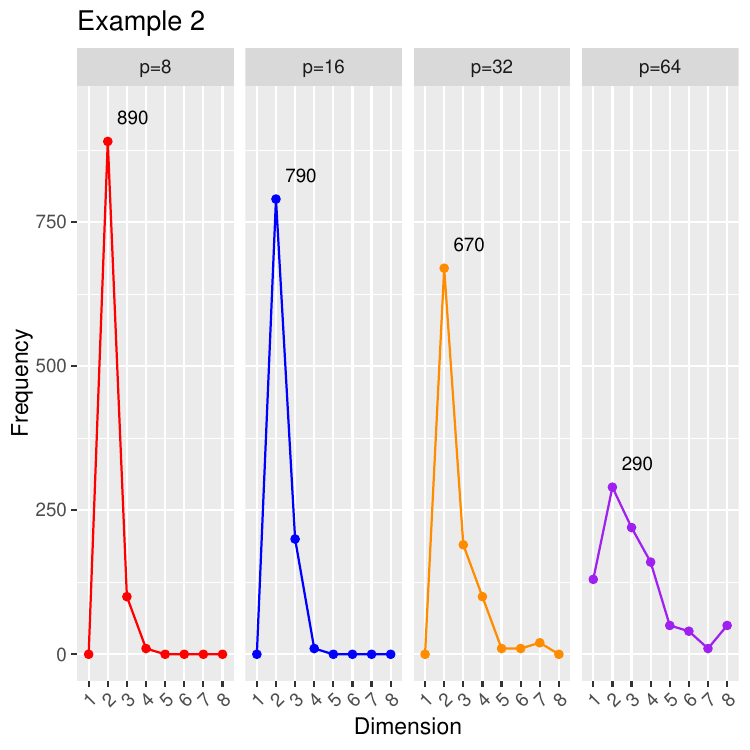}  
\end{subfigure}
\newline
\begin{subfigure}{.45\textwidth}
		\centering
		\includegraphics[width=1\linewidth]{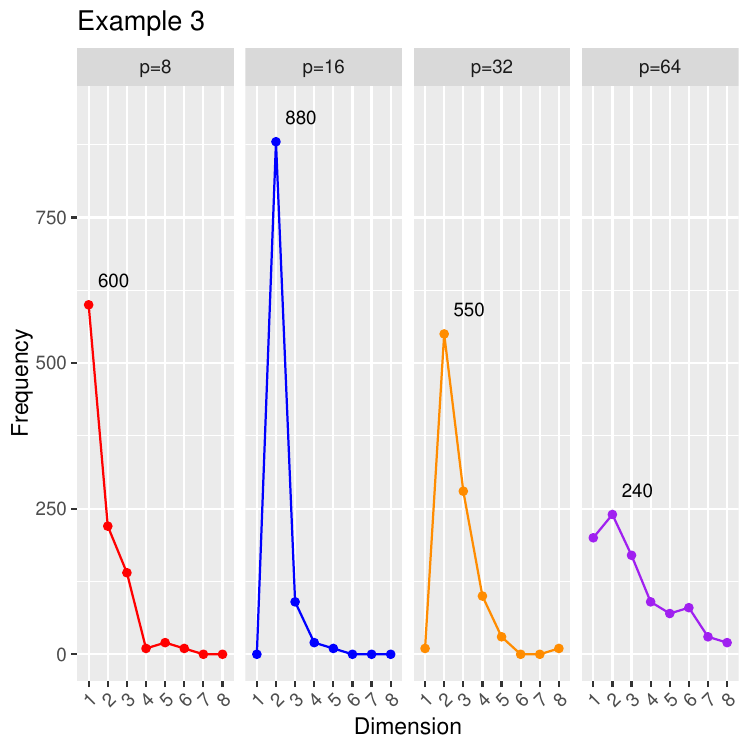}  
\end{subfigure}
	\caption{{Simulation study -- Dimension selection through Algorithm~\ref{alg:feature} over $N=1000$ independent replications of a sample of size $n=1000$, as a function of the dimension $p\in \{8,16,32,64\}$ of the full covariate space. In each panel, the number indicated above the curve gives the number of times the dimension selected in the (absolute or relative) majority of cases was chosen.}}
	\label{fig:dimension_select}
\end{figure}
\newpage

\begin{figure}[H]
\begin{subfigure}{.45\textwidth}
		\centering
		\includegraphics[width=1\linewidth]{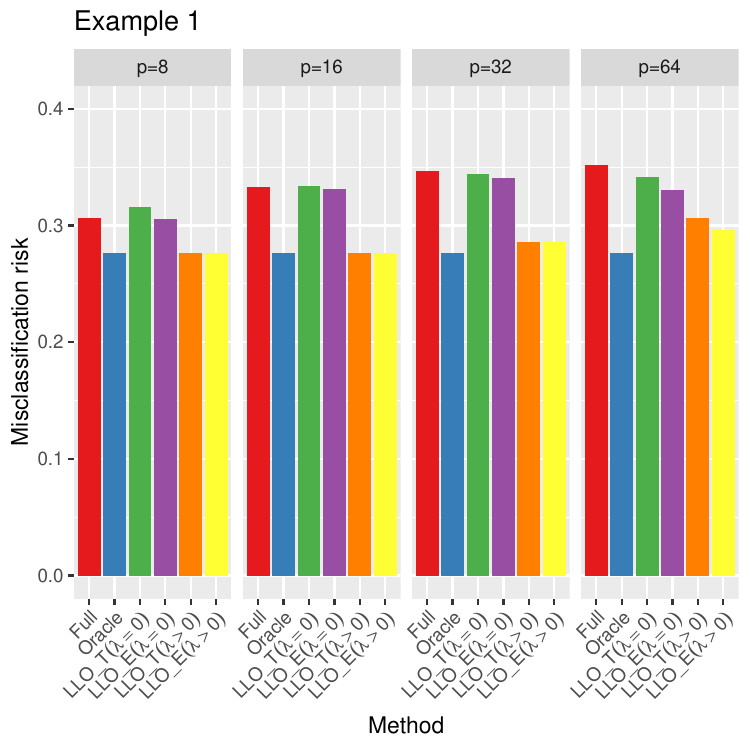}  
\end{subfigure}
\begin{subfigure}{.45\textwidth}
		\centering
		\includegraphics[width=1\linewidth]{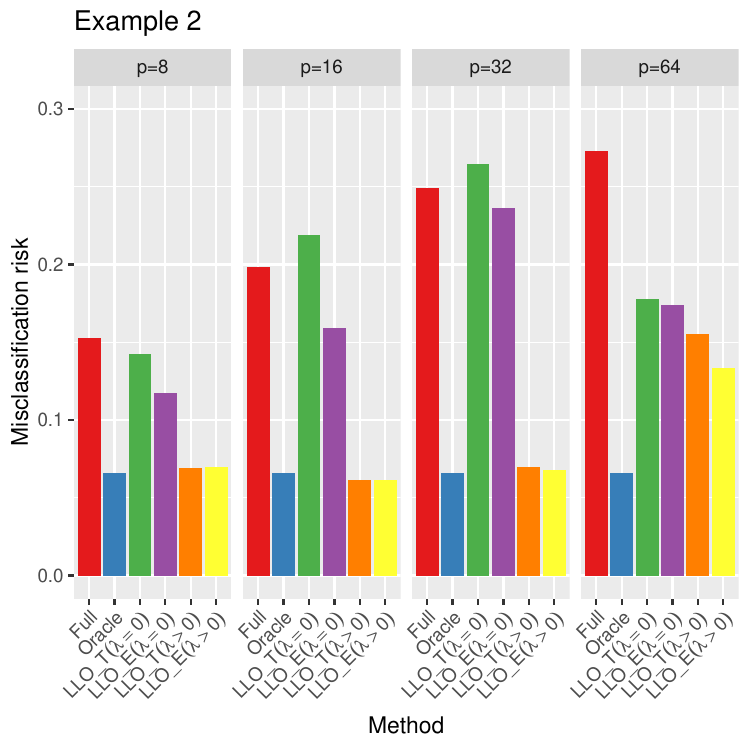}  
\end{subfigure}
\newline
\begin{subfigure}{.45\textwidth}
		\centering
		\includegraphics[width=1\linewidth]{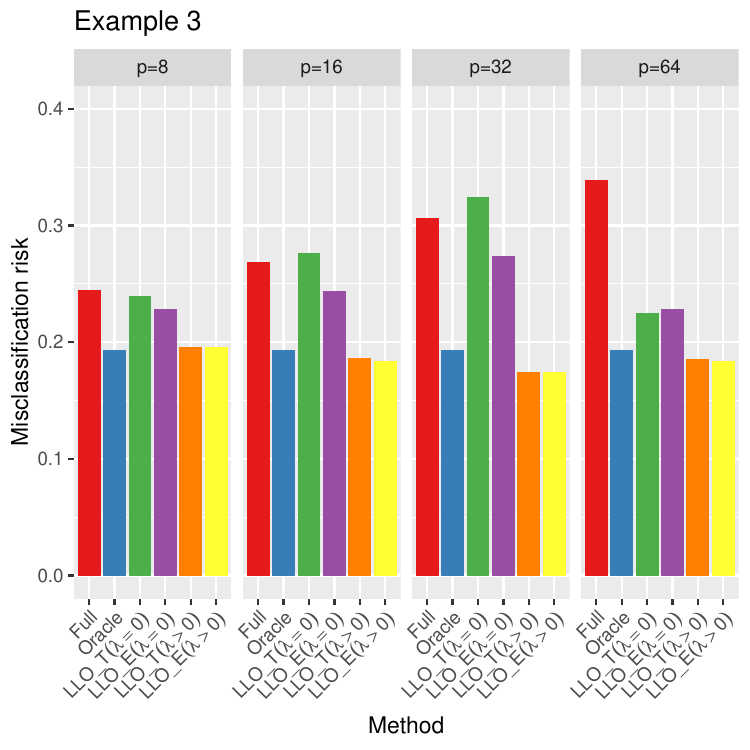}  
\end{subfigure}
	\caption{{Simulation study -- Misclassification risk of the nearest-neighbor classifier with, from left to right, $d=p$ (red bar), the covariates projected on the correct population central subspace (blue bar), the central subspace estimated using the non-penalized LLO($\lambda=0$) method under correct specification of the dimension (green bar) and with the dimension estimated by cross-validation (purple bar), and the central subspace estimated using the penalized LLO($\lambda>0$) method under correct specification of the dimension (orange bar) and with the dimension estimated by cross-validation (yellow bar). All panels are produced using $N=1000$ independent replications of a sample of size $n=1000$ and considering dimensions $p\in \{8,16,32,64\}$ of the full covariate space.}}
	\label{fig:dim_comparison}
\end{figure}
\newpage
\begin{figure}[H]
	\begin{subfigure}{.45\textwidth}
		\centering
		\includegraphics[width=1\linewidth]{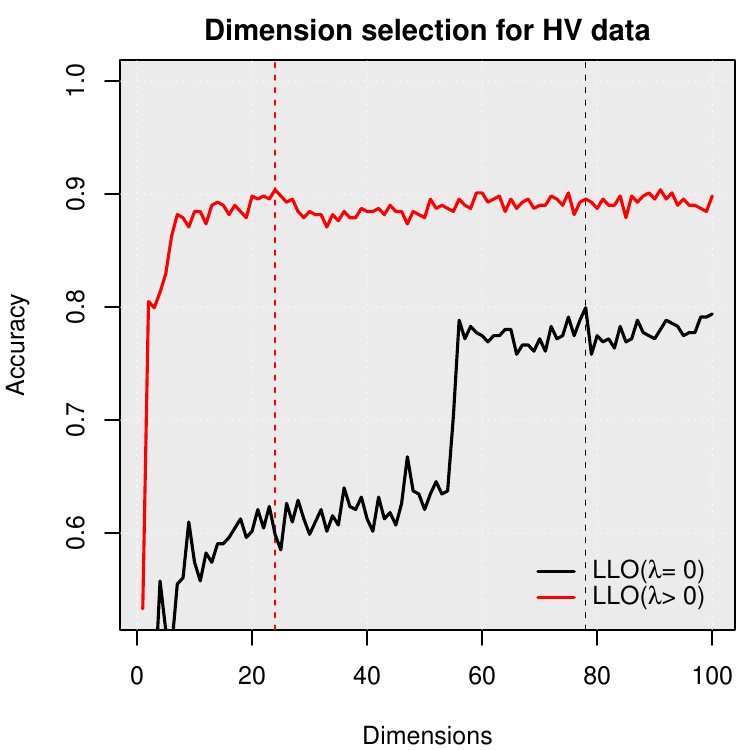}  
		\caption{HV}
	\end{subfigure}
     \begin{subfigure}{.45\textwidth}
		\centering
		\includegraphics[width=1\linewidth]{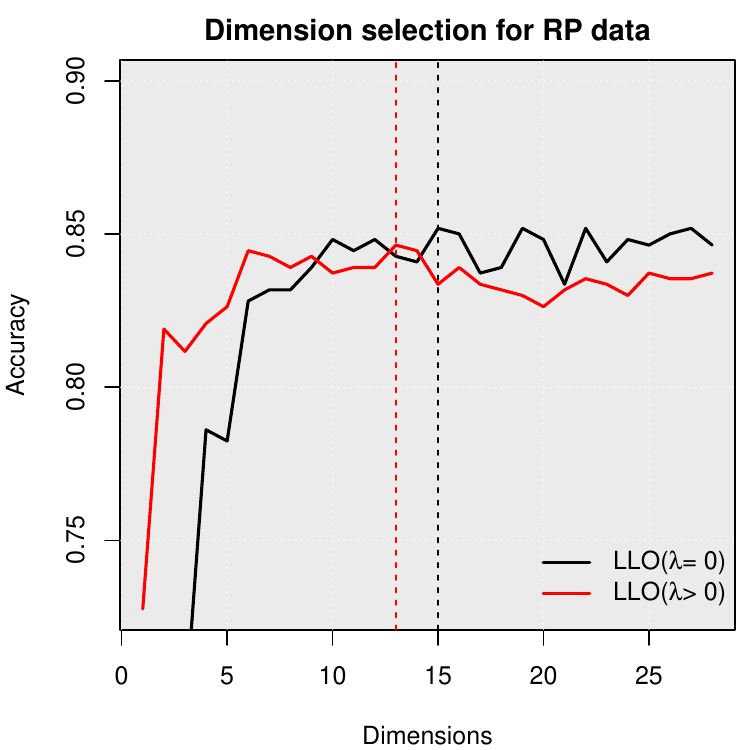}  
		\caption{RP}
	\end{subfigure}
	\begin{subfigure}{.45\textwidth}
		\centering
		\includegraphics[width=1\linewidth]{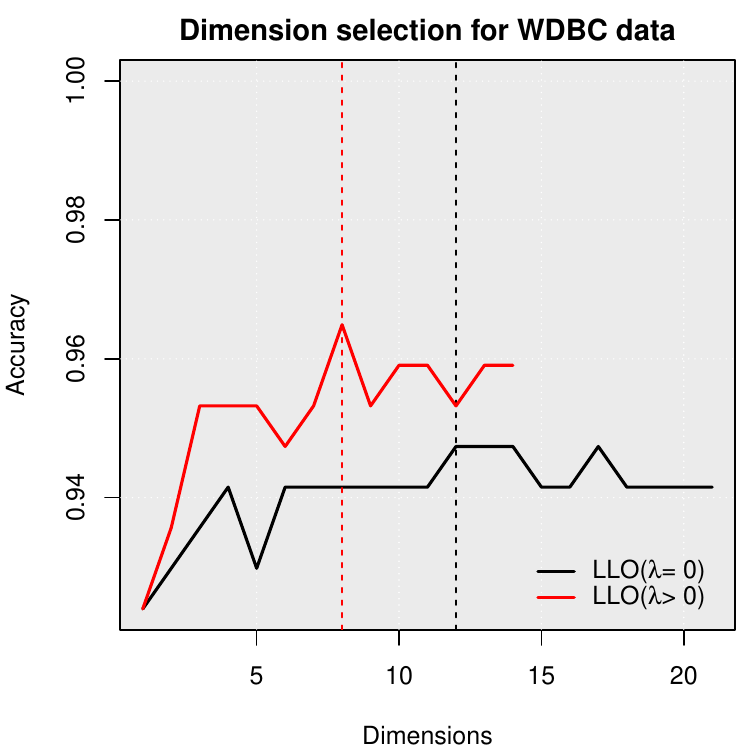}  
		\caption{WDBC}
	\end{subfigure}
	\caption{{Real data analysis -- Dimension selection through cross-validation for LLO($\lambda=0$) and LLO($\lambda>0$), using the random forest classifier. In the WBDC real data analysis, all the eigenvalues of the empirical outer product $\widehat{M}$ were found to be 0 from dimension $d=15$ and $d=22$ onwards when using the LLO($\lambda>0$) and LLO($\lambda=0$), respectively.}}
	\label{fig:Dim_selection_rf}
\end{figure}

\begin{figure}[H]
	\centering
	\begin{subfigure}{.3\textwidth}
		\centering
		\includegraphics[width=1\linewidth]{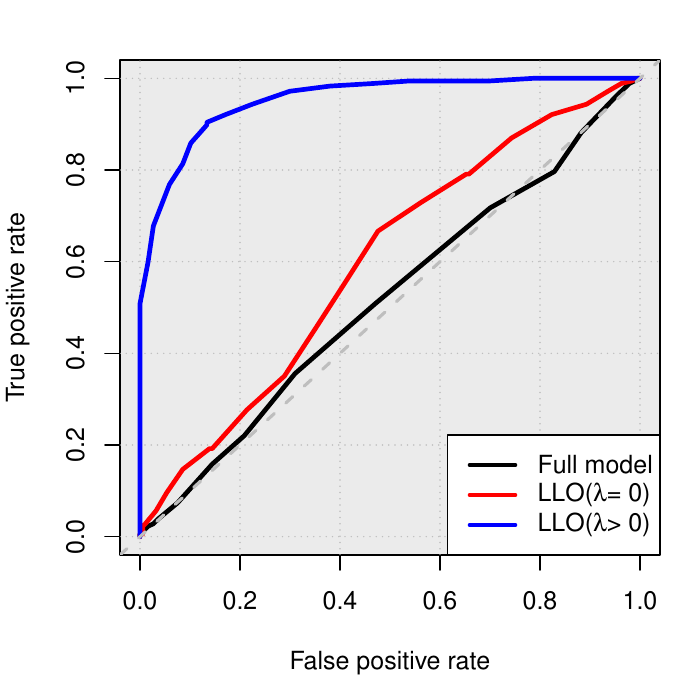}  
		\caption{HV}
	\end{subfigure} \hfill
    \begin{subfigure}{.3\textwidth}
		\centering
		\includegraphics[width=1\linewidth]{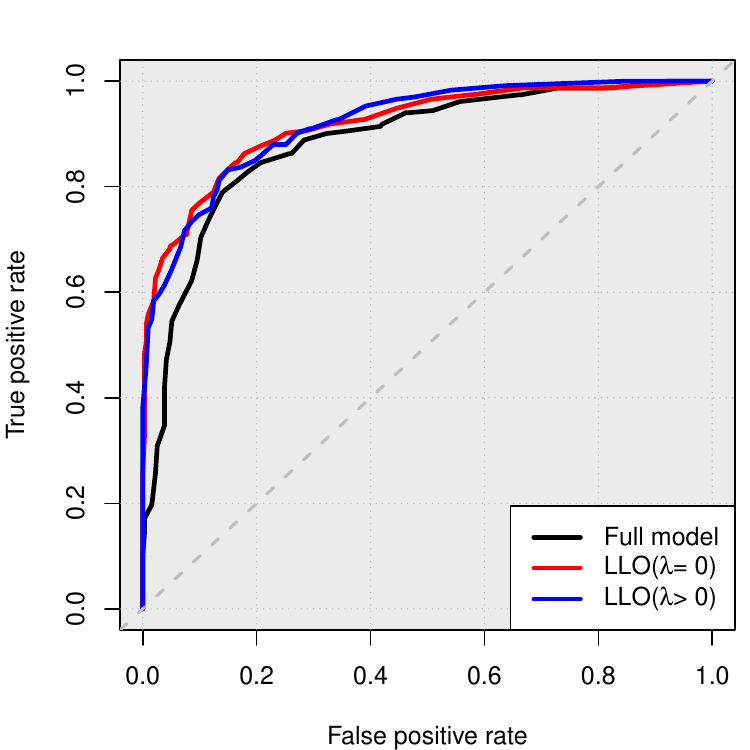}  
		\caption{RP}
	\end{subfigure} \hfill
	\begin{subfigure}{.3\textwidth}
		\centering
		\includegraphics[width=1\linewidth]{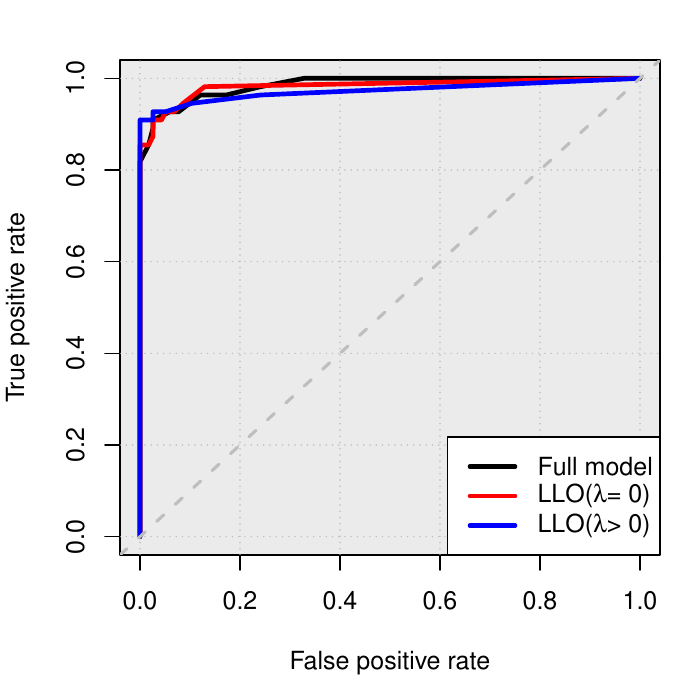}  
		\caption{WDBC}
	\end{subfigure} \hfill
	\begin{subfigure}{.3\textwidth}
		\centering
		\includegraphics[width=1\linewidth]{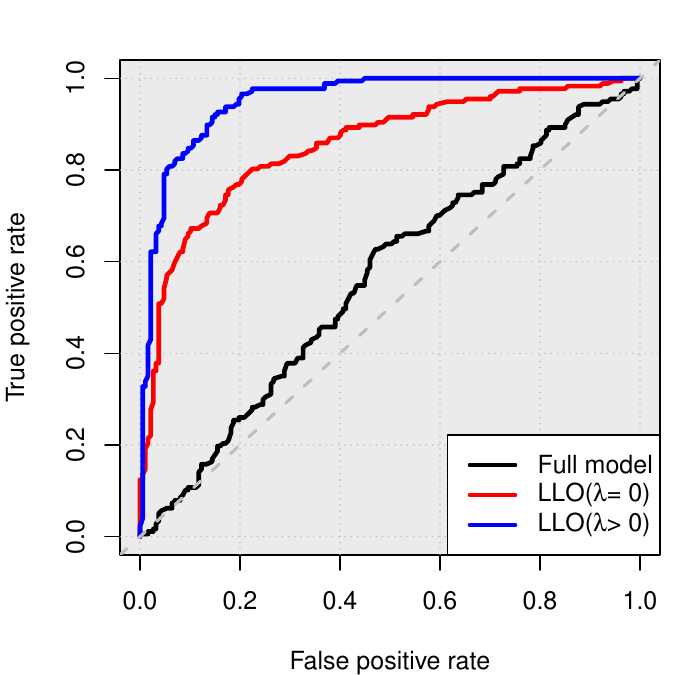}  
		\caption{HV}
	\end{subfigure} \hfill
    \begin{subfigure}{.3\textwidth}
		\centering
		\includegraphics[width=1\linewidth]{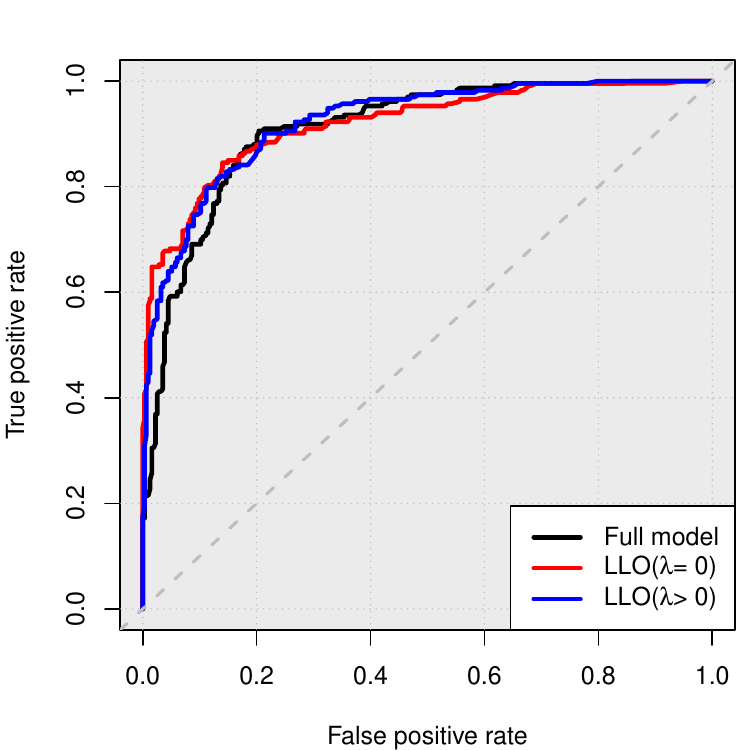}  
		\caption{RP}
	\end{subfigure} \hfill
	\begin{subfigure}{.3\textwidth}
		\centering
		\includegraphics[width=1\linewidth]{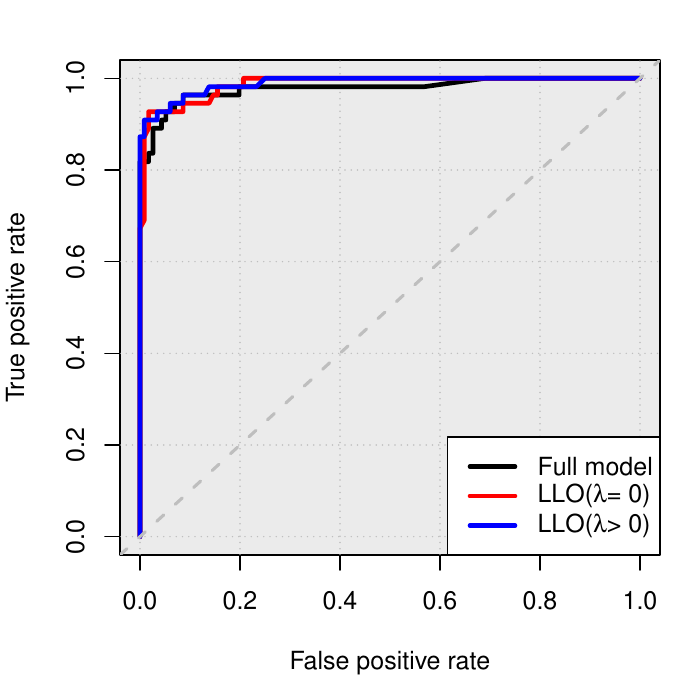}  
		\caption{WDBC}
	\end{subfigure}
	\caption{{Real data analysis -- ROC curve of the nearest-neighbor classifier (top) and random forest classifier (bottom) with no dimension reduction (black curve), dimension reduction following the LLO$(\lambda=0)$ procedure (red curve) and dimension reduction following the LLO$(\lambda>0)$ procedure (blue curve). In each case, the prediction exercise is carried out on the selected testing set.}}
	\label{fig:realapp_ROC}
\end{figure}

\begin{figure}[H]
		\centering
		\includegraphics[width=1\linewidth]{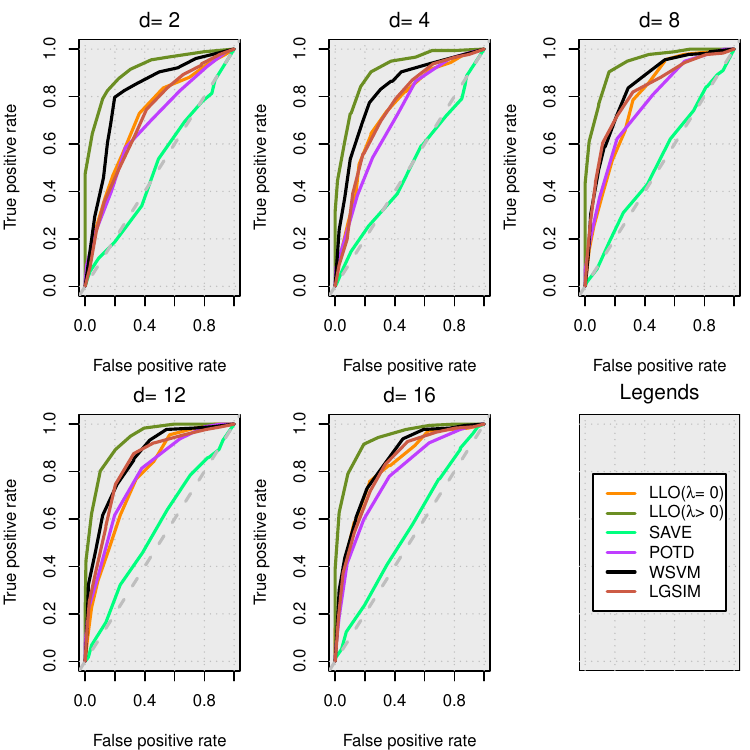}  
		\caption{Real data analysis, HV dataset -- ROC curve of the nearest-neighbor classifier following the LLO$(\lambda=0)$, LLO$(\lambda>0)$, SAVE, POTD, WSVM and LGSIM dimension reduction procedures, for a dimension $d$ of the dimension reduction subspace in $\{2,4,8,12,16\}$. In each case, the prediction exercise is carried out on the selected testing set.}
		\label{fig:comparison_data_roc_HV}
\end{figure}


\begin{figure}[H]
		\centering
		\includegraphics[width=1\linewidth]{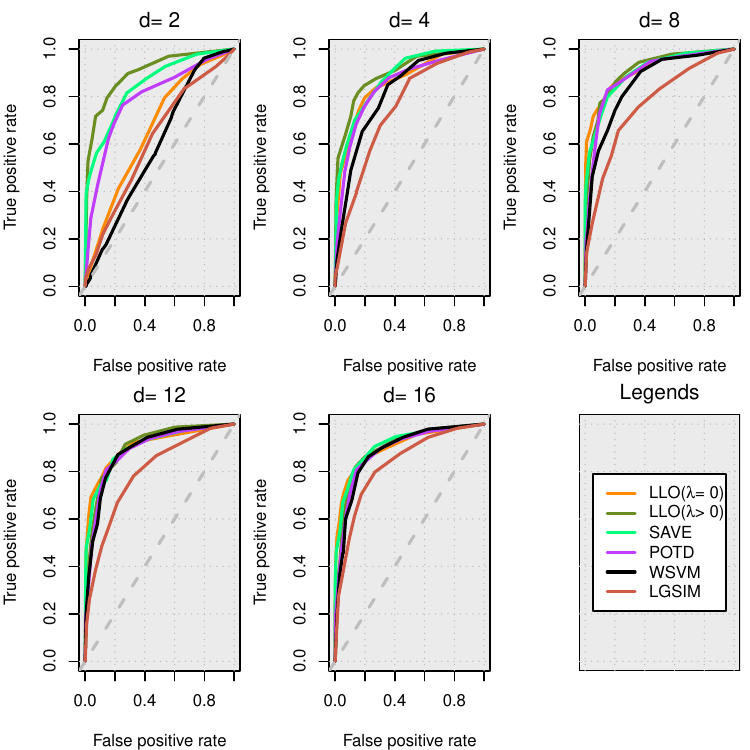}  
		\caption{Real data analysis, RP dataset -- ROC curve of the nearest-neighbor classifier following the LLO$(\lambda=0)$, LLO$(\lambda>0)$, SAVE, POTD, WSVM and LGSIM dimension reduction procedures, for a dimension $d$ of the dimension reduction subspace in $\{2,4,8,12,16\}$. In each case, the prediction exercise is carried out on the selected testing set.}
		\label{fig:comparison_data_roc_MPE}
\end{figure}

\begin{figure}[H]
		\centering
		\includegraphics[width=1\linewidth]{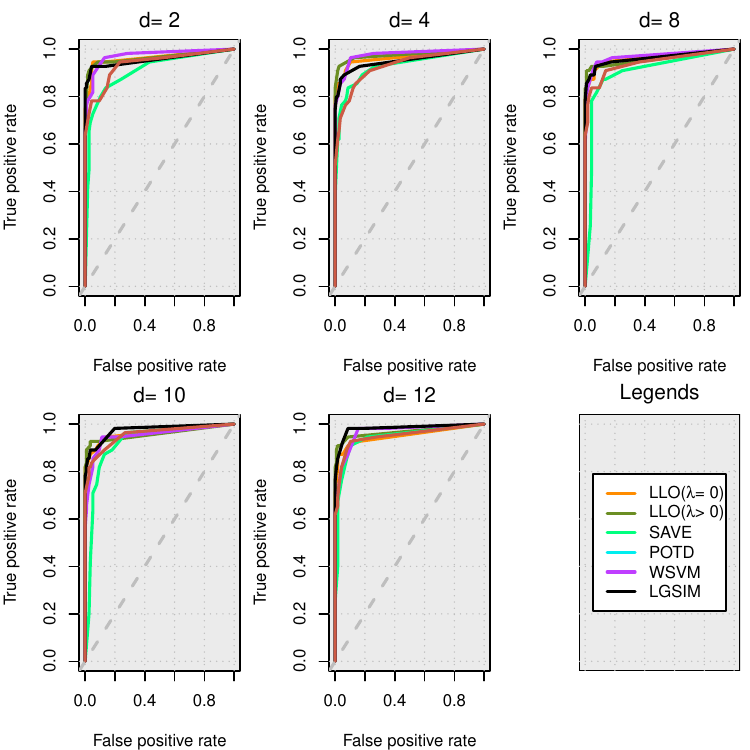}  
		\caption{Real data analysis, WDBC dataset -- ROC curve of the nearest-neighbor classifier following the LLO$(\lambda=0)$, LLO$(\lambda>0)$, SAVE, POTD, WSVM and LGSIM dimension reduction procedures, for a dimension $d$ of the dimension reduction subspace in {$\{2,4,8,10,12\}$}. In each case, the prediction exercise is carried out on the selected testing set.}
		\label{fig:comparison_data_roc_WDBC}
\end{figure}

\end{document}